\documentclass[12pt,a4paper]{article}
\usepackage[utf8]{inputenc}
\usepackage{amsmath}
\usepackage{amsfonts}
\usepackage{dsfont}
\usepackage{mathrsfs}
\usepackage{fullpage,amsmath,amsfonts,amsthm,hyperref,paralist,graphicx,color,float, mathtools,tikz,caption,subcaption}
\usepackage{amssymb}
\usepackage{here}
\usepackage[normalem]{ulem}
\usepackage{centernot}
\usepackage{ifthen,xkeyval, tikz, calc}
\usepackage{xcolor}
\usepackage{amssymb}
\usepackage{ragged2e}
\usepackage[margin=0.85 in]{geometry}
\usepackage{enumitem}
\usepackage{cleveref}
\usepackage{graphicx}
\usepackage{caption}

\usepackage[textsize=scriptsize, backgroundcolor=blue!20!white,bordercolor=red]{todonotes}
\definecolor{darkorange}{RGB}{255,165,0}
\definecolor{altviolet}{RGB}{139,0,139}
\definecolor{turquoise}{RGB}{64,224,208}

\definecolor{nicegreen}{rgb}{0.0, 0.5, 0.0}
\definecolor{niceblue}{rgb}{0.36, 0.54, 0.66}
\usepackage[normalem]{ulem}

\newtheorem{theorem}{Theorem}[section]

\newtheorem{assumption}[theorem]{Assumption}

\newtheorem{corollary}[theorem]{Corollary}

\newtheorem{definition}[theorem]{Definition}

\newtheorem{fact}[theorem]{Fact}

\newtheorem{lemma}[theorem]{Lemma}

\newtheorem{notation}[theorem]{Notation}
\newtheorem{proposition}[theorem]{Proposition}

\newtheorem{remark}[theorem]{Remark}
\newtheorem*{ass*}{Assumption}
\newtheorem*{model*}{Model}
\newtheorem*{coupling*}{Coupling}

\newtheorem{generalassumption}{Assumption}

\newtheorem{weightedassumption}{Assumption}

\newtheorem{nonweightedassumption}{Assumption}

\newcounter{weightedcondition}
\newcommand{\weighteditem}[2]{%
  \item[\textup{(#1)}]%
  \renewcommand{\theweightedcondition}{#1}%
  \refstepcounter{weightedcondition}%
  \label{#2}%
}

\newcommand{\E}{\mathbb E}
\newcommand{\R}{\mathbb R}
\newcommand{\Rd}{\mathbb R^d}

\newcommand{\X}{\mathbb X}
\newcommand{\N}{\mathbb N}

\newcommand{\dd}{\mathrm{d}} 

\newcommand{\be}{\begin{equation}}
\newcommand{\ee}{\end{equation}}

\newcommand{\C}{\mathscr {C}}

\newcommand{\conn}[3]{#1 \longleftrightarrow #2\textrm { in } #3}

\newcommand{\mulfd}{P_\lambda^{L,n}}

\newcommand{\e}{\text{e}}

\DeclarePairedDelimiter\abs{\lvert}{\rvert}

\DeclarePairedDelimiterX{\inner}[2]{\langle}{\rangle}{#1, #2}
\newcommand{\pla}{\mathbb P_{\lambda}}

\newcommand{\rom}[1]{\uppercase\expandafter{\romannumeral #1\relax}}

\numberwithin{equation}{section}

\title{Sharpness of the percolation phase transition for weighted random connection models}
\author{
Alejandro Caicedo\footnote{Ludwig-Maximilians-Universität München, Mathematisches Institut, Theresienstraße 39, 80333 München, Germany;
    Email: caicedo@math.lmu.de
    \includegraphics[height=1em]{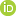}~\url{https://orcid.org/0009-0009-4299-6356}} \and 
Leonid Kolesnikov\footnote{Technische Universität Braunschweig, Institut für Mathematische Stochastik, Universitätspl. 2, 38106 Braunschweig, Germany; Email: leonid.kolesnikov@tu-bs.de
\includegraphics[height=1em]{ORCIDiD_icon.png}~\url{https://orcid.org/0000-0002-3826-0479}}}
\date{\today}

\begin{document}

\maketitle

\begin{abstract}
We establish the sharpness of the percolation phase transition for a class of potentially infinite-range weighted random connection models. The vertex set is given by a marked Poisson point process on $\R^d$, $d\geq 2$, with intensity $\lambda>0$, where each vertex carries an independent weight. Given a pair of vertices, an edge is drawn independently with a probability that depends on both their spatial displacement and their respective weights. Our geometric and moment assumptions allow the expected degree
to be unbounded over the weights. We prove exponential tails
for the number of vertices in the root cluster throughout the
subcritical regime. We also obtain a linear lower bound on the
percolation probability near criticality from above. The class of systems considered includes min-reach models, Poisson--Boolean overlap
models, and suitable softened versions thereof. Our proof combines a finite-volume lattice approximation with
Hutchcroft's exploration approach to the OSSS inequality.
\end{abstract}

\newpage

\section{Introduction}

\subsection{Motivation}
The Gilbert graph, introduced by Gilbert in \cite{Gil61}, is generally considered the first mathematical model of continuum percolation. In this model, points are distributed according to a homogeneous Poisson point process on $\mathbb{R}^d$, and any two points are connected by an edge if their Euclidean distance is less than a fixed threshold. Although Gilbert's original motivation lay in modeling communication networks, the model introduced key probabilistic and geometric ideas that are now central to the study of continuum percolation. The \emph{random connection model} (RCM), introduced by Penrose in \cite{Pen91}, extends this framework by replacing deterministic connections with probabilistic ones: any two points are connected independently with a probability determined by an \emph{adjacency function} (also called a connection function), typically depending on their spatial distance and decaying with it. This model sits conceptually between Poisson-Boolean percolation, where randomness is encoded in the geometry (e.g., in the radii of balls), and Bernoulli percolation, where edges of a lattice graph are retained independently. 

Penrose showed that the RCM exhibits a non-trivial phase transition for $d\geq 2$: under suitable integrability assumptions, there exists a critical intensity $\lambda_T \in (0,\infty)$ such that for $\lambda < \lambda_T$, the expected size of the cluster containing a typical point is finite, while for $\lambda > \lambda_T$, this expectation diverges. He also considered a second critical threshold $\lambda_c$, defined as the infimum value of $\lambda$ for which a typical vertex belongs to an infinite cluster with positive probability, and showed that $\lambda_T \le \lambda_c < \infty$. Meester \cite{Mee95} extended this analysis by adapting the differential inequality method of Aizenman and Barsky \cite{AB87} from Bernoulli percolation, and proved that $\lambda_c = \lambda_T$ for integrable, isotropic, non-increasing adjacency functions, thus establishing the sharpness of the phase transition.

Further progress came with the work of Heydenreich, van der Hofstad, Last, and Matzke \cite{HHLM22}, who, in sufficiently high dimensions, established convergence of the lace expansion for the RCM, derived the triangle condition, proved an infrared bound, and computed the critical exponent $\gamma = 1$, confirming mean-field behavior near criticality. These results were extended to the \emph{marked random connection model} by Dickson and Heydenreich \cite{DH22}, where each vertex is assigned an independent mark influencing its connection behavior. This setting allowed Caicedo and Dickson \cite{CD24} to prove sharpness and, under an additional triangle condition, compute further mean-field critical exponents by extending the framework of \cite{AB91}. However, these sharpness results require a uniform bound, over pairs of marks, on the spatial integral of the adjacency function. For Poisson--Boolean overlap models, this condition excludes radius distributions with unbounded support. More recently, Chebunin and Last~\cite{CL25} established sharpness and exponential cluster-size tails under weaker uniform moment assumptions on the spatially integrated adjacency function.

This is where our contribution lies---we establish sharpness for a class of potentially infinite-range weighted RCMs allowing unbounded weights and, more substantially, an expected number of neighbors that need not
be uniformly bounded over the weights. The replacement geometry and moment assumptions allow us to
control the contribution of vertices with large weights without imposing such a uniform bound. Our argument also recovers sharpness for the classical RCM under isotropy, monotonicity, and integrability of the adjacency function. For this non-weighted setting, exponential cluster-size tails were also obtained independently by Higgs~\cite{Hig25}. 

Our proof draws on both classical and modern techniques. On the classical side, we employ a finite-volume, discrete approximation of the continuum model, adapting the approach of Meester and Roy from the 1990s~\cite{Mee95,MeeRoy96}. This allows us to control the geometry and combinatorics of the system, and to reduce the analysis to a discrete setup where Boolean function tools can be applied. On the modern side, we build on recent developments in the analysis of Boolean functions, particularly based on the OSSS inequality \cite{OSSS05}, which was turned into a powerful tool in percolation theory by Duminil-Copin, Raoufi, and Tassion. Among several other settings~\cite{DRT19b,DRT19a}, they applied it in \cite{DRT20} to establish sharpness of the phase transition for Poisson-Boolean percolation with unbounded radii under suitable moment assumptions. At the same time, Hutchcroft \cite{Hut20} developed a related algorithmic approach to analyze the random-cluster model on graphs, providing a discrete analogue that generalizes well to long-range settings. Hutchcroft’s algorithm, while developed in a discrete setting, extends naturally to our model and provides the backbone of our proof. The main additional difficulty is to control the contribution of vertices with large weights. Using the replacement geometry, we compare magnetization and pivotality across weights, with estimates uniform in the discretization scale. The interplay between the classical finite-volume approximation and these modern techniques is essential to deriving a differential inequality that captures the weighted, long-range structure of the model. We combine these two strands of work, along with original arguments tailored to our setting, which ultimately allows us to prove exponential decay of cluster sizes in the subcritical regime, and a linear lower bound on the percolation probability above criticality. A more nuanced discussion of our work in the context of recent developments will be given in Subsection~\ref{subsect:discuss} --- after stating our main result Theorem~\ref{thm:Susceptibility Mean-Field Bound}. 

\subsection{Framework}

\paragraph{The weighted random connection model.} 

We now give an informal description of the model; the reader is referred to \cite[Section 3]{DH22} for the formal construction.  Let $d\geq 2$. Our vertex set is given by a Poisson point process $\eta$ in $\X = \Rd \times [1,\infty)$  with intensity measure $\lambda \nu$, where $\lambda > 0$ and $\nu = \text{Leb} \otimes \pi$; $\text{Leb}$ refers to the standard Lebesgue measure on $\R^d$ and $\pi$ is a probability measure on $[1,\infty)$ --- our vertices have a spatial coordinate and a weight. The probability of an edge to be drawn will be determined by a measurable adjacency (or connection) function $\varphi : \X^2 \to [0,1]$: Two vertices $x,y \in\eta$ will be connected by an edge with probability $\varphi(x,y)$, with distinct edges being sampled independently. 

We then denote a realization of the full random graph by $\xi$, which means that $\xi$ is a realization of the Poisson point process $\eta$ together with the realization of the edges sampled via $\varphi$ as described above.
 
Next, we want to impose some general assumptions upon the adjacency function $\varphi$.

\begin{generalassumption}[Isotropy, symmetry, monotonicity]
\label{ass:general}
There exists a measurable function
\[
  \widehat\varphi:
  [0,\infty)\times[1,\infty)^2\longrightarrow[0,1]
\]
such that
\[
  \varphi\bigl((x,a),(y,b)\bigr)
  =\widehat\varphi(|x-y|;a,b),
  \quad x,y\in\mathbb R^d,\quad a,b\ge1.
\]
Moreover,
\begin{itemize}
  \item For all $r\ge0$ and $a,b\ge1$, we have
  $\widehat\varphi(r;a,b)=\widehat\varphi(r;b,a)$.
  
  \item For all $a,b\ge1$,
  $r\mapsto\widehat\varphi(r;a,b)$ is non-increasing.

  \item For all $r\ge0$ and $a,b\ge1$,
  $\hat a\mapsto\widehat\varphi(r;\hat a,b)$ and
  $\hat b\mapsto\widehat\varphi(r;a,\hat b)$
  are non-decreasing.
\end{itemize}
Below, we write $\varphi(r;a,b)$ for $\widehat\varphi(r;a,b)$.
\end{generalassumption}

We now introduce two alternative sets of additional assumptions. First, we consider a very general class of non-weighted RCM:

\begin{nonweightedassumption}[Non-weighted assumptions]
\label{ass:nonweighted}
In addition to Assumption~\ref{ass:general}, suppose that
  \[
    \varphi(r;a,b) = \varphi(r;1,1)=:\phi(r),
    \quad r\ge0,\quad a,b\ge1,
  \]
  and 
  \[
    0<
    \int_{\mathbb R^d}
    \varphi(|x|;1,1)\,\mathrm dx
    <\infty.
  \]
  That is, the expected degree at unit intensity is assumed to be finite
  and strictly positive.
\end{nonweightedassumption}

\begin{remark}
This model coincides with the classical random connection model introduced by Penrose~\cite{Pen91}.  
In particular, we allow $\phi(r)>0$ for all $r>0$, so that the connection range may be unbounded.  
\end{remark}

Next, we consider the following class of genuinely weighted RCM:

\begin{weightedassumption}[Weighted assumptions]
\label{ass:weighted}
In addition to Assumption~\ref{ass:general}, suppose that
$\pi$ is either absolutely continuous or purely atomic,
and that 

\begin{itemize}

  \weighteditem{LP}{ass:LP}
  \textbf{Local positivity.}
  There exists $\rho>0$ such that
  \begin{equation}
  \label{eq:local_positivity}
    \varphi(\rho;1,1)>0.
  \end{equation}
  In particular, monotonicity gives
  $\varphi(r;a,b)\ge p_0:=\varphi(\rho;1,1)>0$
  for $0\le r\le\rho$ and $a,b\ge1$.

  \weighteditem{BR}{ass:BR}
  \textbf{Bounded range at minimal marks.}
  There exists $R_0>0$ such that
  \begin{equation}
  \label{eq:baseline_bounded_range}
    \varphi(R_0;1,1)=0.
  \end{equation}
  By monotonicity, this also gives $\varphi(r;1,1)=0$ for all $r\ge R_0$.

  \weighteditem{RG}{ass:RG}
  \textbf{Replacement geometry.}
  There exist a non-decreasing function $\ell:[1,\infty)\to[0,\infty)$ and $\sigma\in(0,\rho)$ such that
  \begin{align}
  \label{eq:replacement_geometry}
    \varphi(r;a,b)>0
    \quad\Longrightarrow\quad
    \varphi(r-\ell(a)+\sigma;1,b)=1,
    \quad a,b\ge1,\quad r>\ell(a).
  \end{align}
  This allows a possible connection to a distant neighbor to be replaced by a deterministic connection from a suitably placed vertex of minimal weight.

  \weighteditem{MR}{ass:MR}
  \textbf{Moments of the replacement scale.}
  The function $\ell$ in~\eqref{ass:RG} satisfies
  \begin{align}
  \label{eq:replacement_moments}
    \int_{[1,\infty)}
    \exp \bigl\{ t \ell(a)^d \bigr\} \pi(\dd a)
    <\infty,
    \quad t>0.
   \end{align} 
\end{itemize}
\end{weightedassumption}

\begin{remark}[Range, integrability, and neighborhood bounds]
\label{rem:consequences_assumptions}
We record some useful consequences of the assumptions above.

\begin{itemize}
\item \noindent\textbf{Range bounds.}
Under Assumption~\ref{ass:weighted}, applying~\eqref{ass:RG} twice, using symmetry and~\eqref{ass:BR}, gives
\begin{equation}
\label{eq:pairwise_range_bound}
  \varphi(r;a,b)>0
  \quad\Longrightarrow\quad
  r\le \ell(a)+\ell(b)+R_0,
  \quad r\ge0,\quad a,b\ge1.
\end{equation}
In particular, each fixed pair of weights has a finite connection range, and the range is bounded uniformly when both weights are restricted to a bounded interval. There need not be a uniform range bound over all weights, nor can the distance to every neighbor generally be bounded using only the weight of one endpoint.

Under Assumption~\ref{ass:nonweighted}, no finite-range condition is imposed: an integrable connection function may have unbounded support.

\medskip
\item \noindent\textbf{Integrability.}
Write
\[
  K(a,b):=
  \int_{\mathbb R^d}
  \varphi(\vert x \vert;a,b) \dd x.
\]
Let $v_d$ denote the Lebesgue volume of the Euclidean unit ball in $\R^d$. By~\eqref{eq:pairwise_range_bound},
\[
  K(a,b)\leq
  v_d\bigl(\ell(a)+\ell(b)+R_0\bigr)^d,
  \quad a,b\geq 1.
\]
Condition~\eqref{ass:MR} therefore yields
\begin{align*}
\label{eq:kernel_integrability}
\int_{[1,\infty)}
  \int_{[1,\infty)}
  K(a,b)^2\,\pi(\mathrm da)\pi(\mathrm db)
  <\infty.
\end{align*}
Since $K(a,b)$ is non-decreasing in each weight,
\[
\int_{[1,\infty)}K(1,b)^2\,\pi(\mathrm db)
\leq \int_{[1,\infty)}\int_{[1,\infty)}
  K(a,b)^2\,\pi(\mathrm da)\pi(\mathrm db)
  <\infty.
\]
For these integrability conclusions, it is enough to assume the weaker moment condition
$\int\ell(a)^{2d}\,\pi(\mathrm da)<\infty$.

Under Assumption~\ref{ass:nonweighted}, $K$ is a finite, strictly positive constant, so both integrability conclusions hold as well.

\medskip
\item \noindent\textbf{Neighborhood bounds.}
Write
\[
  D(a):=
  \int_{[1,\infty)}K(a,b)\,\pi(\mathrm db).
\]
Thus, $\lambda D(a)$ is the expected degree of an inserted vertex of weight $a$. Under Assumption~\ref{ass:weighted}, local positivity, monotonicity in weights, and~\eqref{eq:pairwise_range_bound} give
\[
  0<p_0v_d\rho^d
  \le D(a)
  \le C\bigl(1+\ell(a)^d\bigr)<\infty,
  \quad a\ge1,
\]
where $C$ is a finite constant independent of $a$.
Thus, the expected neighborhood size is bounded below uniformly over the weights and is finite for every fixed weight; it need not be bounded above uniformly over all weights.

For the Boolean overlap kernel
\[
  \varphi(r;a,b)=\mathbf 1_{\{r\le a+b\}},
\]
with an unbounded weight law satisfying~\eqref{ass:MR} for $\ell(a)=a$, we have
\[
  D(a)=v_d\int_{[1,\infty)}(a+b)^d \pi(\mathrm db)
  \ge v_da^d,
  \quad a\ge1.
\]
Hence, $D(a)$ is not bounded above uniformly over the weights.

Under Assumption~\ref{ass:nonweighted}, $D(a)$ is a finite, strictly positive constant.
\end{itemize}
\end{remark}

We next describe how to augment the graph by fixed points and introduce the notation for connection events, probabilities, and expectations.

\paragraph{Augmentation by fixed points.} Since our intensity measure $\nu$ is non-atomic, any given point $x\in\X$ is almost surely not a vertex of $\xi$. Therefore, for studying most interesting events, it is necessary to consider augmented versions of $\xi$. Given a vertex set $\eta$ and a point $x \in \X$, we denote by $\eta \cup \{x\}$ the vertex set where we place a vertex at position $x$. Given a graph $\xi$, we obtain $\xi^x$ by augmenting the vertex set $\eta$ to $\eta \cup \{x\}$ as above and letting $\xi^x$ inherit all the edges that were present in $\xi$. The edge set of $\xi^x$ is completed by additionally drawing edges from $x$ to the vertices of $\eta$, independently of everything else, with probabilities prescribed by the same adjacency function $\varphi$. This can be easily generalized to obtain $\eta\cup \{x_1,\dots,x_m\}$ and $\xi^{x_1,\ldots,x_m}$ for any $m\in\N$ and any set of points $x_1,\dots,x_m \in \X $.

\paragraph{Connection events.}
Given $x,y\in\eta$, we say that $x$ and $y$ are \emph{adjacent} in $\xi$, denoted $x\sim y$, if there exists an edge with endpoints $x$ and $y$. We say that $x$ and $y$ are connected in $\xi$, denoted $\conn{x}{y}{\xi}$, if there exists a finite sequence of adjacent vertices in $\xi$ starting at $x$ and ending at $y$. That is, either $x=y$ or there exist $n\geq 2$ and $x_1,\ldots,x_{n}\in\eta$ such that $x\sim x_1$, $y\sim x_{n}$, and $x_i\sim x_{i+1}$ in $\xi$ for all $i\in\left\{1,\ldots,n-1\right\}$. Next, given $x\in\eta$ and a (possibly random) set $B\subset\X$, we say that $\conn{x}{B}{\xi}$ if there exists $y\in \eta\cap B$ such that $\conn{x}{y}{\xi}$. Furthermore, because of the non-atomicity of $\nu$, almost surely any fixed point $x \in \X$ will not be present in $\eta$, therefore (in most cases) the relevant events to consider are of the form $\conn{x}{y}{\xi^{x,y}}$ or $\conn{x}{B}{\xi^{x}}$.

\paragraph{Probability measure and expectation.}
All random objects introduced above are defined on a common probability space
\((\Omega,\mathcal F,\mathbb P_\lambda)\),
where the parameter $\lambda>0$ denotes the intensity of the underlying marked
Poisson point process. The adjacency function $\varphi$ and the weight distribution $\pi$ remain fixed as $\lambda$ varies.
The stationary law $\mathbb P_\lambda$ governs both the random
vertex set~$\eta$ and the independent edge configuration determined by the
adjacency function~$\varphi$. We choose to handle Palm expectations implicitly via augmented configurations $\xi^x$ rather than by explicitly changing measure. Expectation with respect to~$\mathbb P_\lambda$ is denoted by~$\mathbb E_\lambda$.

\subsection{Main Result}

 We start by defining the relevant quantities in order to state our main result, Theorem~\ref{thm:Susceptibility Mean-Field Bound}. First and foremost, we are interested in the set of all points connected by a path to a given vertex of our graph, i.e., in the cluster of that vertex.

\begin{definition}[Cluster of a point]
For a given point $x \in \X$, we define the cluster $\C(x)$ of $x$ in $\xi^x$ by
\begin{align}
    \C(x)=\C(x,\xi^x):= \{ y \in \eta  \mid \conn x y \xi^x \} \cup \{x\}.
\end{align}
\end{definition}

Throughout, we use the inserted root $o=(\bar0,1)$. All cluster quantities below refer to this fixed minimal-weight root, and the root itself is included in its cluster. We now introduce the main objects of interest.

\begin{definition} [Tail function, percolation probability, expected susceptibility]
Fix $o=(\bar0,1)$. For $\lambda\geq 0$, we define 
\begin{itemize}
    
\item the tail function $\theta_\lambda\colon\N\to \left[0,1\right]$ by
\begin{align}
    \theta_\lambda(k) &:= \pla\left(\abs*{\C(o)}\geq k\right),
\end{align}
\item the percolation probability $\theta(\lambda)$ by
\begin{align}
    \theta(\lambda):=\lim_{k\to\infty}\theta_\lambda(k)
=\mathbb{P}_\lambda\!\left(\,|{\C(o)}|=\infty\,\right),
\end{align}
\item the expected susceptibility $\chi(\lambda)$ by
\begin{align}
    \chi(\lambda) &:= \E_\lambda\left[\abs*{\C(o)}\right].
\end{align}
\end{itemize}
\end{definition}

We use these quantities to describe the behavior of the model in different regimes; the transition between those is governed by the following critical intensities:

\begin{definition}[Critical intensities]
We define the critical intensities
\begin{align*}
    \lambda_T &:= \inf\{\lambda \geq 0\colon \chi(\lambda) =\infty\},\\ \lambda_c &:= \sup\left\{\lambda\geq 0\colon \theta(\lambda)=0 \right\}.
\end{align*}
\end{definition}
A first result states that a non-trivial phase transition in terms of the expected cluster size occurs in the model:
\begin{lemma}[Non-triviality of the phase transition]\label{lem:ex_phase_transition}
Under the assumptions~\ref{ass:nonweighted} or~\ref{ass:weighted}, the critical intensities
$\lambda_T$ and $\lambda_c$ defined above satisfy $0<\lambda_T \leq \lambda_c < \infty.$
\end{lemma}
The proof of Lemma~\ref{lem:ex_phase_transition} is deferred to Appendix~\ref{AII}.~\\

Now, we are ready to formulate our main general result, stating the sharpness of the above phase transition under our assumptions:
\begin{theorem}[Sharpness of the phase transition]\label{thm:Susceptibility Mean-Field Bound}
Assume that either Assumption~\ref{ass:nonweighted} or Assumption~\ref{ass:weighted} is satisfied. Then
\begin{equation}\label{eq:lc=lt}
 0<\lambda_T=\lambda_c<\infty.
\end{equation}
Moreover, the following hold.
\begin{enumerate}
\item[\textup{(a)}] \textit{Subcritical regime.}
For every $0\le\lambda<\lambda_c$, there are constants
$c_\lambda>0$ and $0<C_\lambda<\infty$ such that
\[
 \theta_\lambda(k)\le C_\lambda\exp(-c_\lambda k)
 \qquad\text{for all integers }k\ge1.
\]
\item[\textup{(b)}] \textit{Supercritical regime.}
For every finite $\Lambda>\lambda_c$, there exists $c_\Lambda>0$ such that
\[
 \theta(\lambda)\ge c_\Lambda(\lambda-\lambda_c)
 \qquad\text{for }\lambda_c<\lambda\le\Lambda.
\]
\end{enumerate}
\end{theorem}

\subsection{Examples}

In the above framework, the adjacency function~$\varphi$ may depend on the vertex weights in a monotone way. This general setting covers models of interest beyond the classical non-weighted RCM. We now verify Assumption~\ref{ass:weighted} for two classes of genuinely weighted models: min-reach models and Poisson--Boolean overlap models.

\paragraph{Min-reach models.}
\begin{model*}[Min-reach RCM]\label{model:min-reach}
A model satisfying~\ref{ass:general} is called a min-reach RCM if there is a non-decreasing
function $R:[1,\infty)\to(0,\infty)$ such that
\begin{equation}\tag{MinR}\label{ass:eff-fin-range}
\begin{aligned}
  \varphi(r;a,b) &= 0
    && \text{if } r>R(\min\{a,b\}),\\
  \varphi(r;1,1) &> 0
    && \text{if } r<R(1).
\end{aligned}
\end{equation}
The local connection probabilities in this definition may be strictly
less than one at every positive distance.
\end{model*}

\begin{figure}[!ht]
  \centering
    \includegraphics[
      width=0.75\linewidth,
      trim=0pt 0pt 0pt 0pt,
      clip
    ]{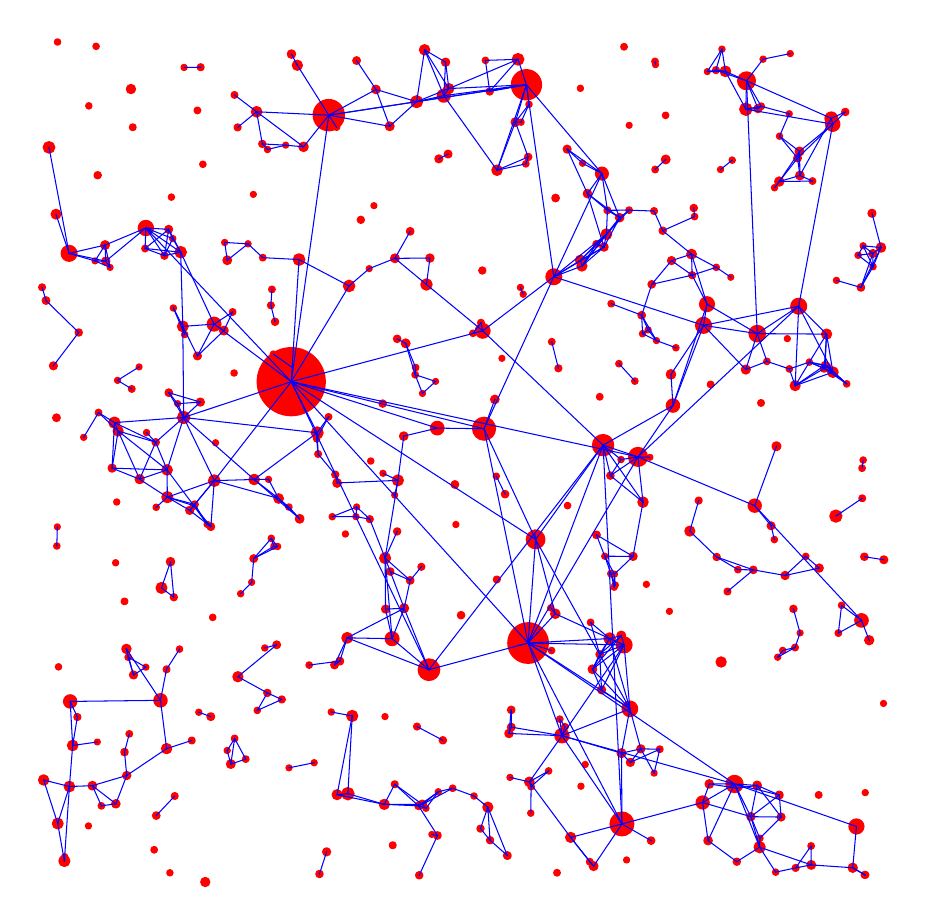}%
  \caption{Realization of a min-reach random connection model in $\R^2$ with a Pareto-tailed weight distribution $\pi$. Displayed node radii are proportional to the weights. The adjacency function is given by $\varphi(r,a,b)
= \mathbf{1}_{\{r\le \min(a,b)\}}\,\bigl(1 - e^{-\,ab/r^{3}}\bigr).$ Notice that the model does not satisfy the moment condition~\eqref{ass:MR}.} 
    \label{fig:minreach_sim}
\end{figure}

\begin{corollary}[Min-reach sharpness]\label{cor:minreach_sharpness}
Consider a min-reach RCM as defined above. Moreover, suppose that $\pi$ is either absolutely continuous or purely atomic and satisfies
\[
 \int_1^\infty e^{tR(a)^d}\,\pi(\mathrm da)<\infty
 \quad\text{for every }t>0.
\]
Then all conclusions of Theorem~\ref{thm:Susceptibility Mean-Field Bound}
hold.
\end{corollary}
\begin{proof}
All we need to do is check the conditions in Assumption~\ref{ass:weighted}: Condition~\eqref{ass:LP} holds for any $0<\rho<R(1)$ and condition~\eqref{ass:BR} for any $R_0>R(1)$, directly by assumption~\eqref{ass:eff-fin-range}. To verify~\eqref{ass:RG}, set $\ell(a):=R(a)$ for $a\ge 1$, then, for $a,b \ge 1$ and $r > \ell(a)$,
we have $r>R(a)\ge R(\min\{a,b\})$,
so $\varphi(r;a,b)=0$. Thus, the implication in~\eqref{ass:RG} holds vacuously. Condition~\eqref{ass:MR} is precisely the assumed moment
condition on $R$, and the required regularity of $\pi$
is assumed as well. 
The conclusion follows from
Theorem~\ref{thm:Susceptibility Mean-Field Bound}.\end{proof}

\setlength{\fboxsep}{0pt}   
\setlength{\fboxrule}{0.8pt}

In the min-reach models, a long connection requires both endpoint weights to be sufficiently large. In Poisson--Boolean overlap, the relevant scale is instead the sum of the radii, so a single large weight can create long connections to vertices of minimal weight, leading to expected degrees that are unbounded over the root weight. There is also a distinction between the connection rules: the preceding min-reach example has probabilistic connections within its cutoff, whereas Boolean overlap connects intersecting balls deterministically. The softened Boolean variant considered below combines a deterministic core with a region of probabilistic connections; both Boolean examples remain compatible with our replacement condition~\eqref{ass:RG}.

\paragraph{Boolean overlap models (and softened versions).}
\begin{model*}[Poisson--Boolean overlap]\label{model:boolean_overlap}
The weights are ball radii bounded below by $1$, and two vertices are joined
when the associated balls overlap: 
\begin{align}\label{eq:boolean_overlap_kernel}
\varphi(r;a,b)=\mathbf1_{\{r\le a+b\}}.
\end{align}
\end{model*}
\begin{corollary}[Poisson--Boolean sharpness]\label{cor:boolean_sharpness}
For~\eqref{eq:boolean_overlap_kernel}, assume that $\pi$ is either absolutely continuous or purely atomic and satisfies
\begin{equation}\label{eq:radius_volume_moments}
 \int_1^\infty e^{t a^d}\,\pi(\mathrm da)<\infty
 \qquad\text{for every }t>0.
\end{equation}
Then all conclusions of Theorem~\ref{thm:Susceptibility Mean-Field Bound}
hold. 
Here $|\C(o)|$ counts balls in the overlap component of the radius-one ball inserted at the origin.
\end{corollary}
\begin{proof}
Again, we just check the conditions in Assumption~\ref{ass:weighted}: For~\eqref{ass:LP} and~\eqref{ass:BR}, choose $\rho= 1$ and $R_0>2$. To check~\eqref{ass:RG}, just choose $\ell(a)=a$, and $0<\sigma<1$.
If $r>a$ and $\varphi(r;a,b)>0$, then
$r-a+\sigma\le b+\sigma\le b+1$, so indeed $\varphi(r-\ell(a)+\sigma;1,b)=1$. Condition~\eqref{ass:MR} is just
\eqref{eq:radius_volume_moments}.\end{proof}

The next model supplements deterministic overlap with probabilistic connections across small gaps.

\begin{model*}[Softened Poisson--Boolean overlap]
\label{model:soft_boolean}
Give a vertex of weight $a$ the physical radius $a-1$ and set
\[
  \varphi(r;a,b)=\min\{1,(a+b-1-r)_+\}.
\]
Overlapping balls connect deterministically, while balls separated by a gap $g\in(0,1)$ connect with probability $1-g$.
\end{model*}

\begin{corollary}\label{cor:soft_boolean_sharpness}
Suppose that $\pi$ satisfies the regularity and moment assumptions of Corollary~\ref{cor:boolean_sharpness}. Then the softened Poisson--Boolean model satisfies Assumption~\ref{ass:weighted}, and hence all conclusions of Theorem~\ref{thm:Susceptibility Mean-Field Bound} hold.
\end{corollary}

\begin{proof}[Proof]
Choose $0<\sigma<\rho<1$, $R_0=1$, and $\ell(a)=a+\sigma$. Local positivity and bounded baseline range are immediate.
If $r>\ell(a)$ and $\varphi(r;a,b)>0$, then $r<a+b-1$, so
\[
  \varphi(r-\ell(a)+\sigma;1,b)
  =\min\{1,(a+b-r)_+\}=1.
\]
This proves~\eqref{ass:RG}, while~\eqref{ass:MR} follows from the assumed moment condition.
\end{proof}

\begin{remark}
Under the regularity and moment assumptions of Corollary~\ref{cor:boolean_sharpness}, the same conclusion holds for
\[
  \varphi_p(r;a,b)
  =(1-p)\mathds1_{\{r\leq\max\{a,b\}\}}
   +p\mathds1_{\{r\leq a+b\}},
  \qquad p\in[0,1],
\]
again with $\ell(a)=a+\sigma$.
For balls of radii $a,b$, connections are certain if either contains the other's center, and occur with probability $p$ if the balls overlap but neither contains the other's center. Thus, $p=0$ corresponds to ``max-reach'' and $p=1$ to Poisson--Boolean overlap.
\end{remark}
  
\begin{figure}[H]
    \includegraphics[
    width=\textwidth,
    height=1\textheight,
    keepaspectratio
    ]{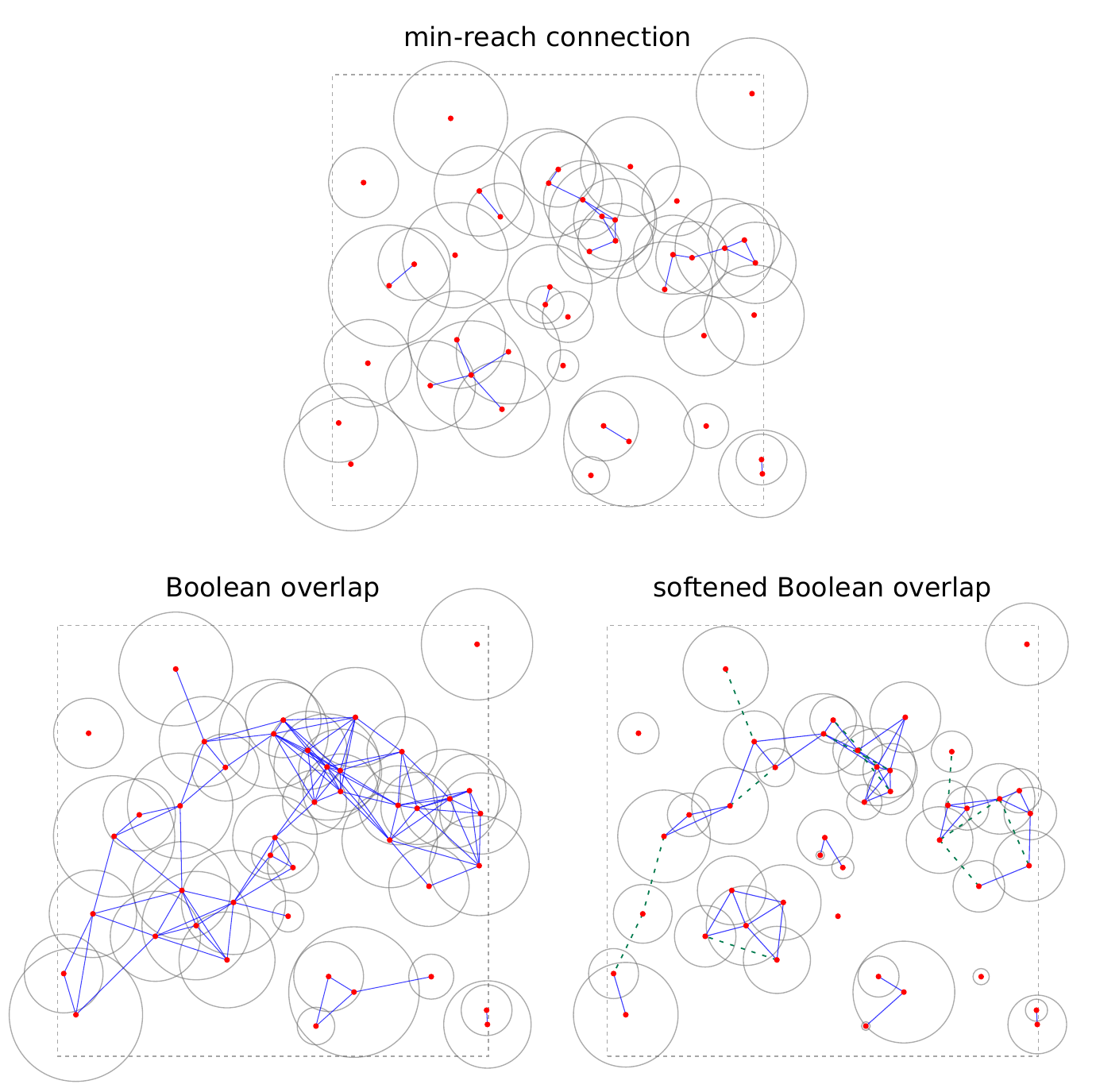}%
    \caption{Three models on the same marked Poisson configuration, with
weight law $\pi((a,\infty))=\exp\{-((a-1)/2)^{2.2}\}$, $a\geq1$.
The adjacency functions are
$\varphi(r;a,b)=\mathds1_{\{r\leq\min(a,b)\}}$,
$\varphi(r;a,b)=\mathds1_{\{r\leq a+b\}}$, and
$\varphi(r;a,b)=\min\{1,(a+b-1-r)_+\}$.
The circles have radii $a$, $a$, and $a-1$, respectively.
Min-reach connects two vertices when each center lies within the
other circle; Boolean overlap connects them when their circles
overlap. Softened Boolean overlap additionally connects circles
separated by a gap $g\in(0,1)$ with probability $1-g$
(dashed edges). All three models satisfy Assumption~\ref{ass:weighted} for $d=2$.}
    \label{fig:comparison_sim}
\end{figure}

\subsection{Discussion}\label{subsect:discuss}

In this section, we first comment on the assumptions and scope of our framework---highlighting possible extensions---before discussing the main results and their relation to previous work.

The broad structural assumptions are summarized in Assumption~\ref{ass:general}: isotropy, symmetry, and monotonicity of the adjacency function. We consider the non-weighted and weighted models separately. In the non-weighted setting, Assumption~\ref{ass:nonweighted} requires non-triviality and integrability of the adjacency function. In the weighted case, Assumption~\ref{ass:weighted} combines local positivity~\eqref{ass:LP}, bounded range at minimal weights~\eqref{ass:BR}, a geometric replacement rule for long connections~\eqref{ass:RG}, and moment conditions on the replacement scale~\eqref{ass:MR}. These conditions provide geometric control of the contribution of vertices with large weights without forcing the expected degree to be uniformly bounded over the weights.

Intuitively, local positivity~\eqref{ass:LP} guarantees that nearby vertices connect with probability bounded away from zero, while~\eqref{ass:BR} gives a finite connection range when both weights are minimal. 

Consider a vertex $x$ of weight $a$ and a potential neighbor $y$ at distance greater than $\ell(a)$. The replacement rule~\eqref{ass:RG} identifies a ball of radius $\sigma$, centered at distance $\ell(a)$ from $x$ towards $y$, such that any third vertex present in this ball, even of minimal weight, certainly connects to $y$. Such a vertex ``replaces'' $x$ in the sense that it retains a connection from $y$ to the replacement region around $x$, even if $x$ is forcibly removed from the configuration. All these balls lie within distance
$\ell(a)+\sigma$ of $x$. Condition~\eqref{ass:MR} imposes exponential moment bounds of every order on the volume of this region.

For the existence of a non-trivial phase transition, local positivity provides a comparison with an unweighted finite-range model, yielding $\lambda_c<\infty$. Bounded baseline range and replacement geometry give bounds on the spatially integrated connection function in terms of the endpoint weights. The resulting integrability allows the branching comparison establishing $\lambda_T>0$. Only polynomial moment consequences of~\eqref{ass:MR} are needed at this stage.

For sharpness, the main difficulty is to compare magnetization and pivotality at different weights. The geometric assumptions
provide a mechanism for reconnecting the open neighbors of a vertex without using that vertex, by modifying the configuration in a region whose size is controlled by its weight. Local positivity controls the probability of the required connections, and the moment condition makes the resulting comparison bounds integrable.

A natural question is whether a more economical replacement construction could reduce~\eqref{ass:MR} to a single exponential moment with a suitable exponent, or replace the deterministic attachment in~\eqref{ass:RG} by a probabilistic comparison. However, for Poisson--Boolean overlap, exponential tails of the cluster diameter necessarily require some positive exponential moment of the radius (cf.~\cite{DRT20}), whereas exponential tails of the cluster cardinality do indeed require such a moment of the volume. Without the corresponding moment, a weaker tail conclusion would be required. In our general framework, it would be interesting to understand how exactly the tail of the replacement scale $\ell$ determines the cluster-size
bounds available under our geometric assumptions.

Another possible extension concerns the deterministic attachment in~\eqref{ass:RG}. One could instead require that a suitably placed replacement vertex connects to the same neighbor with probability at least as large as the original vertex. Whether our conclusions remain valid under this probabilistic replacement rule is a natural question; such an extension would include Poisson--Boolean overlap models with independently thinned edges.

\paragraph{Relation to previous work.}
The present paper contributes to a long line of research on continuum percolation models, from the classical Gilbert graph~\cite{Gil61} and Penrose’s random connection model~\cite{Pen91} (see also~\cite{LasZie17, JKM23, HHLM22}) to modern extensions incorporating marks or weights~\cite{DH22,CD24}.
Penrose already considered unbounded connection ranges but without random weights; in contrast, weighted or marked models such as those studied by Dickson and Heydenreich~\cite{DH22} and by Caicedo and Dickson~\cite{CD24} assume a uniformly bounded influence of marks. Similarly, Pabst~\cite{Pab25} recently investigated higher-dimensional simplicial structures in marked RCMs, but again under bounded-range assumptions. 
Our work establishes sharpness of the phase transition for random connection models that allow for both infinite range of interaction and unbounded weights. Together with the recent work of Chebunin and Last~\cite{CL25}, this places the present results among the first results to extend the class of such weighted models for which sharpness is known beyond the Poisson–Boolean case~\cite{DRT20}.

An important distinction of our results from~\cite{CL25} concerns the dependence on the root weight. Their subcritical cluster
estimates are uniform over the marks, whereas ours concern the fixed root $o=(\bar 0,1)$. When the expected degree is unbounded over the marks, a uniform bound on the expected cluster size cannot hold at any positive intensity. Still, as we demonstrate here, suitable geometric and moment assumptions yield exponential cluster-size tails for a fixed root throughout the subcritical regime. In particular, even after averaging over the neighbor weight, the spatially integrated connection function need not be uniformly bounded
over the root weight.

We have seen that the class of systems covered by our results includes, among others, min-reach models, Poisson--Boolean overlap models with potentially unbounded radii and suitable softened versions thereof. The results in~\cite{CD24} cover neither of these two classes in full. \cite{CL25} covers the min-reach class under weaker moment assumptions, but its uniform bounds exclude Poisson--Boolean overlap with unbounded radii. Our geometric
assumptions also exclude models covered in~\cite{CD24,CL25},
so neither framework subsumes the other.

Our proof combines a lattice approximation with Hutchcroft's clever exploration approach to the OSSS inequality~\cite{Hut20}. Using a kindred line of reasoning, applying OSSS after discretizing space, Faggionato and Mimun~\cite{FM19} obtained exponential decay for the diameter of clusters in finite-range RCM and Poisson–Boolean models. It is instructive to compare such use of the OSSS inequality with that of Last, Peccati, and Yogeshwaran~\cite{LPY21}.
The latter work introduces stopping-set techniques that extend the OSSS inequality directly to continuum models, thereby avoiding any need for spatial discretization.
This approach is particularly elegant because it translates the notion of decision trees and revealments---originally defined for discrete Boolean functions---into a geometric and measure-theoretic framework, providing a conceptual bridge between discrete and continuum structures.  Despite the interesting methodology, the sharpness results therein also apply only to finite-range percolation models. 

For the non-weighted RCM with an integrable, isotropic, non-increasing adjacency function of possibly unbounded support, Higgs~\cite{Hig25} independently obtained exponential cluster-cardinality tails. We learned about this exciting alternative approach --- building upon a stochastic comparison technique due to Vanneuville~\cite{Van24} --- while finalizing a first version of this work. Chebunin and Last~\cite{CL25} use a completely different method, combining spatial Markov properties, stochastic comparison,
and differential inequalities.

 Another interesting development is due to Hirsch, Jahnel, and Muirhead~\cite{HJM22}, who proved a sharp phase transition for a particular class of Cox percolation models via the OSSS inequality in combination with a coarse-graining technique. Thereby they extended the OSSS framework to random environments with spatial dependence, where the vertex process itself is a Cox process, i.e., a Poisson process driven by a random intensity measure.

\subsection{Proof structure}

The paper is organized as follows: In Section~\ref{subsection:lattice_approximation}, we approximate our model by looking at a finite window of the Euclidean space and cutting that weighted space into small boxes. We then consider a Bernoulli percolation on the complete graph on those boxes, which are open with the probability of containing at least one Poisson point scattered with intensity $\lambda$; an edge is open with a probability governed by our adjacency function $\varphi$. In that setting, we derive a Russo-type derivative formula in Section~\ref{subsection:Russo}, Corollary~\ref{derivative_cluster_tail_discrete} (which is non-trivial since the underlying graph is infinite: we discretize but do not truncate the weights). In Sections~\ref{subsection:OSSS} and~\ref{subsection:exploration} following an ingenious idea of Duminil-Copin, Raoufi and Tassion, we employ an exploration scheme and the corresponding OSSS inequality~\eqref{est:ourOSSS}, to get an estimate on the probability of a large cluster containing the origin. The exploration scheme is a modification of Hutchcroft's algorithm from~\cite{Hut20} that accounts for the specifics of our Bernoulli percolation.  In Section~\ref{sect:3}, we combine the OSSS estimate with the derivative formula to obtain differential inequalities for the cluster-size tails (Corollaries~\ref{cor:3.1} and~\ref{cor:3.2}). These bounds rely on the technical Assumption~\ref{mainass} (on the discretized model) which we verify under our weighted continuum assumptions in Section~\ref{sect:Pf_mainass}. In Section~\ref{sect:Inf_vol_lim}, we show in the finite-volume setting that the differential inequality implies exponential decay in the subcritical regime $\lambda<\lambda _T$ (Lemma~\ref{lem:4fcont}) and the first part of our main result Theorem~\ref{thm:Susceptibility Mean-Field Bound} follows by taking the infinite-volume limits (Proposition~\ref{prop:UB_sus}). Similar arguments in the spirit of~\cite{Hut20} allow us to show a positive probability of the infinite cluster containing the origin in the supercritical regime $\lambda>\lambda_T$ (Proposition~\ref{prop:LB_sus}). Both the continuum and the infinite-volume limit transitions
are justified in Appendix~\ref{AIII}.

\section{Discretizing, employing an exploration scheme and using the OSSS inequality}\label{sect:2}

Before deriving the differential inequality leading to our main result, we first introduce a lattice approximation of our model. This allows us to apply the exploration-scheme approach of Hutchcroft~\cite{Hut20}, which is naturally formulated in a discrete setting. In particular, we discretize both the spatial positions and the weight marks.

One might expect that a continuum analogue of Hutchcroft’s arguments is obtained in the spirit of Last, Peccati, and Yogeshwaran~\cite{LPY21}. Their sharpness applications, however, concern finite-range models; applying their framework in our weighted, potentially infinite-range setting would require a separate argument. Since the discretized version already provides the ingredients we need for our proof, we follow this simpler and more direct route.

\subsection{Approximation of lattice models}\label{subsection:lattice_approximation}

We now describe the discrete version of our model, obtained by replacing space and weights by suitable grids and defining a product Bernoulli measure on the resulting graph.

We follow the classical idea of finite lattice approximation, similar to the approach in Meester (Section 3 of \cite{Mee95}), but in our case we also need to discretize (but not truncate) the weights. Indeed, we consider two integers $L,n \in \N$ and define the cylinder $\Lambda_n^L = [-L - 2^{-n - 1},L + 2^{-n - 1}]^d \times [1,\infty)$ which we then divide into small boxes of side length $2^{-n}$. Afterwards, we put a vertex at the spatial center of each box. We consider those to be open independently, with the same probability as a PPP of intensity $\lambda$ having at least one point present inside that box. Next, we draw edges between vertices according to the original probabilities prescribed by $\varphi$, independently for every pair of vertices. This leads to a Bernoulli percolation model on both the vertices and the edges of a complete graph.

More precisely, we consider the (possibly infinite) graph $G_n^L = (V_n^L,E_n^L)$, characterized as follows: $V_n^L:=U^L_n\times M_n$ is the set of vertices, where
\begin{align}
U_n^L := \left\{\, \Bar{u} \in 2^{-n}\mathbb{Z}^d : \Bar{u}_i \in [-L,L] \ \text{for all } i \in \{1,\dots,d\}\, \right\}.
\end{align}

If \(\pi\) is absolutely continuous, set
\[
 M_n=\{1+j 2^{-n}:j\in\mathbb N_0\},\qquad
 \Pi(m,n)=\pi([m,m+2^{-n})),\quad m\in M_n.
\]
The half-open intervals partition \([1,\infty)\), so
\(\sum_{m\in M_n}\Pi(m,n)=1\). 

If \(\pi\) is purely atomic, set 
$$
M_n=\{1\}\cup\{m\ge 1:\pi(\{m\})>0\},
\qquad
\Pi(m,n)=\pi(\{m\}),\quad m\in M_n.
$$
Again \(\sum_{m\in M_n}\Pi(m,n)=1\). 

In the non-weighted case, where \(\varphi(r;a,b)\) is independent of \(a,b\), we may instead take \(M_n=\{1\}\) and \(\Pi(1,n)=1\). In what follows, the same notation covers this case.

The edge set $E_n^L$ consists of all possible distinct pairs of vertices, i.e., we consider the complete graph on $V_n^L$. 

Next, we construct a probability measure on the configuration space $\{0,1\}^{V_n^L\cup E_n^L}$ as follows: For a given realization $\omega\in \{0,1\}^{V_n^L\cup E_n^L}$ and a vertex $(\Bar{u},m)\in V_n^L$, we say that $(\Bar{u},m)$ is open in $\omega$ if $\omega((\Bar{u},m)) = 1$ and we say that it is closed if $\omega((\Bar{u},m)) = 0$. We use the same terminology for the edges and define $P^{L,n}_{\lambda}$ as the probability measure on $\{0,1\}^{V_n^L\cup E_n^L}$  given by the inhomogeneous Bernoulli percolation on both sites and edges such that, for every $u\in V^{L}_n\setminus\{o\}$ and $e\in E^{L}_n$, we have
\begin{align*}
P^{L,n}_{\lambda}\left(u = (\Bar{u},m) \textit{ is open}\right) &=1-\e^{-\lambda 2^{-nd} \times \Pi(m,n)}, \\ P^{L,n}_{\lambda}\left(e=((\Bar{u}_1,m_1),(\Bar{u}_2,m_2)) \textit { is open}\right) &= \varphi(\abs*{\Bar{u}_1 - \Bar{u}_2},m_1,m_2).
\end{align*}
For the root $o=(\Bar{0},1)$, we fix its state to be open.

Notice that the states of all vertices and edges are stochastically independent of one another (in particular, the openness of an edge does not depend on whether its endpoints are open). Moreover, the probability of an edge to be open is constant in the intensity $\lambda$.

For a vertex $v\in V_n^L$, we denote by $\mathcal{C}(v) := \mathcal{C}_n^L(v,\omega)$ the (open) cluster of $v$ which is defined as the set of all open vertices $v'$ connected to it by an open path in $\omega\in \{0,1\}^{V_n^L\cup E_n^L}$. More precisely, $v'\in \mathcal{C}(v)$ if there exists an $l\in\N$ and a sequence of open vertices $(v_1,\dots, v_l = v')$ such that with the convention $v_0=v$ the edges $e_i=\left(v_{i-1},v_i\right)$ are also open, for each $i \in \{1,..,l\}$. Note the slightly unusual convention that we do not require the vertex $v$ itself to be open but always include it in the cluster. We omit the indices to lighten the cluster notation; nevertheless, we use a different calligraphy so it is not confused with its continuum counterpart.

Furthermore, for $k\in\N$, we define 
\begin{align}
\theta_\lambda^{L,n}(k)=P^{L,n}_\lambda(\vert \mathcal{C}(o)\vert\geq k).
\end{align}

Before proceeding, a remark about this construction deserves consideration.

\begin{remark}
Working in finite volume breaks translation invariance: points near the center and near the boundary behave differently with respect to connectivity. One could correct for this, as in \cite{Mee95}, by modifying the connection function --- for example, defining
\[
1 - \varphi_L(\bar{x},a,b)
    = \prod_{\bar{z}\in\mathbb{Z}^d} \bigl(1 - \varphi(\bar{x} + 2L\bar{z},a,b)\bigr).
\]
We choose not to adopt this modification, as translation invariance is not required for our analysis. When a translated window is needed, it can be embedded in a larger centered window; the resulting estimates must be stated with that larger window and its corresponding constants.
\end{remark}

\subsection{Derivative formula}\label{subsection:Russo}

The graph $G_n^L = (V_n^L,E_n^L)$ may be an infinite graph; however, under the measure $P^{L,n}_{\lambda}$, its open components are almost surely finite for any $\lambda>0$. Indeed, even though the number of open edges can be infinite, the number of open vertices is almost surely finite, which allows us to obtain a Margulis-Russo-type differential formula.

The idea is to first consider a finite graph in which we truncate the weights, apply a derivative formula for local functions, and then take limits to derive the formula on the infinite graph.

Fix $\lambda>0$ and $L,n\in\mathbb{N}$. Choose finite sets $M_{n,j}\subset M_n$, $j\in\N$, such that
$1\in M_{n,j}$ for all $j\in\N$ and $M_{n,j}\uparrow M_n$. In the absolutely continuous case, take
\[
  M_{n,j}
  :=\{1+i2^{-n}:0\le i<j2^n\}.
\]
In the purely atomic case, enumerate the atoms other than $1$ and let
$M_{n,j}$ consist of $1$ together with the first $j$ atoms. If there are
finitely many such atoms, keep all of them once the enumeration is
exhausted. In the non-weighted case with \(M_n=\{1\}\), take \(M_{n,j}=\{1\}\) for every \(j\).

Set
\[
  V_{n,j}^L:=U_n^L\times M_{n,j}
\]
and let $E_{n,j}^L$ consist of all distinct pairs of vertices in
$V_{n,j}^L$. The root $o$ is fixed open. The configuration space is
\[
  \Omega_{n,j}^L
  :=\{0,1\}^{V_{n,j}^L\cup E_{n,j}^L}.
\]
Define the probability measure $P_j$ as the Bernoulli product measure on $\Omega_{n,j}^L$ such that, for any $x\in V_{n,j}^L\cup E_{n,j}^L$, we have
\[
P_j(x \text{ is open}) = \mulfd(x \text{ is open}).
\]  

Set
\[
S_{n,j}^L
:=\{v=(\bar u,m)\in V_{n,j}^L\setminus\{o\}:\Pi(m,n)>0\}.
\]
For any $\omega\in\Omega_{n,j}^L$ satisfying $\omega(o)=1$
and $\omega(v)=0$ for every other vertex
$v=(\bar u,m)$ with $\Pi(m,n)=0$, we may write
\begin{align}\label{eq:PH-explicit}
    P_j(\omega)
    = \Bigg[\prod_{v\in S_{n,j}^L}
       p_\lambda(v)^{\omega(v)}
       (1-p_\lambda(v))^{1-\omega(v)}\Bigg]\,
      \alpha_j(\omega),
\end{align}
where
\[
p_\lambda(\bar u,m)=1-e^{-\lambda 2^{-nd}\Pi(m,n)},
\]
and 
\[
\alpha_j(\omega)
= \prod_{e\in E_{n,j}^L}P_j(e\text{ open})^{\omega(e)}P_j(e\text{ closed})^{1-\omega(e)}.
\]
For all other configurations, $P_j(\omega)=0$. Note that $\alpha_j(\omega)$ does not depend on $\lambda$.

We set $$I_{(\bar u,m)}:=\omega((\bar u,m)),$$ denote by $E_{P_j}$ the expectation under $P_j$, and by $\mathrm{Cov}_{P_j}$ the corresponding covariance.

\begin{lemma}
    For $F:\Omega_{n,j}^L\to\mathbb R$, the finite-product Russo formula is
\begin{equation}\label{eq:finite-derivative}
  \frac{\mathrm d}{\mathrm d\lambda}E_{P_j}[F]
  =
  \sum_{v=(\bar u,m)\in S_{n,j}^L}
  \frac{2^{-nd}\Pi(m,n)}{p_\lambda(v)}
  \operatorname{Cov}_{P_j}(F,I_v).
\end{equation}
\end{lemma}

\begin{remark}
   Since we consider percolation on both edges and vertices, the measure $P_j$ corresponds to the product ($q=1$) case of the random-cluster measure from~\cite{Hut20}.  
   Formula~\eqref{eq:finite-derivative} is the analogue of~\cite[Eq.\ (2.1)]{Hut20}. For this proof, replicated here for the reader's convenience, we follow the ideas from~\cite[Proposition 4]{BGK93}.
\end{remark}

\begin{proof}
    For each $(\bar u,m)\in S_{n,j}^L$,
    \[
        \frac{\mathrm{d}}{\mathrm{d}\lambda}p_\lambda(\bar u,m)
        =2^{-nd}\Pi(m,n)\bigl(1-p_\lambda(\bar u,m)\bigr).
    \]

    The derivatives of the two factors in~\eqref{eq:PH-explicit} are
    \begin{align*}
        \frac{\mathrm{d}}{\mathrm{d}\lambda}
        p_\lambda(\bar u,m)^{I_{(\bar u,m)}}
        &=
        I_{(\bar u,m)}\,2^{-nd}\Pi(m,n)
        (1-p_\lambda(\bar u,m))
        p_\lambda(\bar u,m)^{I_{(\bar u,m)}-1},\\
        \frac{\mathrm{d}}{\mathrm{d}\lambda}
        (1-p_\lambda(\bar u,m))^{1-I_{(\bar u,m)}}
        &=
        -(1-I_{(\bar u,m)})\,2^{-nd}\Pi(m,n)
        (1-p_\lambda(\bar u,m))^{1-I_{(\bar u,m)}}.
    \end{align*}

    Combining these contributions and using that $\alpha_j(\omega)$
    is $\lambda$-independent yields
    \begin{align}\label{eq:derivative-of-density}
        \frac{\mathrm{d}}{\mathrm{d}\lambda}P_j(\omega)
        =
        \sum_{(\bar u,m)\in S_{n,j}^L}
        2^{-nd}\Pi(m,n)
        \left(\frac{I_{(\bar u,m)}}{p_\lambda(\bar u,m)}-1\right)
        P_j(\omega).
    \end{align}

    Differentiating under the finite sum gives
    \begin{align*}
        \frac{\mathrm{d}}{\mathrm{d}\lambda}E_{P_j}[F]
        &=
        \sum_{\omega\in\Omega_{n,j}^L}
        F(\omega)\frac{\mathrm{d}}{\mathrm{d}\lambda}P_j(\omega)\\
        &=
        \sum_{(\bar u,m)\in S_{n,j}^L}
        2^{-nd}\Pi(m,n)\,
        E_{P_j}\!\left[
          F\left(\frac{I_{(\bar u,m)}}{p_\lambda(\bar u,m)}-1\right)
        \right].
    \end{align*}

    Finally,
    \[
        E_{P_j}\!\left[
          F\left(\frac{I_{(\bar u,m)}}{p_\lambda(\bar u,m)}-1\right)
        \right]
        =
        \frac{1}{p_\lambda(\bar u,m)}
        \operatorname{Cov}_{P_j}(F,I_{(\bar u,m)}),
    \]
    which gives~\eqref{eq:finite-derivative}.
\end{proof}

Next, define
\[
S_n^L
:=\{v=(\bar u,m)\in V_n^L\setminus\{o\}:\Pi(m,n)>0\}.
\]

For a big class of functions, we can extend the preceding lemma to the infinite graph $G_n^L = (V_n^L,E_n^L)$. In particular, if we fix $k \in \N$ and consider the function $\mathds{1}_{\{\abs*{\mathcal{C}(o)} \geq k\}}$, then we obtain the following result.

\begin{corollary}
\label{derivative_cluster_tail_discrete}
For every $k\in\mathbb N$ and $\lambda>0$,
\begin{align}\label{eq:derivative_cluster_tail_discrete}
  \frac{\mathrm d}{\mathrm d\lambda}
  P_\lambda^{L,n}\bigl(|\mathcal C(o)|\ge k\bigr)
  =
  \sum_{v=(\bar u,m)\in S_n^L}
  \frac{2^{-nd}\Pi(m,n)}{p_\lambda(v)}
  \operatorname{Cov}_{P_\lambda^{L,n}}
  \left(\mathbf1_{\{|\mathcal C(o)|\ge k\}},I_v\right)
\end{align}
and the series converges absolutely.
\end{corollary}

\begin{proof}
Let
\[
  N_j:=U_n^L\times(M_n\setminus M_{n,j}),
  \qquad
  A_j:=\{I_v=0\text{ for every }v\in N_j\},
\]
and put
\[
  q_{n,j}:=\sum_{m\notin M_{n,j}}\Pi(m,n)\downarrow 0.
\]
Independence yields
\[
  P_\lambda^{L,n}(A_j)
  =\exp\{-\lambda 2^{-nd}|U_n^L|q_{n,j}\}
  \longrightarrow 1.
\]
We work under the conditional measure $P_\lambda^{L,n}(\,\cdot\mid A_j)$.
On $A_j$, the cluster $\mathcal C(o)$ in the infinite graph coincides with the cluster in the graph restricted to $V_{n,j}^L$. By the product structure, the conditional law of the configuration on $V_{n,j}^L\cup E_{n,j}^L$ given $A_j$ is $P_j$. Hence
\[
  P_j\bigl(|\mathcal C(o)|\ge k\bigr)
  =P_\lambda^{L,n}\bigl(|\mathcal C(o)|\ge k\mid A_j\bigr),
\]
and, for $v\in S_{n,j}^L$,
\[
  \operatorname{Cov}_{P_j}
  \bigl(\mathbf1_{\{|\mathcal C(o)|\ge k\}},I_v\bigr)
  =
  \operatorname{Cov}_{P_\lambda^{L,n}}
  \bigl(\mathbf1_{\{|\mathcal C(o)|\ge k\}},I_v\mid A_j\bigr).
\]

Applying the finite-product Russo formula gives
\begin{align}\label{id:applying_finite_Russo}
    \frac{\mathrm d}{\mathrm d\lambda}
  P_\lambda^{L,n}\bigl(|\mathcal C(o)|\ge k\mid A_j\bigr)
  =
  \sum_{v=(\bar u,m)\in S_{n,j}^L}
  \frac{2^{-nd}\Pi(m,n)}{p_\lambda(v)}
  \operatorname{Cov}_{P_\lambda^{L,n}}
  \bigl(\mathbf1_{\{|\mathcal C(o)|\ge k\}},I_v\mid A_j\bigr).
\end{align}
If $v\in S_n^L$ and $m\notin M_{n,j}$, then $I_v=0$ under $A_j$, so the corresponding conditional covariance vanishes. Thus we may equivalently take the sum over $S_n^L$.

As $j\to\infty$, $P_\lambda^{L,n}(A_j)\to1$. Thus, for each  $v\in S_n^L$,

$$\operatorname{Cov}_{P_\lambda^{L,n}}
  \bigl(\mathbf1_{\{|\mathcal C(o)|\ge k\}},I_v\mid A_j\bigr)\to \operatorname{Cov}_{P_\lambda^{L,n}}
  \bigl(\mathbf1_{\{|\mathcal C(o)|\ge k\}},I_v\bigr).$$ 
Moreover,
\[
  \mathbb E_{P_\lambda^{L,n}}[I_v\mid A_j]
  =
  \begin{cases}
    p_\lambda(v), & m\in M_{n,j},\\
    0, & m\notin M_{n,j}.
  \end{cases}
\]
Since both functions in the covariance are indicators,
\[
  \left|
  \operatorname{Cov}_{P_\lambda^{L,n}}
  \bigl(\mathbf1_{\{|\mathcal C(o)|\ge k\}},I_v\mid A_j\bigr)
  \right|
  \le
  \mathbb E_{P_\lambda^{L,n}}[I_v\mid A_j]
  \le p_\lambda(v).
\]
Consequently, the absolute value of each summand on the right-hand side of~\eqref{id:applying_finite_Russo} is at most
$2^{-nd}\Pi(m,n)$, and
\[
  \sum_{v=(\bar u,m)\in S_n^L}2^{-nd}\Pi(m,n)
  \le 2^{-nd}|U_n^L|<\infty.
\]
Dominated convergence, using the pointwise convergence of the conditional covariances and the summable bound $2^{-nd}\Pi(m,n)$, provides for every $\lambda>0$
\[
  \lim_{j\to\infty}
  \frac{\mathrm d}{\mathrm d\lambda}
  P_\lambda^{L,n}\bigl(|\mathcal C(o)|\ge k\mid A_j\bigr)
  =
  \sum_{v=(\bar u,m)\in S_n^L}
  \frac{2^{-nd}\Pi(m,n)}{p_\lambda(v)}
  \operatorname{Cov}_{P_\lambda^{L,n}}
  \bigl(\mathbf1_{\{|\mathcal C(o)|\ge k\}},I_v\bigr).
\]
For each fixed $v$, the conditional covariances converge locally uniformly in $\lambda$, since $P_\lambda^{L,n}(A_j)\to1$ locally uniformly. The bound $2^{-nd}\Pi(m,n)$, summable over $v$ and uniform in $j$ and $\lambda$, therefore upgrades the displayed convergence of the finite-product derivatives to locally uniform convergence. Absolute convergence of the series follows from the same summable bound. 

Since the conditional probabilities converge to $P_\lambda^{L,n}(|\mathcal C(o)|\ge k)$, this locally uniform convergence of their derivatives permits interchanging the limit with differentiation, proving \eqref{eq:derivative_cluster_tail_discrete} for every $\lambda>0$.
\end{proof}

\begin{remark}\label{rem:russi_cov-pif}
    Notice that for all $(\bar{u},m)\in V^L_n$,
    \begin{align*}
        \mathrm{Cov}_{\mulfd}\big(\mathds{1}_{\{\abs*{\mathcal{C}(o)} \geq k\}},
        I_{(\bar{u},m)}\big)
        =&\ p_\lambda(\bar{u},m)\big(1-p_\lambda(\bar{u},m)\big)\times\\
        &\times\mulfd\Big((\bar{u},m) \text{ is pivotal for }
        \{\abs*{\mathcal{C}(o)} \geq k\}\Big),
    \end{align*}
    where we call a vertex $v$ \emph{pivotal} for an increasing event $A$ if $A$
    occurs when $v$ is open but not when it is closed, with all other states unchanged.
    This yields the alternative formula
    \begin{align*}
        \frac{\dd}{\dd \lambda}\,
        \mulfd\big(\abs*{\mathcal{C}(o)} \geq k\big)
        = &\sum_{(\bar{u},m) \in S_n^L}
          \frac{\dd}{\dd \lambda}P^{L,n}_\lambda\Big((\bar{u},m)\textit{ is open}\Big)\times\\
          &\times\mulfd\Big((\bar{u},m) \text{ is pivotal for }
          \{\abs*{\mathcal{C}(o)} \geq k\}\Big).
    \end{align*}
\end{remark}

\subsection{OSSS inequality}\label{subsection:OSSS}

As previously stated, in order to derive the differential inequality present in the next section, we make use of the OSSS inequality. More precisely, we use its generalization to covariances from~\cite{DRT19b} extended to decision forests in~\cite{Hut20}. Before we state it, we need to make some auxiliary definitions. In the following, $S$ denotes an arbitrary countable set.

\begin{definition}
    A probability measure $\mu$ on $\{0,1\}^S$ is said to be \textit{monotonic} if 
    \begin{align*}
        \mu(\omega(s) = 1 \mid \omega_{\vert_{F}} = \xi) \geq \mu(\omega(s) = 1 \mid \omega_{\vert_{F}} = \zeta)
    \end{align*}
    whenever $F \subset S$, $s \in S$ and $\xi, \zeta \in \{0,1\}^F$ are such that $\xi \geq \zeta$ (coordinatewise). Here $\omega\vert_F$ denotes the restriction of $\omega \in \{0,1\}^S$ to $F$.
\end{definition}

Notice that any product measure on $\{0,1\}^S$ is monotonic. Furthermore, if $\mu_1$ is a monotonic measure on $\{0,1\}^{S_1}$ and $\mu_2$ is a monotonic measure on $\{0,1\}^{S_2}$, then the product $\mu_1 \otimes \mu_2$ is monotonic on $\{0,1\}^{S_1\sqcup S_2}$. Here $S_1 \sqcup S_2$ denotes the disjoint union of the sets $S_1$ and $S_2$.

\begin{definition}
    A \textit{decision tree} is a function $T: \{0,1\}^S \to S^\N$ satisfying the following property: $T_1(\omega) = s_1$ for some fixed $s_1 \in S$ and for every $r \geq 2$ there exists a function $f_r: \left(S \times \{0,1\}\right)^{r-1} \to S$ such that
\begin{align*}
    T_r(\omega) = f_r[(T_i,\omega(T_i))_{i = 1}^{r-1}].
\end{align*}
\end{definition}
This means that $T$ is a deterministic procedure that first queries the value of $\omega(s_1)$ and at each subsequent step it chooses which element of $S$ to investigate next based on the information it has gathered so far.

Let $\mu$ be a probability measure on $\{0,1\}^S$ and let $\omega$ be a random variable with law $\mu$. Given a decision tree $T$ and $r \geq 1$, let $\mathcal{F}_r(T)$ be the $\sigma$-algebra generated by the random variables $\left\{T_j(\omega),\omega(T_j(\omega)) : 1 \leq j \leq r \right\}$ and let $\mathcal{F}(T)=\sigma(\bigcup_r \mathcal{F}_r(T))$. For a measurable function $f:\{0,1\}^S \rightarrow[-1,1]$, we say that $T$ determines $f$ if $f(\omega)$ is measurable with respect to the $\mu$-completion of $\mathcal{F}(T)$. 

Then, for any $s \in S$ we define the \textit{revealment probability} 
\begin{align}
    \delta_s(T) := \mu(\exists r \geq 1 \text{ such that } T_r(\omega) = s).
\end{align}

Following Hutchcroft~\cite[Eq.~(2.3)]{Hut20}, which in turn adopts the
notation introduced in~\cite{OSSS05}, we define,
for measurable functions $f,g : \{0,1\}^{S} \to \mathbb{R}$,
\begin{align}\label{def:Covr}
    \mathrm{Covr}_{\mu}[f,g]
    :=
    \mu\otimes\mu\bigl(|f(\omega_{1}) - g(\omega_{2})|\bigr)
    -
    \mu\bigl(|f(\omega_{1}) - g(\omega_{1})|\bigr),
\end{align}
where $\omega_{1},\omega_{2}$ are independent $\mu$-distributed samples.
If $f$ and $g$ are $\{0,1\}$-valued then
\begin{align}\label{def:Covr-Boolean}
    \mathrm{Covr}_{\mu}[f,g]
    =
    2\mathrm{Cov}_{\mu}[f,g].
\end{align}

Now, we extend this setup to the case of forests instead of trees, in the same way as in~\cite{Hut20}. 

\begin{definition}
 A \textit{decision forest} is a set of decision trees $F = \{T^i : i \in I\}$ for a countable set $I$. Given a decision forest $F$, let $\mathcal{F}(F)$ be the smallest $\sigma$-algebra containing all of the $\sigma$-algebras $\mathcal{F}(T^i)$. Given a measure $\mu$ on $\{0,1\}^S$, a function $f : \{0,1\}^S \to \R$ and a decision forest $F$, we say that $F$ computes $f$ if $f$ is measurable with respect to the $\mu$-completion of the $\sigma$-algebra $\mathcal{F}(F)$. We also define the revealment probability $\delta_s(F,\mu)$ to be the probability under $\mu$ that there exists $i \in I$ and $r \geq 1$ such that $T^i_r(\omega) = s$.   
\end{definition} We are finally ready to state the generalization of the OSSS inequality, named after O’Donnell–Saks–Schramm–Servedio, Corollary 2.4 in \cite{Hut20}.

\begin{proposition}[OSSS inequality]
    \label{prop:osss_inequality}
    Let $S$ be a finite or countably infinite set and let $\mu$ be a monotonic measure on $\{0,1\}^S$. Then for every pair of measurable, $\mu$-integrable functions $f, g : \{0,1\}^S \to \R$ with $f$ increasing and every decision forest $F$ computing $g$ we have that 
    \begin{align}
    \label{eq:osss_ineq}
    \frac{1}{2}\abs*{\emph{Covr}_{\mu}[f,g]} \leq \sum_{s \in S}\delta_s(F,\mu)\emph{Cov}_{\mu}[f,\omega(s)]. 
    \end{align}
\end{proposition}

We use this inequality in a particular way: in our case the measure $\mu$ will
always be a product measure, and it is therefore convenient to work with
influences rather than covariances.  
Given a product measure $\mu$, a $\{0,1\}$-valued function $f$, and $s\in S$, we
define the influence of the $s$-th coordinate by
\begin{align}
    \textsc{Inf}_s^\mu(f)
    = \mu(f(\omega)\neq f(\tilde\omega)),
\end{align}
where $\tilde\omega$ is obtained from $\omega$ by resampling the $s$-th
coordinate independently according to the marginal $\mu_s$.

\begin{corollary}\label{our_OSSS}
    Let $S$ be a finite or countably infinite set and let $\mu$ be a product
    measure on $\{0,1\}^S$.  
    Then for every pair of measurable, $\mu$-integrable functions
    $f,g:\{0,1\}^S\to\{0,1\}$ with $f$ increasing and every decision forest $F$
    computing $g$ we have 
    \begin{align}\label{est:ourOSSS}
        \abs{\mathrm{Cov}_{\mu}[f,g]}
        \le \frac12 \sum_{s\in S}\delta_s(F,\mu)\,\textsc{Inf}_s^\mu(f).
    \end{align}
\end{corollary}

\begin{proof}
Since $f$ and $g$ are $\{0,1\}$–valued, identity~\eqref{def:Covr-Boolean}
gives
\[
\frac{1}{2}\mathrm{Covr}_\mu[f,g]=\mathrm{Cov}_\mu[f,g].
\]
Next, because $\mu$ is a product measure and $f$ is increasing in $s\in S$, we can use
the standard pivotal identity
\[
\mathrm{Cov}_\mu(f,\omega(s))
    = \frac 12\textsc{Inf}_s^\mu(f),
\]
compare also with the first statement in Remark~\ref{rem:russi_cov-pif}.
Inserting these into the OSSS inequality~\eqref{eq:osss_ineq} from the preceding proposition,
\[
\tfrac12\,\abs{\mathrm{Covr}_\mu[f,g]}
    \le \sum_{s\in S}
       \delta_s(F,\mu)\,\mathrm{Cov}_\mu(f,\omega(s)),
\]
yields~\eqref{est:ourOSSS}.
\end{proof}

\subsection{Cluster exploration scheme}\label{subsection:exploration}

In this subsection, we construct an augmented variant of the graph \(G_n^L\). The purpose of this augmentation is to enable a coupling between a magnetization-based argument and the OSSS inequality. To that end, we introduce a specific cluster–exploration procedure (formulated as a tailored decision forest) together with the associated Boolean functions. These will serve as the objects to which we subsequently apply the OSSS inequality.

\paragraph{Algorithm setup.}
Fix $L,n,k\in\N$ and let $\tilde\gamma>0$ be an intensity parameter.  
We introduce an auxiliary Bernoulli field, via a set of ``copy vertices'', to decide which original vertices of $G_n^L$ are declared ``green''.  
More precisely, we add a copy vertex $\tilde v$ for each $v\in V_n^L$.  
Each copy vertex $\tilde v$ is open with probability
\(
h := 1 - \exp\{-\tilde\gamma\},
\)
independently of everything else.  
We then declare the original vertex $v$ to be green whenever its copy $\tilde v$ is open.  
Finally, a vertex $u\in G_n^L$ is said to be in the ``green component’’ whenever it is connected in $G_n^L$ to at least one green vertex (i.e., one whose copy is open).

More formally, let $\omega \in \{0,1\}^{V_n^L \cup E_n^L}$ be distributed according to $P=P_\lambda^{L,n}$.  
Let $\tilde{V}_n^L$ be a disjoint copy of $V_n^L$, and for each $v\in V_n^L$ denote by
$\tilde v$ its copy.  
Let $\rho\in\{0,1\}^{\tilde{V}_n^L}$ be independent of $\omega$ and distributed according 
to the Bernoulli product measure
\[
Q = Q_{\tilde\gamma}^{L,n},
\]
so that $\rho(\tilde v)=1$ with probability $h$ independently for all
$v\in V_n^L$.  
Set
\[
\mu = \mu^{L,n}_{\lambda,\tilde\gamma} := P\otimes Q,
\]
which is a product measure on $\{0,1\}^{V_n^L \cup \tilde{V}_n^L \cup E_n^L}$.

Next, define the increasing functions
$f,g:\{0,1\}^{V_n^L\cup E_n^L \cup \tilde{V}_n^L}\to\{0,1\}$ by
\begin{align}\label{def: fg}
    f(\omega,\rho) = \mathds{1}(\abs{\mathcal{C}(o)} \ge k),
    \qquad
    g(\omega,\rho) = \mathds{1}\big(\rho(\tilde{u}) = 1 
        \text{ for some } u\in \mathcal{C}(o)\big).
\end{align}
The function $g$ is the one computed by the exploration scheme described below.

\paragraph{Informal description.} 
Before introducing the algorithm for the exploration scheme, we first convey the informal idea behind the procedure. Our goal is to explore the clusters of all green vertices in the complete graph simultaneously, according to the following rules. For each vertex in the graph, we initiate an exploration only if the vertex is green, and also open. In that case, its open cluster is successively revealed --- including both edges and vertices. An edge is revealed only if it is incident to a vertex that has already been revealed to be open, and not incident to any vertex that has already been revealed to be closed. Once an edge is revealed to be open, the scheme immediately reveals its second endpoint, provided it has not yet been revealed. The exploration started at any vertex need not terminate in finite time, even when the open cluster is finite. However, we want it to terminate in the case that we reach $o$.

The revealment probability of any edge is thus by design controlled by the probability that an endpoint of the edge is open. This is crucial, because it ensures that the total contribution of the edge influences vanishes in the continuum limit --- as established in Appendix~\ref{AIII}.

\paragraph{Formal description.}
Let us describe the procedure precisely now. A decision forest $F = \{T^{u} \mid u \in V_n^L\}$ which computes $g$ is defined as follows: 

 For the root $o$, the decision tree $T^o$ first reveals the $\rho$-status of $\tilde{o}$ and the exploration halts. If $\rho(\tilde{o})=0$, set $T_r^o(\omega,  \rho) = \tilde{o}$ for $r \in \N$, otherwise set $T_r^o(\omega,  \rho) = o$ for $r\geq 2$.

 Next, fix any vertex $u \in V_n^L\setminus\{o\}$ and fix an enumeration of the edges in $E_n^L$. Set $T_1^u(\omega,\rho) = \tilde{u}$. If $\rho(\tilde{u}) = 0$, halt the exploration and set $T_r^u(\omega,\rho) = \tilde{u}$ for every $r \geq 2$. Otherwise, if $\rho(\tilde{u}) = 1$, set $T_2^u(\omega,\rho) = u$. If $\omega(u) = 0$, halt the exploration and set $T_r^u(\omega,\rho) = u$ for every $r \geq 2$. 

Finally, if we found $\rho(\tilde{u}) = 1$ and $\omega(u) = 1$, define $T_r^u(\omega,\rho)$, for $r \geq 2$, the following way: At each step of the decision tree, we will have a set of revealed open vertices $A_r^u$, a set of revealed closed vertices $B_r^u$, a set of revealed open edges $O_r^u$ and a set of revealed closed edges $C_r^u$. Moreover, we set $I_r^u:=\left\{
e=\{x,y\}\in E_n^L: \{x,y\}\cap A_r^u\neq\emptyset, \{x,y\}\cap B_r^u=\emptyset, e\notin O_r^u\cup C_r^u\right\}$.  We initialize by setting $A_2^u = \{u\}$ and $B_2^u = O_2^u = C_2^u = \emptyset$. Assume that $r \geq 2$ and that we have computed $(A_s^u,B_s^u,O_s^u,C_s^u,T_s^u)$ for every $s \leq r$. We are now ready to formulate the algorithm:
    \begin{itemize}
       \item \underline{Case 1}: We have $T_r^u \in E_n^L$, $\omega(T_r^u) = 1$ and an endpoint of $T_r^u$ is not in $A_r^u\cup B_r^u$.
       \medskip
       
        We set $T_{r+1}^u$ to be the endpoint of $T_r^u$ that is not in $A_r^u\cup B_r^u$. If $T_{r+1}^u=o$, halt and set $T_{s}^u=o$ for all $s>r$. Otherwise, if $\omega(T_{r+1}^u) = 1$, we set $A_{r+1}^u = A_r^u \cup \{T_{r+1}^u\}$, set $B_{r+1}^u = B_r^u$, set $O_{r+1}^u = O_r^u$ and set $C_{r+1}^u = C_r^u$. Otherwise, $\omega(T_{r+1}^u) = 0$ and we set $A_{r+1}^u = A_r^u$, set $B_{r+1}^u = B_r^u \cup \{T_{r+1}^u\}$, set $O_{r+1}^u = O_r^u$ and set $C_{r+1}^u = C_r^u$.  
        \item \underline{Case 2}: One of the following three scenarios occurs: 
(i) $T_r^u \in E_n^L$, $\omega(T_r^u)=1$, and both endpoints of $T_r^u$ lie in $A_r^u$; 
(ii) $T_r^u \in E_n^L$ and $\omega(T_r^u)=0$; or 
(iii) $T_r^u \in V_n^L$.
\medskip

        If $I_r^u$ is empty, halt the exploration and set $T_n^u=T_r^u$ for all $n\geq r$. Otherwise, set $T_{r+1}^u$ to be the edge minimal in $I_r^u$ with respect to the fixed edge enumeration. If $\omega(T_{r+1}^u) = 1$, we set $A_{r+1}^u = A_r^u$, set $B_{r+1}^u = B_r^u$, set $O_{r+1}^u = O_r^u \cup \{T_{r+1}^u\}$ and set $C_{r+1}^u = C_r^u$. Otherwise, $\omega(T_{r+1}^u) = 0$ and we set $A_{r+1}^u = A_r^u$, set $B_{r+1}^u = B_r^u$, set $O_{r+1}^u = O_r^u$ and set $C_{r+1}^u = C_r^u \cup \{T_{r+1}^u\}$.
    \end{itemize}
The above procedure does indeed define a decision tree $T^u$, $u\neq o$.

We now proceed to make use of the OSSS inequality, Corollary \ref{our_OSSS}, applying it to the decision forest $F$ defined above and the associated Boolean functions $f,g$. To do so, we consider the set $S = V_n^L \cup \tilde{V}_n^L \cup E_n^L $ and the measure
\begin{align}
\mu=P\otimes Q=P_\lambda^{L,n}\otimes Q_{\tilde\gamma}^{L,n} ,   
\end{align}
which is a product measure on $\{0,1\}^S$.

\begin{fact}[Bound for revealment probabilities of vertices]\label{fact:revealment_bound}
Let $\delta_{\tilde\gamma}(\Bar{u},m)=\delta^{L,n}_{\tilde\gamma}(\Bar{u},m)$ denote the revealment probability of the vertex $(\Bar{u},m)$ with respect to $\mu$ in the above exploration scheme. We can identify this probability by noticing that a vertex $(\Bar{u},m)$ being explored implies that it is connected to a green vertex by a path of open edges and vertices (apart from possibly $(\bar u,m)$ itself). To be more precise: Let $\lambda, \tilde\gamma>0$, $L,n\in\N$ and $(\Bar{u},m)\in V_n^L$, then 
\begin{align*}
\nonumber\delta_{\tilde\gamma}(\Bar{u},m)
\leq& 
\mu\Big(\textit{there exists a green vertex }v\in \mathcal{C}(\Bar{u},m)\Big)
\nonumber
\\=& 
E\left[1-\e^{-{\tilde\gamma}\vert \mathcal C(\Bar{u},m)\vert} \right].
\end{align*}
Equality need not hold (other than at $(\Bar{u},m)=o$) because we choose to halt the explorations upon reaching $o$.
\end{fact}

\begin{remark}[Connection to magnetization]
The above estimate relates the revealment probability to the magnetization at $(\Bar{u},m)$ under a (reparametrized) external field. In the exploration scheme, a vertex is \emph{revealed} precisely when the algorithm uncovers its state, which occurs only if it is connected to a green (ghost) vertex. This mirrors the standard representation of magnetization in random–cluster models (cf.~Aizenman~\cite{Aiz82}).
\end{remark}

Notice that Corollary \ref{our_OSSS} yields the following inequality for the choice of functions $f,g$ as in~\eqref{def: fg}:
\begin{align}
\text{Cov}_\mu[f,g]&\leq  \sum_{u\in V_n^L}\delta_{\tilde\gamma}(u) \textsc{Inf}^P_u(f)+\sum_{e\in E_n^L}\delta_{\tilde\gamma}(e) \textsc{Inf}^P_e(f)+\sum_{\tilde{u}\in \tilde{V}_n^L}\delta_{\tilde\gamma}(\tilde{u})\textsc{Inf}^Q_{\tilde{u}}(f) \nonumber
\\&=\sum_{u\in V_n^L}\delta_{\tilde\gamma}(u) \textsc{Inf}^P_u(f)+\sum_{e\in E_n^L}\delta_{\tilde\gamma}(e) \textsc{Inf}^P_e(f).
\label{cov_f_g_bound}
\end{align}
The sum over the copy vertices vanishes because the value of $f$ does not depend on the states of copy vertices by definition (and therefore $\textsc{Inf}^Q_{\tilde{u}}(f) = 0$ for any $\tilde{u}\in \tilde{V}_n^L$). Hence, we get the following proposition:

\begin{proposition}\label{prop:OSSSmain}
    Let $L,n \in \N$. We have
    \begin{align}\label{est:OSSSmain}
        \nonumber&\sum_{u \in V_n^L} \delta_{\frac\gamma k}(u)\textsc{Inf}^P_u(f) \\\geq& \bigg(1 - \e^{-\gamma} - E\bigg[1-\exp\left\{-\frac{\gamma\vert \mathcal{C}(o)\vert}{k}\right\}\bigg]\bigg)P(\abs*{\mathcal{C}(o)} \geq k) - \sum_{e\in E_n^L}\delta_{\frac \gamma k}(e) \textsc{Inf}^P_e(f)
    \end{align}
    for all $k\geq 1$ and $\gamma >0$.
\end{proposition}
\begin{proof}
Let $\tilde\gamma=\frac{\gamma}{k}$ and denote by $\rho(A)$ the number of ``green" vertices in $A\subset V_n^L$. Then we can bound the left side of~\eqref{cov_f_g_bound} as follows
 \begin{align*}
\text{Cov}_\mu[f,g] =& \mu(\abs*{\mathcal{C}(o)} \geq k,\rho(\mathcal{C}(o)) \geq 1 ) - \mu(\rho(\mathcal{C}(o)) \geq 1 ) P(\abs*{\mathcal{C}(o)} \geq k) \nonumber \\
=& E\left[\left(1-\exp\left\{-{\gamma\vert \mathcal{C}(o)\vert}/{k}\right\}\right)\mathds1_{\{\abs*{\mathcal{C}(o)} \geq k\}}\right] - E\left(1-\exp\left\{-{\gamma\vert \mathcal{C}(o)\vert}/{k}\right\}\right)P(\abs*{\mathcal{C}(o)} \geq k) \nonumber \\
\geq& (1 - \e^{- \gamma})P(\abs*{\mathcal{C}(o)} \geq k) - E\left[1-\e^{-\gamma\vert \mathcal{C}(o)\vert/k}\right]P(\abs*{\mathcal{C}(o)} \geq k),
\end{align*}   
so that the claim follows directly from~\eqref{cov_f_g_bound}.
\end{proof}

\section{Derivation of differential inequalities.}\label{sect:3}

In this section, we use Proposition~\ref{prop:OSSSmain} in combination with the variant of Russo's formula from Subsection~\ref{subsection:Russo} to derive differential inequalities for the tail function of the cluster size $\vert\mathcal C(o) \vert$ --- here in the discrete finite-volume setup given by our approximation of the random connection model. We do not need to impose Assumption~\ref{ass:weighted} here, instead we consider a set of technical assumptions for the discrete finite-volume setup, in particular to have some uniform control over ratios of vertex influences (naturally, we will later verify that these technical assumptions are verified under Assumption~\ref{ass:weighted}, see Section~\ref{sect:Pf_mainass}). Taking the continuum limit $n\to\infty$ and the infinite-volume limit $L\to\infty$ for those differential inequalities will provide our main result in Section~\ref{sect:Inf_vol_lim} --- under suitable moment conditions for the weight distribution $\pi$.~\\

First of all,  we want to make some simple observations with regard to the revealment probabilities and influences appearing on the left side of~\eqref{est:OSSSmain} in Proposition~\ref{prop:OSSSmain} and introduce some additional notation in the process. We will continue to write $P=P^{L,n}_\lambda$ in the following.

\begin{notation}[Pivotality events and their probabilities]
 We let $A^k(\Bar{u},m)$ denote the event that $(\Bar{u},m)$ is closed and pivotal for the size of the open cluster of $o=(\Bar{0},1)$ being at least $k$; let $\mathscr P^k(\Bar{u},m)$ denote the probability of this event. To be more precise: Let $\lambda>0$, $L,n,k\in\N$ and $(\Bar{u},m)\in V_n^L$, then 
\begin{align*}
\mathscr P^k(\Bar{u},m)= P\Big(A^k(\Bar{u},m)\Big),
\end{align*} where
\begin{align*}
A^k(\Bar{u},m)= \{(\Bar{u},m)\text{ is closed and pivotal for the event }\vert \mathcal{C}(o)\vert\geq k \}.
\end{align*}
Recall that we call a vertex $v$ pivotal for an increasing event $A$ if $A$ occurs when $v$ is open but not when it is closed, with all other states unchanged.
\end{notation}
\begin{fact}[Identity for the influences of vertices]
Let $\lambda>0$, $L,n,k\in\N$ and $(\Bar{u},m)\in V_n^L$. Then, using the notation above, we have
\begin{align}\label{id:infl1}
\nonumber&\textsc{Inf}_{(\Bar{u},m)}^P\left(\mathds{1}_{\{\vert \mathcal{C}(o)\vert\geq k\}}\right)
\\=& 2\Big(1-\e^{-\lambda 2^{-nd}\Pi(m,n)}\Big) \mathscr P^k(\Bar{u},m),
\end{align}
since for any vertex $v\in V_n^L$ the pivotality of $v$ is independent from its closedness and
\begin{align*}
\nonumber &\textsc{Inf}_{v}^P\left(\mathds{1}_{\{\vert \mathcal{C}(o)\vert\geq k\}}\right)\\=&2P\left(v\text{ is open}\right)P\left(v\text{ is closed}\right)P\left(v\text{ is pivotal for }\vert \mathcal{C}(o)\vert\geq k\right),
\end{align*}
compare also with the first statement in Remark~\ref{rem:russi_cov-pif}.
\end{fact}

Define the magnetization $M_{\tilde\gamma}^{L,n}$ by
\[
M_{\tilde\gamma}^{L,n}(x)
:=
E_\lambda^{L,n}
\left[1-\e^{-\tilde\gamma\vert\mathcal C(x)\vert}\right],\quad x\in V_n^L.
\]
The exploration scheme from the preceding section (cf. Fact~\ref{fact:revealment_bound}) gives
\[
\delta_{\tilde\gamma}^{L,n}(x)
\leq M_{\tilde\gamma}^{L,n}(x),
\qquad
\delta_{\tilde\gamma}^{L,n}(o)
=
M_{\tilde\gamma}^{L,n}(o)
\leq M_{\tilde\gamma}^{2L,n}(o).
\]
Before we can apply our variant of Russo's formula from Corollary~\ref{derivative_cluster_tail_discrete} to the expression on the left side of~\eqref{est:OSSSmain}, we still need some additional input. The following assumption is essential for the proof of Proposition~\ref{thm:4main}, the main result of this section.

\begin{assumption}\label{mainass}
The magnetizations and pivotality probabilities satisfy the following comparisons:

\begin{enumerate}

\item[a)]\label{mainass_a}
For every $\lambda>0$, there exists a measurable function
$\mathfrak R_\lambda:[1,\infty)\to[1,\infty)$ such that
\begin{align}\label{est:mainass_a}
M_{\tilde\gamma}^{L,n}(\Bar{u},m)
\leq
\mathfrak R_\lambda(m) M_{\tilde\gamma}^{2L,n}(o),
\quad
(\Bar{u},m)\in V_n^L,\quad \tilde\gamma>0.
\end{align}

\item[b)]\label{mainass_b}
For every $\lambda>0$, there exists a measurable function
$\mathfrak r_\lambda:[1,\infty)\to[1,\infty)$ such that
\begin{align}\label{est:mainass_b}
\sum_{\Bar{u}\in U_n^L}\mathscr P^k(\Bar{u},m)
\leq
\mathfrak r_\lambda(m)
\sum_{b\in M_n}\Pi(b,n)
\sum_{\Bar{u}\in U_n^L}\mathscr P^k(\Bar{u},b),
\quad
m\in M_n,\quad \Pi(m,n)>0.
\end{align}

\end{enumerate}

The comparison functions $\mathfrak R_\lambda$, $\mathfrak r_\lambda$ do not depend on $L,n,k$ or $\tilde\gamma$.
For every compact interval $I\subset(0,\infty)$, there exist
$L_0(I)\in\N$ and, for every $k\in\N$, $n_0(I,k)\in\N$ such that
both~\eqref{est:mainass_a},~\eqref{est:mainass_b} hold whenever
\[
\lambda\in I,\qquad L\geq L_0(I),\qquad n\geq n_0(I,k).
\]
Moreover, for every such interval $I$, we assume
\[
\sup_{\lambda\in I}\sup_{n\in\N}
\sum_{m\in M_n}
\Pi(m,n)\mathfrak r_\lambda(m)\mathfrak R_\lambda(m)
<\infty.
\]
\end{assumption}

\begin{remark}
In the non-weighted case, where we take $M_n=\{1\}$ and $\Pi(1,n)=1$,
the pivotality comparison in Assumption~\ref{mainass}\textup{(b)}
holds trivially with $\mathfrak r_\lambda=1$.
The magnetization comparison in Assumption~\ref{mainass}\textup{(a)} with $\mathfrak R_\lambda=2$ is due to a simple spatial comparison with the larger window.
The verification of these assumptions for the weighted models,
together with bounds on $\mathfrak r_\lambda$ and
$\mathfrak R_\lambda$, is deferred to Section~\ref{sect:Pf_mainass}.
\end{remark}

Next, let us introduce useful notation to adopt from now on throughout the paper.

\begin{notation}(The relation $\lesssim,\, \gtrsim$)
We will denote inequalities up to positive multiplicative constants by $\lesssim,\, \gtrsim$ in the obvious fashion. Unless otherwise specified, these constants will not depend on any parameter (apart from model assumptions like the dimension $d$), in particular the constants do not depend on the intensity $\lambda>0$ or the approximation parameters $L,n\in\N$.    
\end{notation}

Finally, we are ready to state the following proposition which is the main result of this section:
\begin{proposition}\label{thm:4main}
Let $\lambda>0$ and suppose that the estimates from
Assumption~\ref{mainass} hold for some functions
$\mathfrak r_\lambda$ and $\mathfrak R_\lambda$.
Then, for every fixed $k\in\N$, for $L,n$ in the range specified
by Assumption~\ref{mainass}, and for every $\tilde\gamma>0$, we have
\begin{align*}
&
\sum_{(\Bar{u},m)\in V_n^L}
\delta_{\tilde\gamma}(\Bar{u},m)
\textsc{Inf}_{(\Bar{u},m)}^P
\left(\mathds{1}_{\{\vert\mathcal C(o)\vert\geq k\}}\right)
\\\lesssim\ &
\lambda M_{\tilde\gamma}^{2L,n}(o)
\bigg(
\sum_{m\in M_n}
\Pi(m,n)\mathfrak r_\lambda(m)\mathfrak R_\lambda(m)
\bigg)
\frac{\dd}{\dd\lambda}\theta_\lambda^{L,n}(k).
\end{align*}
\end{proposition}

\begin{proof}
Fix $\lambda>0$ and $k\in\N$.
First, we use the revealment bound
$\delta_{\tilde\gamma}(\Bar{u},m)
\leq M_{\tilde\gamma}^{L,n}(\Bar{u},m)$
and Assumption~\ref{mainass}\textup{(a)} to obtain
\begin{align*}
&
\sum_{(\Bar{u},m)\in V_n^L}
\delta_{\tilde\gamma}(\Bar{u},m)
\textsc{Inf}_{(\Bar{u},m)}^P
\left(\mathds{1}_{\{\vert\mathcal C(o)\vert\geq k\}}\right)
\\\leq\ &
M_{\tilde\gamma}^{2L,n}(o)
\sum_{m\in M_n}\mathfrak R_\lambda(m)
\sum_{\Bar{u}\in U_n^L}
\textsc{Inf}_{(\Bar{u},m)}^P
\left(\mathds{1}_{\{\vert\mathcal C(o)\vert\geq k\}}\right)
\\=\ &
2M_{\tilde\gamma}^{2L,n}(o)
\sum_{m\in M_n}\mathfrak R_\lambda(m)
\Big(1-\e^{-\lambda2^{-nd}\Pi(m,n)}\Big)
\sum_{\Bar{u}\in U_n^L}\mathscr P^k(\Bar{u},m),
\end{align*}
where the last equality follows from~\eqref{id:infl1}.

Assumption~\ref{mainass}\textup{(b)} yields
\begin{align}\label{eq:est1}
\nonumber&
\sum_{m\in M_n}\mathfrak R_\lambda(m)
\Big(1-\e^{-\lambda2^{-nd}\Pi(m,n)}\Big)
\sum_{\Bar{u}\in U_n^L}\mathscr P^k(\Bar{u},m)
\\\leq\ &
\bigg(
\sum_{m\in M_n}
\mathfrak r_\lambda(m)\mathfrak R_\lambda(m)
\Big(1-\e^{-\lambda2^{-nd}\Pi(m,n)}\Big)
\bigg)
\sum_{b\in M_n}\Pi(b,n)
\sum_{\Bar{u}\in U_n^L}\mathscr P^k(\Bar{u},b).
\end{align}
Terms with $\Pi(m,n)=0$ vanish. Since
\(
1-\e^{-\lambda2^{-nd}\Pi(m,n)}
\leq \lambda2^{-nd}\Pi(m,n),
\)
we conclude that
\begin{align*}
&
\sum_{(\Bar{u},m)\in V_n^L}
\delta_{\tilde\gamma}(\Bar{u},m)
\textsc{Inf}_{(\Bar{u},m)}^P
\left(\mathds{1}_{\{\vert\mathcal C(o)\vert\geq k\}}\right)
\\\leq\ &
2\lambda M_{\tilde\gamma}^{2L,n}(o)
\bigg(
\sum_{m\in M_n}
\Pi(m,n)\mathfrak r_\lambda(m)\mathfrak R_\lambda(m)
\bigg)
\sum_{\Bar{u}\in U_n^L}\sum_{b\in M_n}
2^{-nd}\Pi(b,n)\mathscr P^k(\Bar{u},b).
\end{align*}

The claim follows from Russo's formula
(Corollary~\ref{derivative_cluster_tail_discrete}
and Remark~\ref{rem:russi_cov-pif}), which gives
\[
\sum_{\Bar{u}\in U_n^L}\sum_{b\in M_n}
2^{-nd}\Pi(b,n)\mathscr P^k(\Bar{u},b)
=
\frac{\dd}{\dd\lambda}\theta_\lambda^{L,n}(k).
\]
\end{proof}

Following Hutchcroft's strategy in~\cite{Hut20}, we combine the obtained
result with~\eqref{est:OSSSmain} to obtain the following inequalities:

\begin{corollary}\label{cor:3.1}
\label{cor:4main}
Let $\lambda>0$ and suppose that the estimates from
Assumption~\ref{mainass} hold for some functions
$\mathfrak r_\lambda$ and $\mathfrak R_\lambda$.
Then we have
\begin{align*}
&
\bigg(
\frac{1-\e^{-\tilde\gamma k}}
{M_{\tilde\gamma}^{2L,n}(o)}
-1
\bigg)\theta_\lambda^{L,n}(k)
-
\frac{1}{M_{\tilde\gamma}^{2L,n}(o)}
\sum_{e\in E_n^L}
\delta_{\tilde\gamma}(e)
\textsc{Inf}_e^P
\big(\mathds{1}_{\{\vert\mathcal C(o)\vert\geq k\}}\big)
\\\lesssim\ &
\bigg(
\lambda\sum_{m\in M_n}
\Pi(m,n)\mathfrak r_\lambda(m)\mathfrak R_\lambda(m)
\bigg)
\frac{\dd}{\dd\lambda}\theta_\lambda^{L,n}(k)
\end{align*}
for every fixed $k\geq1$, for $L,n$ in the range specified by
Assumption~\ref{mainass}, and for every $\tilde\gamma>0$.
\end{corollary}

\begin{proof}
Choose $\gamma=k\tilde\gamma$ in Proposition~\ref{prop:OSSSmain}.
Since
\[
E_\lambda^{L,n}
\left[1-\e^{-\tilde\gamma\vert\mathcal C(o)\vert}\right]
=
M_{\tilde\gamma}^{L,n}(o)
\leq M_{\tilde\gamma}^{2L,n}(o),
\]
the inequality follows from Proposition~\ref{thm:4main}
after division by $M_{\tilde\gamma}^{2L,n}(o)>0$.
\end{proof}

\begin{corollary}\label{cor:3.2}
Let $\lambda>0$ and suppose that the estimates from
Assumption~\ref{mainass} hold for some functions
$\mathfrak r_\lambda$ and $\mathfrak R_\lambda$.
Then we have
\begin{align}\label{est: 1cor3.10}
\nonumber&
\bigg(
\frac{k}{E_\lambda^{2L,n}[\vert\mathcal C(o)\vert]}
-1
\bigg)\theta_\lambda^{L,n}(k)
-
\sup_{\tilde\gamma>0}
\frac{1}{M_{\tilde\gamma}^{2L,n}(o)}
\sum_{e\in E_n^L}
\delta_{\tilde\gamma}(e)
\textsc{Inf}_e^P
\big(\mathds{1}_{\{\vert\mathcal C(o)\vert\geq k\}}\big)
\\\lesssim &
\left(
\lambda\sum_{m\in M_n}
\Pi(m,n)\mathfrak r_\lambda(m)\mathfrak R_\lambda(m)
\right)
\frac{\dd}{\dd\lambda}\theta_\lambda^{L,n}(k).
\end{align}
Moreover, we have
\begin{align}\label{est: 2cor3.10}
\nonumber&
\left(
\frac{k(1-\e^{-1})}
{\sum_{i=1}^{k}\theta_\lambda^{2L,n}(i)}
-1
\right)\theta_\lambda^{L,n}(k)
-
\sup_{\tilde\gamma>0}
\frac{1}{M_{\tilde\gamma}^{2L,n}(o)}
\sum_{e\in E_n^L}
\delta_{\tilde\gamma}(e)
\textsc{Inf}_e^P
\left(\mathds{1}_{\{\vert\mathcal C(o)\vert\geq k\}}\right)
\\\lesssim &
\left(
\lambda\sum_{m\in M_n}
\Pi(m,n)\mathfrak r_\lambda(m)\mathfrak R_\lambda(m)
\right)
\frac{\dd}{\dd\lambda}\theta_\lambda^{L,n}(k).
\end{align}
Both inequalities hold for every fixed $k\in\N$ and for $L,n$
in the range specified by Assumption~\ref{mainass}.
\end{corollary}

\begin{remark}
These differential inequalities correspond to the one
from~\cite[Proposition 1.4]{Hut20}.
Our estimates contain an additional sum of edge influences and
use reference cluster quantities in the larger window.
The edge contribution vanishes in the continuum limit, as established
in Appendix~\ref{AIII}; the reference quantities are subsequently
handled by taking the infinite-volume limit.
\end{remark}

\begin{proof}
We set $\tilde\gamma=\gamma/k$.
Using the estimate
\[
M_{\gamma/k}^{2L,n}(o)
\leq
\frac{\gamma}{k}
\sum_{i=1}^{\left\lceil k/\gamma\right\rceil}
\theta_\lambda^{2L,n}(i)
\]
and setting $\gamma=1$, inequality~\eqref{est: 2cor3.10}
follows directly from Corollary~\ref{cor:3.1}, after bounding
the normalized edge contribution by its supremum over
$\tilde\gamma>0$.

On the other hand, inequality~\eqref{est: 1cor3.10}
follows from Corollary~\ref{cor:3.1} by taking
$\gamma\downarrow0$, using
\[
M_{\gamma/k}^{2L,n}(o)
\leq
\frac{\gamma}{k}
E_\lambda^{2L,n}[\vert\mathcal C(o)\vert]
\]
and
\[
\liminf_{\gamma\downarrow0}
\frac{1-\e^{-\gamma}}{M_{\gamma/k}^{2L,n}(o)}
\geq
\lim_{\gamma\downarrow0}
\frac{1-\e^{-\gamma}}{\gamma}
\frac{k}{E_\lambda^{2L,n}[\vert\mathcal C(o)\vert]}
=
\frac{k}{E_\lambda^{2L,n}[\vert\mathcal C(o)\vert]}.
\]
The expectation is finite because the expected number of open
vertices in the finite spatial window is finite.
Again, we bound the normalized edge contribution by its supremum
before taking the limit.
\end{proof}

\section{Verifying the technical Assumptions~\ref{mainass}}\label{sect:Pf_mainass}

In this section, we show that our continuum assumptions from Theorem~\ref{thm:Susceptibility Mean-Field Bound} imply
Assumption~\ref{mainass}, allowing us to apply the general results of
Section~\ref{sect:Inf_vol_lim} and finalize the proof of Theorem~\ref{thm:Susceptibility Mean-Field Bound}. We already observed that these technical assumptions are satisfied in the non-weighted case but the weighted case remains to be considered. Naturally, the growth rates of the bounds $\mathfrak r_\lambda$ and
$\mathfrak R_\lambda$ determine the required moment assumptions.
Using the estimates obtained here, which are at most exponential in
$1+\ell(m)^d$, we obtain the required moment bound from
Assumption~\ref{ass:weighted}. Any improvement in the order of these
bounds would yield a corresponding weakening of the moment condition.

\subsection{Preparations: Notation and geometric observations}

Before we proceed with the proof, we still need to introduce some additional notation and make some simple yet important observations.

Consider an adjacency function $\varphi$ and a weight distribution $\pi$ satisfying Assumption~\ref{ass:weighted}. Throughout this section, let $\rho$, $\sigma$ and $\ell$ be as in that assumption, and recall
\[
  p_0=\varphi(\rho;1,1)>0,
\]
so that 
\[
  \varphi(r;a,b)\geq p_0,
  \quad 0\leq r\leq\rho,\quad a,b\geq1,
\]
by monotonicity.

We say that $Z$ is a cylinder set if $Z=C\times M_n$ for some
$C\subset U_n^L$. For $r>0$ and $\Bar{u}\in U_n^L$, let
\[
  B_r(\Bar{u})
  :=\prod_{i=1}^d[\Bar{u}_i-r,\Bar{u}_i+r)
\]
be the spatial half-open cube of side length $2r$ centered at
$\Bar{u}$ and set
\[Z_r(\Bar{u}):=(B_r(\Bar{u})\cap U^L_n)\times M_n.\]

We partition $[-L,L]^d$ into such cubes of a fixed dyadic side length $s=2^{-n_*}$, $n_*< n$, using half-open
intervals except at the upper boundary, where the endpoint is included.
Let
\[
  \mathcal U_s^L
  :=\left\{-L+\left(j+\frac12\right)s:
  j=0,\ldots,\frac{2L}{s}-1\right\}^{d}
\]
be the set of their centers. These form a finite window of a
hypercubic lattice of mesh size $s$, coarser than the original
lattice of mesh size $2^{-n}$. Since $n> n_*$, we have
$\mathcal U_s^L\subset U_n^L$.
The corresponding cylinders are $Z_{s/2}(\bar z)$,
$\bar z\in\mathcal U_s^L$, with the upper-boundary convention
stated above. Two partition cubes share a face precisely when
their centers differ by $s$ in one coordinate and agree in all others.

Choose $n_*\in\N$ large enough so that
\[
  s\sqrt{d+3}<\rho,
  \qquad
  s\sqrt d<\sigma.
\]
\begin{itemize}
    \item The first condition on $s$ guarantees that two vertices whose cylinders
share a spatial face are connected with probability at least $p_0$.
    \item Together with the replacement geometry condition, the second condition
on $s$ ensures that every distant neighbor is deterministically
connected to all open vertices of a suitably located cylinder:

Take two vertices $x=(\bar u, a)\neq y=(\bar v,b)$ such that
$\varphi(r;a,b)>0$ for
$r:=|\bar v-\bar u|>\ell(a)$. Consider the replacement point
\begin{align}\label{def:replacement_point}
     z:=\bar u+\frac{\ell(a)}{r}(\bar v-\bar u).
\end{align}
The point $z$ lies in one of the selected cubes,
say $B_q$. For every $\bar w\in B_q\cap U_n^L$, we have
\[
  |\bar w-\bar v|
  \leq r-\ell(a)+s\sqrt d
  <r-\ell(a)+\sigma.
\]
Consequently, the replacement geometry condition and monotonicity yield
\[
  \varphi(|\bar w-\bar v|;c,b)
  \geq
  \varphi(r-\ell(a)+\sigma;1,b)
  =1,
  \quad c\geq1.
\]
Hence, all edges between $y$ and vertices of $Z_q=(B_q\cap U_n^L)\times M_n$ are open,
almost surely.
\end{itemize}

 For $x=(\bar u,a)\in V_n^L$, let $B_1,\ldots,B_{M(a)}$ be the
partition cubes intersecting \([-L,L]^d\cap
  \bigl(\bar u+[-\ell(a),\ell(a)]^d\bigr)\)
and set
\[
  Z_q:=(B_q\cap U_n^L)\times M_n,
  \quad 1\leq q\leq M(a),
  \qquad
  D(\bar u,a):=\bigcup_{q=1}^{M(a)} Z_q.
\]
The centers of the selected cubes form a block of the hypercubic lattice $\mathcal U_s^L$.
Moreover,
\[
  M(a)\leq \left(\frac{2\ell(a)}{s}+2\right)^d
  \leq C_0\bigl(1+\ell(a)^d\bigr),
\]
where the constant $C_0>0$ depends on $d$ and $s$ only.

To connect all open neighbors of a closed vertex $x=(\bar u,a)$,
it suffices to ensure that every cylinder $Z_q$, $1\leq q\leq M(a)$,
forming $D(\bar u,a)$ contains an open vertex and that all open
vertices in $D(\bar u,a)$ are connected.
Indeed, neighbors at distance at most $\ell(a)$ already belong to
$D(\bar u,a)$, while every open neighbor $y$ outside $D(\bar u,a)$
has a deterministically open edge to an open vertex in the
cylinder containing its replacement point as defined in~\eqref{def:replacement_point}.

Both comparisons below rely on the cylinders constructed above.
For magnetization, it suffices to connect a vertex of minimal mark
to a green component. For pivotality, we must reconnect all
components adjacent to the closed vertex under consideration.
We therefore modify the configuration cylinder by cylinder,
recording any step at which the cluster-size event occurs and
bounding its contribution through vertex pivotality.

\subsection{Comparing magnetizations}

First, we establish the magnetization comparison required in
Assumption~\ref{mainass}\textup{(a)}. We begin with vertices of
minimal weight.

\begin{lemma}\label{lem:aux_rev}
For every $\lambda,\tilde\gamma>0$, $L,n\in\N$ and
$\bar u\in U_n^L$, we have
\begin{equation}\label{est:aux_rev1}
  M_{\tilde\gamma}^{L,n}(\bar u,1)
  \leq 2M_{\tilde\gamma}^{2L,n}(o).
\end{equation}
\end{lemma}

\begin{proof}
Write $v=(\bar u,1)$. If $v=o$, the claim follows simply by enlarging
the spatial window. From here, we assume $v\neq o$.

Let $\mathcal C^{(-o)}(v)$ denote the cluster of $v$ after forcing 
$o$ closed, and let $\mathcal C^{(-v)}(o)$ denote the cluster of $o$
after forcing $v$ closed. Assume that $v$ is connected to a green vertex, i.e.,
there is a simple path of open edges and open vertices (other than possibly $v$ itself) connecting them. Any such simple path
either avoids $o$, or its segment from $o$ to the green vertex
avoids $v$. Therefore,
\begin{align*}
  M_{\tilde\gamma}^{L,n}(v)
  \leq{}&
  E_\lambda^{L,n}\left[
    1-\exp\{-\tilde\gamma|\mathcal C^{(-o)}(v)|\}
  \right]
  \\
  &+
  E_\lambda^{L,n}\left[
    1-\exp\{-\tilde\gamma|\mathcal C^{(-v)}(o)|\}
  \right],
\end{align*}
cf. identity from Fact~\ref{fact:revealment_bound}.

By our cluster convention, the state of $v$ does not affect $\mathcal C^{(-o)}(v)$. Translating by $-\bar u$ moves $v$
to the origin, the spatial window remains inside $[-2L,2L]^d$. The original root is now at $(-\bar u,1)$ and remains forced
closed. Allowing it to open freely and enlarging the window to $[-2L,2L]^d$ can only increase the cluster. Thus the first expectation is at most $M_{\tilde\gamma}^{2L,n}(o)$.

For the second expectation, forcing $v$ closed can only decrease the cluster of $o$. Hence, it is at most $M_{\tilde\gamma}^{L,n}(o)\leq M_{\tilde\gamma}^{2L,n}(o)$. This proves~\eqref{est:aux_rev1}.
\end{proof}

Next, we establish the desired comparison between magnetizations for different weights.

\begin{theorem}\label{thm:ta_rev}
There exists a constant $C\geq1$ such that
Assumption~\ref{mainass}\textup{(a)} holds for the choice
\[
  \mathfrak R_\lambda(a)=C\bigl(1+\ell(a)^d\bigr).
\]
The constant $C$ is independent of $\lambda$, $L$, $n$ and
$\tilde\gamma$, and the comparison holds for sufficiently large $n$.
\end{theorem}

\begin{proof}
Fix $x=(\bar u,a)$ and let $Z_1,\ldots,Z_{M(a)}$ be the cylinders
selected in the preparations. Denote the centers of their spatial bases by
$\bar z_1,\ldots,\bar z_{M(a)}$ and set $v_q=(\bar z_q,1)$ for $1\leq q\leq M(a)$.
Write $\mu=P_\lambda^{L,n}\otimes Q_{\tilde\gamma}^{L,n}$ as in the previous section.

Consider the set $S$ of potential neighbors of $x$, i.e., vertices $y=(\bar v,b)$ that connect to $x$ with positive probability $\varphi(\vert \bar v - \bar u \vert; a,b )>0$. Assign to each cylinder $Z_q$ every potential neighbor $y\in S$ such that either 
\begin{itemize}
    \item $y\in Z_q$ or,
    \item if $y \notin D(\bar u,a)$, the associated replacement point $z$ (given by~\eqref{def:replacement_point}) lies in $B_q$.
\end{itemize}
This defines a partition of $S$ into  deterministic sets $S_1,\ldots,S_{M(a)}$. By the simple geometric observation from the preparations above,
\[
  \varphi(v_q,y)=\varphi(|\bar z_q-\bar v|;1,b)\geq p_0,
  \quad y=(\bar v,b)\in S_q\setminus\{v_q\}.
\]

Now, let $H_q$ denote the event that some open vertex in $S_q$ connects to an open green vertex, and let $G_q$ denote the event that $v_q$ (open or closed) connects to an open green vertex or is itself green. We claim that
\[
  \mu(H_q)\leq p_0^{-1}\mu(G_q)
  =p_0^{-1}M_{\tilde\gamma}^{L,n}(v_q).
\]

To prove this, let $H_q^0$ denote the event $H_q$ in the graph
with $v_q$ forced closed. On $H_q^0$, choose a vertex
$y\in S_q$ that is open and connects to an open green vertex, using only the configuration
with $v_q$ forced closed, by some deterministic rule. The edge between $v_q$ and $y$ is independent
of that configuration. As we just observed, it is open with probability at least $p_0$.
This consideration yields
\[
  \mu(G_q\cap H_q^0)\geq p_0\mu(H_q^0).
\]
On $H_q\setminus H_q^0$, there exists a path of open vertices and edges connecting an open vertex in $S_q$ to an open green vertex and such paths necessarily go through $v_q$, so $G_q$ occurs. Since $H_q^0\subset H_q$, we obtain
\begin{align*}
  \mu(H_q)=\mu(H_q^0)+\mu(H_q\setminus H_q^0)
  \leq 
  p_0^{-1}\mu(G_q\cap H_q^0)+\mu(G_q\setminus H^0_q)
  \leq 
  p_0^{-1}\mu(G_q),
\end{align*}
as claimed.

Trivially, if $\mathcal C(x)$ contains a green vertex, then either $x$
itself is green, or there exists a simple path of open edges from $x$ to an open green vertex, with all vertices other than possibly \(x\) open, that
first visits an open vertex in some $S_q$. The remainder of this path avoids $x$ and connects an open
vertex in $S_q$ to an open green vertex, so that $H_q$ occurs. Therefore,
\begin{align*}
  M_{\tilde\gamma}^{L,n}(x)
  \leq 1-e^{-\tilde\gamma}
       +\sum_{q=1}^{M(a)}\mu(H_q)
  \leq 1-e^{-\tilde\gamma}
       +p_0^{-1}\sum_{q=1}^{M(a)}
          M_{\tilde\gamma}^{L,n}(v_q)
  \leq
       \left(1+\frac{2M(a)}{p_0}\right)
       M_{\tilde\gamma}^{2L,n}(o).
\end{align*}
The last inequality follows from Lemma~\ref{lem:aux_rev} and
$1-e^{-\tilde\gamma}\leq M_{\tilde\gamma}^{2L,n}(o)$. Using the geometric estimate $M(a)\leq C_0(1+\ell(a)^d)$ observed in preparations yields the claimed bound, with $C=1+2C_0/p_0$.
\end{proof}

\subsection{Comparing pivotality probabilities}
Next, we establish the comparison required in Assumption~\ref{mainass}\textup{(b)}. We again use the cylinders
selected in the preparations. This time, we connect all components adjacent to a closed pivotal vertex by modifying the configuration
cylinder by cylinder, and estimate the probability that the cluster-size event occurs during one of these modifications.

\paragraph{Outline of the proof.} Consider a closed pivotal vertex $x=(\bar u,a)$ and the cylinders forming $D(\bar u,a)$.
Their spatial bases form a connected block of the coarse lattice. Choose a spanning tree of this block.

We first modify the configuration within each cylinder so that it contains an open vertex and its open vertices are connected. We then proceed along the edges of the spanning tree, modifying the configuration in the corresponding pairs of neighboring cylinders so that their open vertices become connected. Each modification preserves all previously established connections. At the end, all open vertices in $D(\bar u,a)$ belong to one component, to which every original open neighbor of $x$ is attached by the geometric observations from the preparations. This reconnects the components that opening $x$ would join. A final insertion accounts for the possibly closed vertex $x$.

The probability that the construction changes $\{|\mathcal C(o)|\ge k\}$ from false to true is bounded below by a positive multiple of the probability that $x$ is closed and pivotal. We obtain a corresponding upper bound by showing that the step at which this event becomes true produces a pivotal vertex in one of the modified cylinders and bounding how much the modifications can increase pivotality probabilities relative to the original law.

\paragraph{Graph-theoretic observations} Before we proceed, we first introduce some additional terminology and two elementary graph-theoretic observations.
A vertex is called a \emph{cut vertex} if deleting it (and its incident edges) increases the number of connected components.
In our Bernoulli percolation setting, an open vertex $v$ is therefore a cut vertex precisely when it is pivotal for $\{u\leftrightarrow w\}$ for some other open vertices $u,w$: they are connected when $v$ is open but disconnected when
$v$ is closed. We call a graph \emph{doubly-connected} if it is connected and has no cut vertices. Notice that, under this convention, both a singleton and a graph consisting of two vertices joined by an edge are doubly-connected.

Now we record two elementary but essential observations in the following lemma.

\begin{lemma}\label{lem:aux_inf} Let $V$ be any finite non-empty set.
\begin{enumerate}
    \item[(a)] Fix a graph $G$ on $V$. Suppose $V$ can be partitioned into at most $q$ sets, each inducing a doubly-connected subgraph of $G$. Then $G$ has at most $q^2$ cut vertices and deleting any one vertex of $G$ leaves at most $q$ connected components.
    \item[(b)] For every $J\in\N$, there exists $c_0=c_0(J,p_0)>0$ such that the following holds: Let $G$ be the random graph on $V$, where every potential edge $\{u,v\}$ is drawn independently, with probability $p_{u,v}\geq p_0$. Then with probability at least $c_0$ deleting any set of at most $J$ vertices from $G$ (together with the incident edges) leaves a connected graph (if at least one vertex remains).
\end{enumerate}

\end{lemma}

\begin{proof}
In part $(a)$, $q$ being the component bound is trivial, since deleting a vertex leaves each nonempty remainder of a partition block connected, so at most $q$ components remain. For the bound on the number of cut vertices notice: For a partition of $V$, as above, let $H$ be some block with at least two vertices. For each cut vertex $v\in H$, consider the vertices in the original
component of $H$ that become disconnected from $H\setminus\{v\}$ after deleting $v$. These sets are nonempty and pairwise disjoint. Moreover, each of them contains an entire other block of the partition, since the graph on each block is connected by assumption. Thus, $H$ contains at most $q-1$ cut vertices. Singleton sets contribute at most one cut vertex each, giving the bound $q^2$.

For part $(b)$, it suffices to consider constant edge probability $p_0$ (since probability of any edge insertion is uniformly bounded from below by $p_0$ and inserting edges can only favor the connectivity event considered). In a graph on $N\geq 2$ vertices, the probability that some pair has at most $J$ common neighbors is bounded by
\[
  \binom N2
  \mathbb P\bigl(\operatorname{Bin}(N-2,p_0^2)\le J\bigr)
  \longrightarrow0.
\]
Otherwise, every surviving pair remains connected via a common neighbor after deleting at most $J$ vertices. This shows the bound for sufficiently large $N$. Naturally, for the remaining finitely many values of $N$, requiring all edges to be open gives a positive lower bound.

\end{proof}

\paragraph{Local modifications.}
Fix $\lambda>0$ and $L,n,k\in\N$, and write
\[
  P=P_\lambda^{L,n},
  \qquad
  f=\mathds{1}_{\{|\mathcal C(o)|\ge k\}}.
\]
For simpler bookkeeping, we record only the graph induced by the open vertices, whose law we also denote by $P$. When opening a closed vertex, we draw its incident edges independently with their original probabilities. This gives the same pivotality probabilities $\mathscr P^k$ as the original construction. The root remains fixed open throughout.

For any cylinder $Z$ associated to the partition of cubes with side-length $s$ constructed in the preparations, select a vertex $v\in Z\setminus\{o\}$ with probability
\begin{equation}\label{eq:pivotal_selection}
  q_v^Z
  :=
  \frac{p_\lambda(v)/(1-p_\lambda(v))}
       {\displaystyle
        \sum_{w\in Z\setminus\{o\}}
        p_\lambda(w)/(1-p_\lambda(w))},
\end{equation}
where we use $p_\lambda(v):=P(v\text{ is open})$. 
We choose these probabilities so that the factors $p_\lambda(v)/(1-p_\lambda(v))$ compensate for the cost of opening $v$, keeping the resulting factors in the comparison of pivotality probabilities bounded uniformly in $n$. Of course, other choices are possible, since we are only interested in such a bound.

Let $\mathcal Z$ consist of one cylinder or of two cylinders whose spatial bases share a face, and fix one cylinder $Z$ in $\mathcal Z$ for vertex selection. If the open vertices in $\mathcal Z$ already induce a nonempty connected graph, leave the configuration unchanged.
Otherwise, select $v$ according to~\eqref{eq:pivotal_selection} and consider the subgraph induced by the open vertices in $\mathcal Z$ with $v$ removed.

Draw the edges from $v$ to the remaining open vertices in $\mathcal Z$ with their original probabilities, conditioned on $v$ having an open edge to every component of that graph. This conditioning has positive probability by local positivity~\eqref{ass:LP} and the choice of the coarse side length $s$. If no vertices remain, the condition is void.
If $v$ is closed, open it and draw its edges to open vertices outside $\mathcal Z$ independently with their original probabilities. 
If $v$ is already open, keep those outside edges unchanged.

The resulting induced graph in $\mathcal Z$ is nonempty and connected. Moreover, every connection between previously open vertices is preserved: a path using a changed edge can replace that edge by a path through $\mathcal Z$.
In particular, $f$ cannot decrease during the modification.

Applying these modifications along a spanning tree of the selected cubes $B_1,\ldots,B_{M(a)}$ connects all open neighbors of an initially closed pivotal vertex $x$. Controlling their total probability cost yields the following comparison.

\begin{theorem}\label{thm:ta_inf}
Under Assumption~\ref{ass:weighted}, there exists $C_\lambda\ge1$, independent of $L,n,k$, such that Assumption~\ref{mainass}\textup{(b)} holds with
\[
  \mathfrak r_\lambda(a)
  =C_\lambda^{\,1+\ell(a)^d}.
\]
The function $\lambda\mapsto C_\lambda$ can be chosen non-increasing.
\end{theorem}

\begin{remark}[Moment condition is satisfied]
    Together with Theorem~\ref{thm:ta_rev}, this also verifies
the moment condition in Assumption~\ref{mainass}.
Indeed, for every compact interval $I\subset(0,\infty)$,
\[
  \sup_{\lambda\in I}\sup_n \sum_{m\in M_n}
  \Pi(m,n)\mathfrak r_\lambda(m)\mathfrak R_\lambda(m)
  \le
  \int_{[1,\infty)}
  e^{C_I(1+\ell(a)^d)}\,\pi(\dd a)
  <\infty
\]
by~\eqref{ass:MR} and monotonicity of $\ell$.

\end{remark}

\begin{proof}

Fix $x=(\bar u,a)\ne o$ and write $M=M(a)$. Let $Z_1,\ldots,Z_M$ be the
selected cylinders forming $D(x)=D(\bar u,a)$.

~\\\emph{The configurations used in the construction.}
Set $J=2d+2$. Let $\mathcal G_x$ be the event that 
\begin{itemize}
    \item $x$ is closed and,
    \item each of the induced graphs on
\[
  \mathcal C(o)\cap Z_j
  \quad\text{and}\quad
  \bigl(\{\text{open vertices}\}\setminus \mathcal C(o)\bigr)\cap Z_j,
  \qquad 1\le j\le M,
\]
stays connected after deleting any set of at most $J$ vertices, whenever vertices remain.
\end{itemize}

Recall that we denote by $A^k(x)$ the event that $x$ is $f$-pivotal and closed. We claim that
\begin{equation}\label{eq:good_pivotal_configurations}
  P\bigl(A^k(x)\cap\mathcal G_x\bigr)
  \ge c_0^{2M}\mathscr P^k(x),
\end{equation}
where $c_0=c_0(J,p_0)$ is given by Lemma~\ref{lem:aux_inf}\textup{(b)}.

We condition on the open vertex set, with $x$ closed, and on the exact vertex set $\mathcal C(o)$. All edges between
$\mathcal C(o)$ and the open vertices outside $\mathcal C(o)$ are closed. The internal edges of $\mathcal C(o)$ have their
product law conditioned on connectivity. All other edges retain their original independent laws, including the edges incident to $x$.

Under this conditioning, the internal edges of $\mathcal C(o)$ are independent of the remaining edges. Moreover, their particular configuration does not affect $A^k(x)$.
Thus,
\[
\begin{aligned}
&P\bigl(\mathcal G_x\cap A^k(x)
  \mid \text{open vertex set},\mathcal C(o)\bigr)
\\
&\quad=
P\bigl(\text{inside connectivity properties hold}
  \mid \text{open vertex set},\mathcal C(o)\bigr)
\\
&\qquad\times
P\bigl(A^k(x)\cap
  \{\text{outside connectivity properties hold}\}
  \mid \text{open vertex set},\mathcal C(o)\bigr).
\end{aligned}
\]

The first factor concerns the event that every graph induced by $\mathcal C(o)\cap Z_j$ remains connected after deleting
at most $J$ vertices, whenever vertices remain. This event is increasing in the internal edges, as is connectivity of $\mathcal C(o)$. Harris--FKG
(see~\cite[Section~2.2]{Gri99}) therefore bounds its conditional probability below by its probability under the product law
on this fixed vertex set. The latter is at least $c_0^M$: Lemma~\ref{lem:aux_inf}\textup{(b)} gives a factor $c_0$
for each cylinder, and distinct cylinders involve disjoint edge sets.

Analogously, Lemma~\ref{lem:aux_inf}\textup{(b)} and independence give conditional probability at least $c_0^M$ for the
corresponding connectivity properties of the open vertices outside $\mathcal C(o)$. If $|\mathcal C(o)|\ge k$, the pivotal
event is empty. Otherwise, with the vertex sets fixed, both these outside connectivity properties and
$A^k(x)$ are increasing in the outside edges and the edges incident to $x$: adding such edges can only enlarge
the cluster obtained by opening $x$. These edges retain their product law, so Harris--FKG bounds the second factor
below by the product of the two conditional probabilities.

Combining these bounds gives
\begin{align*}
    P\bigl(\mathcal G_x\cap A^k(x)
  \mid \text{open vertex set},\mathcal C(o)\bigr)
\ge
c_0^{2M}
P\bigl(A^k(x)
  \mid \text{open vertex set},\mathcal C(o)\bigr).
\end{align*}
Averaging over the conditioning proves~\eqref{eq:good_pivotal_configurations}.

~\\\emph{Connecting the selected cylinders.}
In the first stage, starting from a graph $G_0$ with law $P$, apply the local modification once to every $Z_j$. In the second stage, apply it to the union of the two cylinders corresponding to each edge of a fixed spanning tree of their face-adjacency graph. For each such pair $\mathcal Z$, choose its selection cylinder $Z$ by some deterministic rule. This procedure consists of $T=2M-1$ modifications in total, and every cylinder is used at most $2d+1$ times.

The first stage makes every cylinder occupied and all open vertices within it interconnected. Each subsequent modification in the second stage preserves all old connections and joins the open vertices of its two cylinders. The spanning tree therefore connects all open vertices in $D(x)$.

On $\mathcal G_x$, the graphs induced by the open vertices in the modification regions each satisfy the partition condition of
Lemma~\ref{lem:aux_inf}\textup{(a)}, with the number of groups bounded solely in terms of $d$. Indeed, remove the vertices selected for modifications from each of the two original groups in a cylinder: the initially open vertices belonging to the original root cluster and
those outside it. At most $2d+1$ vertices have been selected there. Upon removal of the selected vertices, each remaining group is doubly-connected by the definition of $\mathcal G_x$, and its internal edges have not been modified.

Thus, in a region $\mathcal Z$ consisting of one or two cylinders, the open vertices can be partitioned into at most four such remaining groups and at most $2(2d+1)$ selected vertices, regarded as singleton groups. Lemma~\ref{lem:aux_inf}\textup{(a)} bounds the number of components after removing any single vertex by $q:=4+2(2d+1)$ and the number of cut vertices in the modified subgraph by $q^2$.

On $A^k(x)\cap\mathcal G_x$, the original clusters satisfy
\[
  \mathcal C(o)\subseteq \mathcal C(x)\setminus \{x\},\qquad \vert \mathcal C(x) \vert \geq k.
\]
Every original open neighbor of $x$ in $D(\bar u,a)$ now
belongs to the connected collection we just constructed. Every remaining neighbor has a deterministically open edge to an occupied
replacement cylinder by~\eqref{ass:RG}. Since all old connections are preserved, the root cluster after the $T=2M-1$ modifications contains $\mathcal C(x)\setminus \{x\}$ and all open vertices in $D(\bar u,a)$.

~\\\emph{The final insertion.}
Fix one selected cylinder $Z_*$. If the root cluster has fewer than $k$ vertices, choose a vertex according to $q^{Z_*}$ and, if it is closed, open it (with incident edges drawn independently according to their original probabilities). If the root cluster already has at least $k$ vertices, leave the graph unchanged. Denote the resulting graph by $G_{\mathrm{fin}}^x$. To simplify notation, we will also use $P$ for probabilities involving the original configuration and the random choices made during the modifications.

We have seen that, on $A^k(x)\cap\mathcal G_x$, the modified root cluster before this insertion contains at least $k-1$ vertices and all open vertices in $Z_*$. Thus, if we do the insertion, we have exactly $k-1$ vertices in the root cluster and at most $k-1$ open vertices in $Z_*$. The selection probabilities satisfy
\(
  q_v^{Z_*}\le C2^{-nd}\Pi(b,n)\text{ for } v=(\bar z,b)
\)
whenever $\lambda2^{-nd}\leq 1$; see \eqref{est:pivotal_selection_estimates} below.
Hence, for sufficiently large $n$, the probability of selecting a closed vertex is at least
\[
  1-C(k-1)2^{-nd}\ge\frac12.
\]
The inserted vertex connects to an existing open vertex
of $Z_*$ with probability at least $p_0$, bringing the
root-cluster size to at least $k$. Therefore, using that $f(G_0)=0$ holds on $A^k(x)$, 
\begin{align}\label{eq:pivotality_lower_gain}
   P \bigl( \mathcal G_x,\ f(G_0)=0,\ f(G_{\mathrm{fin}}^x)=1 \bigr)
  \geq& \nonumber
    P \bigl( \mathcal G_x\cap A^k(x),\ f(G_{\mathrm{fin}}^x)=1 \bigr)
  \\\geq&\nonumber
  \frac{p_0}{2}P(A^k(x)\cap\mathcal G_x)
 \\\geq& 
  \frac{p_0}{2}c_0^{2M}\mathscr P^k(x).
\end{align}
Indeed, on $A^k(x)\cap\mathcal G_x$, our construction changes the value of $f$ from $0$ to $1$ with probability at least $p_0/2$, while~\eqref{eq:good_pivotal_configurations} gives the lower bound for the probability of such initial configurations.

~\\\emph{Comparison with ordinary pivotality probabilities.}
Lemma~\ref{lem:pivotality_cost} below shows that, for some locally bounded, non-increasing $A_\lambda\ge1$,
\begin{align}\label{est:estimates_for_pivotality_construction}
      \sum_{\substack{\bar u\in U_n^L\\(\bar u,a)\neq o}}
   P \bigl( \mathcal G_{(\bar u,a)},\ f(G_0)=0,\ f\big(G_{\mathrm{fin}}^{(\bar u,a)}\big)=1 \bigr)
  \leq
  A_\lambda^{\,1+\ell(a)^d}
  \sum_{b\in M_n}\Pi(b,n)
  \sum_{\bar z\in U_n^L}\mathscr P^k(\bar z,b).
\end{align}
Combining this estimate with
\eqref{eq:pivotality_lower_gain} and
$M\leq C_0(1+\ell(a)^d)$ proves the claim---after increasing the base of the exponential by setting $C_\lambda:=\frac{2}{p_0} c_0^{-2C_0} A_\lambda$.

The restrictions on the mesh size are satisfied uniformly for $\lambda$ in a compact interval once $n$ is sufficiently large, with the threshold depending on $k$ but not on $a$ or $L$.
\end{proof}

\paragraph{Probability estimates for the pivotality construction.}
We still need to justify the estimate~\eqref{est:estimates_for_pivotality_construction} used in the proof of Theorem~\ref{thm:ta_inf}. The following auxiliary lemma provides the desired justification. All notation refers to the construction in that proof.

\begin{lemma}\label{lem:pivotality_cost}
There exists a locally bounded, non-increasing function $\lambda\mapsto A_\lambda\ge1$, independent of $L,n,k$, such that
\[
  \sum_{\substack{\bar u\in U_n^L\\(\bar u,a)\neq o}}
   P \bigl( \mathcal G_{(\bar u,a)},\ f(G_0)=0,\ f\big(G_{\mathrm{fin}}^{(\bar u,a)}\big)=1 \bigr)
  \leq
  A_\lambda^{\,1+\ell(a)^d} \sum_{b\in M_n}\Pi(b,n) \sum_{\bar z\in U_n^L}\mathscr P^k(\bar z,b)
\]
for sufficiently large $n$, uniformly on compact intensity intervals.
\end{lemma}

\begin{proof} 
Throughout the proof, the constants $C_i$ depend only on
$d,p_0,s$ and are independent of $\lambda,L,n,k,a$.
We have
$$
W_Z
:=
\sum_{v\in Z\setminus\{o\}}
  \frac{p_\lambda(v)}{1-p_\lambda(v)}
\geq
\lambda s^d/2.
$$
Thus, for sufficiently large $n$, the selection probabilities
in~\eqref{eq:pivotal_selection} satisfy
\begin{equation}\label{est:pivotal_selection_estimates}
  q_v^Z\le C_1 2^{-nd}\Pi(b,n),
  \quad 
  v=(\bar z,b)\ne o,
\end{equation}
for a constant $C_1>0$.

Fix $x=(\bar u,a)$ and its $M=M(a)$ selected cylinders.
On $\mathcal G_x$, the graph-theoretic bounds from the proof of Theorem~\ref{thm:ta_inf} bound both the number of components after deleting a vertex by $q_*:=4+2(2d+1)$ and the number of cut vertices in the resulting induced graph by  $q_*^2$. By local positivity~\eqref{ass:LP}, conditioning the newly drawn edges to connect the chosen vertex to
every remaining component therefore increases probabilities by at most $c_0^*:=p_0^{-q_*}$.

Let $G_j$ denote the graph after the first $j$ local modifications, $0\le j\le T=2M-1$. We claim that, for a sufficiently large constant $C_2$ and $D_\lambda=C_2(1+\lambda^{-1})\geq 1$,
\begin{equation}\label{eq:restricted_modified_law}
  P(\mathcal G_x,\;G_j\in B)
  \le D_\lambda^j P(B)
\end{equation}
for every measurable set of graphs $B$. The assertion is immediate for $j=0$. For the induction step, retain the graph-theoretic restrictions above while applying the bound at time $j$:
\begin{itemize}
    \item Unchanged graphs contribute at most $D_\lambda^j P(B)$.
    \item If the modification involves the closed selected vertex being opened, the previous induced graph in the modified region was empty or disconnected. Hence the inserted vertex is either the only vertex of the resulting induced graph or one of its cut vertices. For each resulting graph, there are therefore at most $q_*^2$ possible inserted vertices. For each such vertex,
\[
  q_v^Z\frac{1-p_\lambda(v)}{p_\lambda(v)}
  =
  \frac1{W_Z}
  \leq
  2 s^{-d}\lambda^{-1}.
\]
The new incident edges have their original probabilities, apart from the conditioning described above. These modifications therefore contribute at most $q_*^2c_0^*(2 s^{-d}\lambda^{-1})D_\lambda^jP(B)$.
    \item For a modification with an already open vertex being selected, condition on all variables except its incident edges within the modified region. The new local edge law is bounded by a constant times the original product law. Integrating over the old local edges and summing the selection probabilities bounds this contribution by $c_0^*D_\lambda^jP(B)$.
\end{itemize}
Adding the three bounds yields~\eqref{eq:restricted_modified_law} for a sufficiently large $C_2>0$ in the definition of $D_\lambda$, e.g., with the choice $C_2:=\max \{1+c_0^*, 2q_*^2c_0^* s^{-d}\}$.

Next, consider a modification with selection cylinder $Z$. If opening a closed vertex changes $f$ from zero to one, that vertex is pivotal. If redrawing the local edges of an open vertex $v$ changes $f$ from zero to one, deleting $v$ gives the same graph before and after the modification. This graph has $f=0$, so $v$ is open and pivotal after the modification.

Using~\eqref{eq:restricted_modified_law} and the bound $c_0^*$ for conditioning the new edges gives
\begin{align*}
  P\bigl(\mathcal G_x,\;f(G_j)=0,\;f(G_{j+1})=1\bigr)
  &\leq
  c_0^* D_\lambda^j \sum_{v\in Z\setminus\{o\}}q_v^Z
  \left(\mathscr P^k(v)+P(v\text{ is open and pivotal})\right)
  \\&=
  c_0^* D_\lambda^j
  \sum_{v\in Z\setminus\{o\}}
  \frac{q_v^Z}{1-p_\lambda(v)}\mathscr P^k(v)
  \\&\leq
  \e 2^{-nd}  c_0^*  C_1 D_\lambda^j
  \sum_{v=(\bar z,b)\in Z\setminus\{o\}}
  \Pi(b,n)\mathscr P^k(v),
\end{align*}
where the last inequality holds by~\eqref{est:pivotal_selection_estimates} and since $(1-p_\lambda(v))^{-1}\leq e$ for $n$ sufficiently large.

For the final insertion, the probability of changing $f$ from zero to one when starting from a graph with law $P$ is exactly $\sum_{v\in Z_*\setminus\{o\}}q_v^{Z_*}\mathscr P^k(v)$. Applying~\eqref{eq:restricted_modified_law} at time $T$ and~\eqref{est:pivotal_selection_estimates} gives the same upper bound for this last step.

Since $f$ cannot decrease, it changes from zero to one at most once. Each cylinder is used for vertex selection at most $2d+2$ times, including the final insertion. Summing over the successive modifications therefore gives
\[
  P\bigl(\mathcal G_x,\: f(G_0)=0,\: f(G_{\mathrm{fin}}^x)=1\bigr)
  \leq
  C_3 2^{-nd} D_\lambda^{2M} \sum_{\substack{v=(\bar z,b)\in D(\bar u,a)\\v\ne o}} \Pi(b,n)\mathscr P^k(v),
\]
e.g., for the choice $C_3:= \e(2d+2) c_0^* C_1 > 0$.

Finally, the preparations give $M\le C_0(1+\ell(a)^d)$. For a vertex in $D(\bar u,a)$ with position $\bar z$, we have
\[
  \|\bar z-\bar u\|_\infty\le\ell(a)+2s,
\]
so at most $C_4 2^{nd}(1+\ell(a)^d)$ positions $\bar u$ contribute, where $C_4$ is obtained by increasing $C_0$
if necessary. Summing over such positions $\bar u$ yields the desired bound, since
\[
  C_3 C_4(1+\ell(a)^d)
  D_\lambda^{2C_0(1+\ell(a)^d)}
  \le A_\lambda^{\,1+\ell(a)^d}
\]
for a sufficiently large $A_\lambda\ge1$.
As $D_\lambda=C_2(1+\lambda^{-1})$, $A_\lambda$ can be chosen non-increasing and locally bounded on $(0,\infty)$.
\end{proof}

\section{Infinite-volume differential inequalities and the proof of Theorem~\ref{thm:Susceptibility Mean-Field Bound}}\label{sect:Inf_vol_lim}
Throughout this section, we will assume that Assumption~\ref{mainass} holds for some families of functions $(\mathfrak r_\lambda)_{\lambda>0}$ and $(\mathfrak R_\lambda)_{\lambda>0}$. Recall that we therefore have
\begin{align}\label{ass:tmc_5}
F_I:=\sup_{\lambda\in I}\sup_{n\in\N}
\sum_{m\in M_n}
\Pi(m,n)\mathfrak r_\lambda(m)\mathfrak R_\lambda(m)
<\infty
\end{align}
for any $I=[\lambda_1,\lambda_2]\subset (0,\infty)$. Naturally, to show exponential decay in the subcritical regime $\lambda<\lambda_T$, it would be sufficient to assume that the moment condition~\eqref{ass:tmc_5} holds for every compact interval $I\subset(0,\lambda_T)$. 
Moreover, notice that for every $\lambda\in I$, we have
\begin{align}\label{est:F_mon}
1\leq
\sum_{m\in M_n}
\Pi(m,n)\mathfrak r_\lambda(m)\mathfrak R_\lambda(m)
\leq F_I.
\end{align}

We now turn to the continuum system and derive the differential inequalities that will ultimately be used in infinite volume.  
Since our main results concern the infinite-volume quantities 
$\theta_\lambda(k)$ and $\E_\lambda[|\C(o)|]$, we must first introduce the
\emph{finite-volume continuum} versions of these objects.  
These will arise as limits of the finite-volume discrete approximations 
constructed before.

The results of this chapter are obtained by combining ideas from the proof of
\cite[Theorem~1.5]{Hut20} with the limiting procedures established in
Appendix~\ref{AIII}.  
Using Corollary~\ref{cor:4main}, we begin by taking the limit $n\to\infty$ in our
discrete approximation in order to obtain the differential inequalities for the
continuum system in a finite box.

\medskip

For $\lambda>0$ and $L,k\in\mathbb{N}$, we work under the infinite-volume measure
$\mathbb{P}_\lambda$, but we define the \emph{finite-volume cluster}
$\C_L(\bar 0,1)$ to be the connected component of $(\bar 0,1)$ in the graph
obtained from $\xi^o$ by \emph{removing all edges that leave the cylinder} $\Lambda^L := [-L,L]^d \times [1,\infty)$. Equivalently, vertices outside of this box are treated as absent.  
This is the standard finite-volume restriction in continuum percolation
(see, e.g., \cite[Chapter~3]{MeeRoy96}, \cite[Chapter~6]{Pen03}).

We then define the finite-volume cluster-tail function $\theta^L_\lambda$ by setting
\begin{align}
\theta^L_\lambda(k)
:=
\mathbb{P}_\lambda\big(|\C_L(o)|\ge k\big),
\end{align}
In the discrete approximation constructed before, we similarly defined
\begin{align}
\theta^{L,n}_\lambda(k)
:=
{P}^{L,n}_\lambda\big(|\mathcal{C}(o)|\ge k\big).
\end{align}

The continuum limit proved in Lemma~\ref{lem:fla} yields
\begin{align*}   
\theta^L_\lambda(k)
=
\lim_{n\to\infty}\theta^{L,n}_\lambda(k),
\end{align*}
for every fixed $\lambda>0$, $L\in\mathbb{N}$, and $k\in\mathbb{N}$.  
Together with the convergence of derivatives established
in Appendix~\ref{AIII}, this provides the differential inequality for the finite-volume continuum
system, which is the starting point for the infinite-volume analysis carried out
below.

\begin{lemma}\label{lem:cfv-diff} Let $I=[\lambda_1,\lambda_2]$. For any $\lambda\in I$, $L\geq L_0(I)$ and for all $k\in \N$, we have
\begin{align}\label{est:cfv-diff}
c_I\left(\frac{k}{\E_\lambda[|\C_{2L}(o)|]}-1\right) \theta_\lambda^{L}(k)\leq \frac{\dd}{\dd \lambda}\theta^{L}_\lambda(k),
\end{align}
where 
\begin{align}\label{def:un_const_c}
    c_I:=\frac{c}{\lambda_2 F_I}>0
\end{align}
for a sufficiently small universal constant $c>0$.
\end{lemma}
\begin{proof}
The claim follows from the first estimate~\eqref{est: 1cor3.10} claimed in Corollary~\ref{cor:3.2} by taking the limit $n\to\infty$ and using Lemma~\ref{lem:fla}, applying it to the cluster-tail
probabilities, their derivatives, and the expected
cluster sizes. 

The remainder of the proof of the lemma is essentially given by Lemma~\ref{lem:app1} from Appendix~\ref{AIII} since the contribution of the edge influences on the left side of~\eqref{est: 1cor3.10} vanishes in the limit $n\to\infty$.
\end{proof}

Instead of taking the infinite-volume limit $L\to\infty$ now, we first derive an exponential bound for $\theta^L_\lambda$ in finite volume:
\begin{lemma}\label{lem:4fcont} Let $0<\lambda<\lambda_0<\lambda_T$. Then, for $L$ sufficiently large, there exist constants $c(\lambda,\lambda_0)>0$ and $C(\lambda,\lambda_0)>0$ such that, for all $k\in\N$, we have
\begin{align*}
{\theta^L_\lambda(k)}\leq C(\lambda,\lambda_0)\exp\left\{-\frac{c(\lambda,\lambda_0)}{\E_{\lambda_0}[|\C_{2L}(o)|]}\times k\right\}.
\end{align*}
\end{lemma}

\begin{proof}
Fix $\lambda_0<\lambda_T$ and $\lambda<\lambda_0$, set $I=[\lambda,\lambda_0]$. Let $L>L_0(I)$, then the inequality~$\eqref{est:cfv-diff}$ --- which holds by Lemma~\ref{lem:cfv-diff} --- is equivalent to 
\begin{align}\label{est:cfv-d2}
\frac{\dd}{\dd s}\log\theta^{L}_s(k)
\geq
c_I\bigg(\frac{k}{\E_s[|\C_{2L}(o)|]}-1\bigg), \quad s\in I.
\end{align}
Integrating inequality~\eqref{est:cfv-d2} over $I$ and exploiting the monotonicity $\E_s[\vert \C_{2L}(o)\vert]\leq \E_{\lambda_0}[|\C_{2L}(o)|]$ for $s\in I$, we obtain
\begin{align}\label{est:cfv-d3}
    \log\bigg(\frac{\theta^L_{\lambda_0}(k)}{\theta^L_\lambda(k)}\bigg)\geq  c(\lambda,\lambda_0)\left(\frac{k}{\E_{\lambda_0}[|\C_{2L}(o)|]}-1\right)
\end{align}
for $c(\lambda,\lambda_0)=c_I(\lambda_0-\lambda)=\frac{c(\lambda_0-\lambda)}{\lambda_0 F_I}>0$. Taking the exponential on both sides then yields the claim, with the choice $C(\lambda,\lambda_0)=\e^{c(\lambda,\lambda_0)}$.
\end{proof}

Now, we are ready to get our main result in the subcritical regime --- by taking the infinite-volume limit $L\to\infty$:
\begin{proposition}\label{prop:UB_sus}
 Let $0<\lambda<\lambda_0<\lambda_T$. There exist constants $c(\lambda, \lambda_0)>0$ and $C(\lambda,\lambda_0)>0$ such that, for all $k\in\N$, we have
\begin{align}
\label{cluster_tails_ratio}
{\theta_\lambda(k)}\lesssim  C(\lambda,\lambda_0)\exp\left\{-\frac{c(\lambda, \lambda_0)}{\E_{\lambda_0}[\vert \C(o)\vert]}\times k\right\}.
\end{align}
\end{proposition}
\begin{proof}
Given Lemma~\ref{lem:converges_theta_inf_vol} and Lemma~\ref{lem:converges_expected_cluster_inf_vol}, the claim follows immediately from the inequality in Lemma~\ref{lem:4fcont} by taking the infinite-volume limit $L\to\infty$. Notice that $\E_{\lambda_0}[\vert \C(o)\vert]<\infty$ holds since $\lambda_0<\lambda_T$.
\end{proof}

In the subcritical case, we obtained the desired result as a continuum and infinite-volume limit of~\eqref{est: 1cor3.10} under suitable moment conditions --- but it does not provide a meaningful estimate in the supercritical case. We turn to the more universally useful bound~\eqref{est: 2cor3.10} next.

\begin{proposition}
\label{prop:infinite_derivative_inequality} Let $I\subset(0,\infty)$ be a compact interval. We have 
\begin{align}\label{est:5.13}
c_I\left(\frac{k\left(1-\e^{-1}\right)}{ \sum_{i=1}^{k}\theta_\lambda(i)}-1\right) \theta_\lambda(k)
\lesssim  
\frac{\dd}{\dd \lambda}\theta_\lambda(k), \quad \lambda\in I,
\end{align}
with the constant $c_I>0$ given by~\eqref{def:un_const_c}. Moreover, let $0<\lambda_1< \lambda_0$, then, for every $k \geq 1$, we have
\begin{align}
\label{cluster_tails_ratio2}
    \theta_{\lambda_1}(k) \lesssim C(\lambda_1,\lambda_0) \theta_{\lambda_0}(k) \exp\left\{-\frac{(1-\e^{-1})c(\lambda_1,\lambda_0)k}{\sum_{i = 1}^k \theta_{\lambda_0}(i)}\right\},
\end{align}
where the constants $c(\lambda_1,\lambda_0)$, $C(\lambda_1,\lambda_0)>0$ are as in the proof of Lemma~\ref{lem:4fcont}.
\end{proposition}

\begin{proof}
Starting from the second estimate~\eqref{est: 2cor3.10} claimed in Corollary~\ref{cor:3.2}, we take $n\to\infty$
using Lemmas~\ref{lem:app1} and~\ref{lem:fla}. The resulting inequality is
\[
\frac{\dd}{\dd\lambda}\theta_\lambda^L(k)
\geq
c_I\left(
\frac{k(1-\e^{-1})}{\sum_{i=1}^k\theta_\lambda^{2L}(i)}
-1
\right)\theta_\lambda^L(k),
\quad \lambda\in I,
\]
for $L$ sufficiently large, with $c_I$ as
in~\eqref{def:un_const_c}. Fix $\lambda\in I$. Integrating over any compact interval $[\lambda_1,\lambda_0]$ satisfying $\lambda\in [\lambda_1,\lambda_0]\subset I$ and taking $L\to\infty$ (by using
Lemma~\ref{lem:converges_theta_inf_vol} and dominated convergence) yields
\[
\theta_{\lambda_0}(k)-\theta_{\lambda_1}(k)
\geq
c_I\int_{\lambda_1}^{\lambda_0}
\left(
\frac{k(1-\e^{-1})}{\sum_{i=1}^k\theta_s(i)}
-1
\right)\theta_s(k) \dd s.
\]
For all $i\in\N$, the maps $\lambda\mapsto \theta_\lambda(i)$ are differentiable (see
~\cite[Lemma~10.4]{CL24}); in particular, the integrand on the right side is continuous. Thus, dividing by $\lambda_0-\lambda_1$ and letting both endpoints tend to $\lambda\in I$, we obtain~\eqref{est:5.13}.

Finally, dividing~\eqref{est:5.13} by $\theta_\lambda(k)$,
integrating over $[\lambda_1,\lambda_0]$, using
monotonicity of $\lambda \mapsto \sum_{i=1}^k\theta_\lambda(i)$, and finally taking the exponential on both sides gives
\eqref{cluster_tails_ratio2}.
\end{proof}

We now move on to the supercritical regime where we closely follow the proof strategy for the second part of~\cite[Theorem 1.5]{Hut20} to derive the following result:
\begin{proposition}\label{prop:LB_sus}
Let $\lambda>\lambda_T$. Then we have 
\begin{align*}
    \theta(\lambda)=\lim_{k\to\infty}\theta_\lambda(k)>0
\end{align*}
and thus $\lambda_T\geq \lambda_c$ holds. To be more precise, for any $\lambda_0 >\lambda_T$, there exists a constant $c_0 >0$ such that we have $\theta(\lambda) \geq c_0(\lambda - \lambda_T)$ for all $\lambda \in (\lambda_T,\lambda_0]$. 
\end{proposition}

\begin{proof}
We start by defining an auxiliary critical parameter 
\begin{align*}
    \nonumber\widetilde{\lambda}_c &:= \sup\{\lambda \geq 0 :  \exists \alpha,\beta >0 \text{ such that } \theta_\lambda(k) \leq \alpha k^{-\beta} \text{ for every } k \geq 1\} \\
    &= \inf \left\{ \lambda \geq 0 : \limsup_{n \to \infty} \frac{\log \theta_\lambda(n)}{\log n} \geq 0\right\}
\end{align*}
Trivially, $\widetilde{\lambda}_c \leq \lambda_c$ and $\lambda_T\leq \widetilde\lambda_c$ hold by Proposition~\ref{prop:UB_sus} (which shows that the thresholds for finite expected cluster size and exponential cluster-size tails coincide). The proof strategy is to show first $\widetilde{\lambda}_c=\lambda_T$ and then $\widetilde{\lambda}_c=\lambda_c$. Indeed, we actually have exponential decay below $\widetilde{\lambda}_c$ which we prove in the following: Fix $\lambda_0 < \widetilde{\lambda}_c$ so that there exist $\alpha, \beta > 0$ such that $\theta_{\lambda_0}(k) \leq \alpha k^{-\beta}$ for any $k \geq 1$. Without loss of generality, we can assume that $\beta<1$ holds. Once again, we use the notation $\lesssim$ to denote inequalities up to a multiplicative constant (here we allow these constants to depend only on $\alpha, \beta, \lambda_0$ and model assumptions like the dimension $d\geq 2$). For every $k \geq 1$, we get
\begin{align*}
    \sum_{i = 1}^k \theta_{\lambda_0}(i) \lesssim k^{1 - \beta}
\end{align*}
since $\sum_{i=1}^k i^{-\beta}\leq 1+\frac{1}{1-\beta} k^{1-\beta}$ holds for $\beta<1$. By \eqref{cluster_tails_ratio2}, we then get
\begin{align}\label{ineq:5.17}
    \theta_{\lambda_1}(k) \lesssim k^{- \beta}\exp\left\{-(1-\e^{-1})c_1 k^\beta\right\},
\end{align}
for some $c_1=c_1(\lambda_1,\lambda_0)>0$, any $0 < \lambda_1 < \lambda_0$ and $k \geq 1$. Summing inequality~\eqref{ineq:5.17} over $k$ leads to
\begin{align*}
    \E_{\lambda_1}[\abs{\mathscr C(o)}] \lesssim \sum_{k = 1}^\infty k^{ - \beta}\exp(-(1-\e^{-1})c_1k^\beta) 
    \lesssim (1-\e^{-1})^{1 - 1/\beta}c_1^{1 - 1/\beta}<\infty
\end{align*}
for any $0 < \lambda_1 < \lambda_0$. Since $\lambda_1<\lambda_0<\widetilde\lambda_c$ were arbitrary, we conclude $\lambda_T\geq \widetilde\lambda_c$.~\\

Next, we prove that $\theta(\lambda) > 0$ holds for every $\lambda > \widetilde{\lambda_c}$. Consider, for $\lambda>0$,
\begin{align*}
    T_N(\lambda) = \frac{1}{\log N}\sum_{i = 1}^N\frac{1}{i}\theta_{\lambda}(i),\quad N\geq 2,
\end{align*}
and notice that we have $\lim_{N \to \infty} T_N(\lambda) = \theta(\lambda)$. We use the notation $\Sigma_k (\lambda) = \sum_{i = 0}^{k - 1}\theta_{\lambda}(i)$.

Applying~\eqref{est:5.13} from
Proposition~\ref{prop:infinite_derivative_inequality},
we obtain
\begin{align}\label{est:sup_aux_1}
\frac{\dd}{\dd\lambda}T_N(\lambda)
\geq
\frac{c_I}{\log N}
\sum_{k=1}^N
\left[
(1-\e^{-1})
\frac{\theta_\lambda(k)}{\Sigma_k(\lambda)}
-
\frac{\theta_\lambda(k)}{k}
\right]
\end{align}
for every compact interval $I\subset (0,\infty)$, every $\lambda\in I$ and every $N\geq2$.

Notice now that
\begin{align}\label{est:sup_aux_2}
\frac{\theta_\lambda(k)}{\Sigma_k(\lambda)}
\geq
\int_{\Sigma_k(\lambda)}^{\Sigma_{k+1}(\lambda)}
\frac{1}{x}\,\dd x
=
\log\Sigma_{k+1}(\lambda)-\log\Sigma_k(\lambda).
\end{align}
Combining~\eqref{est:sup_aux_1}
and~\eqref{est:sup_aux_2} yields
\begin{align*}
\frac{\dd}{\dd\lambda}T_N(\lambda)
\geq
c_I(1-\e^{-1})
\frac{\log\Sigma_{N+1}(\lambda)}{\log N}
-
c_IT_N(\lambda)
\end{align*}
for every $\lambda\in I$.

Now, fix $\widetilde\lambda_c<\lambda_1<\lambda_2$
with $[\lambda_1,\lambda_2]\subset I$.
By monotonicity,
\begin{align*}
\frac{\dd}{\dd\lambda}T_N(\lambda)
\geq
c_I(1-\e^{-1})
\frac{\log\Sigma_{N+1}(\lambda_1)}{\log N}
-
c_IT_N(\lambda_2)
\end{align*}
for almost every $\lambda\in[\lambda_1,\lambda_2]$.
Since $T_N$ is non-decreasing, integration gives
\begin{align*}
T_N(\lambda_2)-T_N(\lambda_1)
\geq
c_I(\lambda_2-\lambda_1)
\left[
(1-\e^{-1})
\frac{\log\Sigma_{N+1}(\lambda_1)}{\log N}
-
T_N(\lambda_2)
\right].
\end{align*}

Furthermore,
\begin{align*}
\frac{\log\Sigma_{N+1}(\lambda_1)}{\log N}
\geq
\frac{\log(N\theta_{\lambda_1}(N))}{\log N}
=
1+\frac{\log\theta_{\lambda_1}(N)}{\log N}.
\end{align*}
Since $\lambda_1>\widetilde\lambda_c$ and
$\Sigma_{N+1}(\lambda_1)\leq N+1$, this implies
\[
\limsup_{N\to\infty}
\frac{\log\Sigma_{N+1}(\lambda_1)}{\log N}
=1.
\]
Taking a subsequence along which this ratio converges
to $1$, we obtain
\begin{align*}
\theta(\lambda_2)-\theta(\lambda_1)
\geq
c_I(\lambda_2-\lambda_1)
\left[
1-\e^{-1}-\theta(\lambda_2)
\right].
\end{align*}
Consequently,
\begin{align}\label{est:super_lb_lin}
\theta(\lambda_2)
\geq
\frac{(1-\e^{-1})c_I(\lambda_2-\lambda_1)}
{1+c_I(\lambda_2-\lambda_1)}
>0.
\end{align}
The choice of
$\lambda_2>\lambda_1>\widetilde\lambda_c$
was arbitrary. Thus
$\lambda_c\leq\widetilde\lambda_c$, and we conclude
\[
\lambda_T=\widetilde\lambda_c=\lambda_c.
\]

Finally, fix $\lambda_0>\lambda_T$ and take
$I=[\lambda_T,\lambda_0]$, recalling that
$\lambda_T>0$. Letting $\lambda_1\downarrow\lambda_T$
in~\eqref{est:super_lb_lin} gives
\[
\theta(\lambda)
\geq
c_0(\lambda-\lambda_T),
\quad \lambda\in[\lambda_T,\lambda_0],
\]
where
\[
c_0
:=
\frac{(1-\e^{-1})c_I}
{1+c_I(\lambda_0-\lambda_T)}
>0.
\]
\end{proof}

The proof of our main result Theorem~\ref{thm:Susceptibility Mean-Field Bound}---postulating the sharpness of the phase transition for models satisfying either Assumption~\ref{ass:nonweighted} or Assumption~\ref{ass:weighted}---now follows immediately from the results of this section:

\begin{proof}[Proof of Theorem~\ref{thm:Susceptibility Mean-Field Bound}]
We have shown in Section~\ref{sect:Pf_mainass} that Assumption~\ref{ass:weighted} implies
Assumption~\ref{mainass}. In particular, the locally uniform
moment bound required there follows from~\eqref{ass:MR}.

In the non-weighted case, Assumption~\ref{mainass} holds
with
\[
  \mathfrak r_\lambda\equiv 1,
  \qquad
  \mathfrak R_\lambda\equiv 2.
\]
Therefore, Propositions~\ref{prop:UB_sus} and~\ref{prop:LB_sus} apply to  both the weighted and non-weighted models considered here. Lemma~\ref{lem:ex_phase_transition} gives $0<\lambda_T\leq\lambda_c<\infty$. Proposition~\ref{prop:UB_sus} yields exponential cluster-size
tails for $\lambda<\lambda_T$, while Proposition~\ref{prop:LB_sus} establishes $\lambda_T=\lambda_c$ and the claimed linear lower bound on the percolation probability above criticality.
\end{proof}

\appendix

\part*{Appendix}

\setcounter{section}{0}
\renewcommand\thesection{\Roman{section}}
\section{Tiny tool box}\label{AI}

In order to prove results presented later in the appendix, we will require several well-known technical tools. For proofs and further discussion of the following theorems, the reader is referred to Last and Penrose \cite{LasPen17}.

\paragraph{Mecke's Formula.}
Given $m\in\N$ and a nonnegative measurable function $f$ of a graph and $m$ points, the Mecke equation for $\xi$ states that 
\begin{equation}
    \E_\lambda \left[ \sum_{\vec x \in \eta^{(m)}} f(\xi, \vec{x})\right] = \lambda^m \int
				\E_\lambda\left[ f\left(\xi^{x_1, \ldots, x_m}, \vec x\right)\right] \nu^{\otimes m}\left(\dd \vec{x}\right),  \label{eq:prelim:mecke_n}
\end{equation}
where $\vec x=(x_1,\ldots,x_m)$, $\eta^{(m)}=\{(x_1,\ldots,x_m)\colon x_i \in \eta, x_i \neq x_j \text{ for } i \neq j\}$, and $\nu^{\otimes m}$ is the $m$-product measure of $\nu$ on $\X^m$.

\paragraph{Restriction property.}
Let $\eta$ be a Poisson point process on a measurable space $(\X,\mathcal{X})$ with intensity 
measure $\lambda\nu$, and let $A\subset \X$ be measurable.  
Then the restricted process $\eta\cap A$ is a Poisson point process with 
intensity $\lambda\nu|_A$.  
If $(A_i)_{i\in I}$ are disjoint measurable subsets of $\X$, the restrictions 
$\eta\cap A_i$ are independent.

\paragraph{Projection property.}
Let $\eta$ be a Poisson point process on $\X\times \mathbb{Y}$ with intensity 
$\mu(\mathrm dx)\,\nu(\mathrm dy)$.  
Under the projection map $(x,y)\mapsto x$, the image of $\eta$ is a Poisson 
point process on $\X$ with intensity measure $\mu(\mathrm dx)\,\nu(\mathbb{Y})$.  
More generally, the image of a Poisson point process under any measurable map 
is a Poisson process with the pushed-forward intensity measure.~\\

We first apply these tools to establish the existence of a non-trivial phase transition.

\section{Existence of phase transition}\label{AII}

In this section, we provide a proof of Lemma~\ref{lem:ex_phase_transition}, which guarantees the existence of a phase transition under Assumption~\ref{ass:weighted}. In the non-weighted setting of Assumption~\ref{ass:nonweighted}, the result is well-known; the following argument also applies in that setting.

The subcritical part follows the classical Galton–Watson comparison method used in continuum percolation, as developed in the monographs by Meester and Roy~\cite{MeeRoy96} and Penrose~\cite{Pen03}; see also Last and Penrose~\cite{LasPen17} for the modern formulation via Mecke’s formula and Slivnyak’s theorem.

For the supercritical part, condition~\eqref{ass:LP} ensures that our model stochastically dominates an unweighted RCM with finite range. The latter percolates at sufficiently large intensities for $d\geq2$ by classical results (see~\cite[Chapter~6]{MeeRoy96}), yielding $\lambda_c<\infty$.

We now give the details.

\begin{proof}[Proof of Lemma~\ref{lem:ex_phase_transition}]
We first show that $\lambda_T>0$, i.e. that the expected cluster size 
$\chi(\lambda)=\E_\lambda[|\C(o)|]$ is finite for sufficiently small~$\lambda>0$.

We explore the cluster $\C(o)$ in the graph $\xi^{o}$ in a breadth–first manner.  
For $n \ge 0$, let $\mathcal{Z}_n$ denote the set of vertices discovered at generation~$n$ and set $Z_n = |\mathcal{Z}_n|$.  
The exploration starts with $\mathcal{Z}_0 = \{o\}$, hence $Z_0=1$.  
For $n \ge 0$, the next generation $\mathcal{Z}_{n+1}$ consists of those points of $\eta$ that, in $\xi^{o}$, are adjacent to some vertex in $\mathcal{Z}_n$ but do not belong to $\mathcal{Z}_0 \cup \cdots \cup \mathcal{Z}_n$.

We define
\[
\mathcal{E}_n = \bigcup_{k=0}^{n} \mathcal{Z}_k,
\qquad
\mathcal{R}_n = \eta \setminus \mathcal{E}_n,
\]
so that $\mathcal{E}_n$ represents the vertices already discovered up to generation~$n$, and $\mathcal{R}_n$ the remaining unexplored vertices.  

Each vertex $v$ in $\mathcal Z_n$ can be reached from $o$ by a path of vertices $o=v_0,\ldots, v_{n}=v$ with $v_j=(\bar{u}_j, a_j)\in \mathcal{Z}_j$, $0\leq j\leq n$. By Mecke’s formula and the independence of the edges $\{v_{j-1},v_{j}\}$ being open, the expected number of points in $\mathcal{Z}_{n}$ admits the following upper bound:
\begin{align*}\label{eq:subcrit_conditional_bound}
\E_\lambda[Z_{n}]
\leq&
\E_\lambda[\#\{\text{open paths of length } n \text{ starting in o}\}]
\\=&
\lambda^n \int_{[1,\infty)^{n}} \bigg( \prod_{j=1}^{n}\int_{\R^d} \varphi(\vert \bar{x}_j \vert; a_{j-1}, a_{j})\dd \bar{x}_j\bigg) \prod_{j=1}^{n} \pi(\dd a_j).
\end{align*}
Recall the notation $K(a,b)=\int_{\R^d} \varphi(\vert \bar{x} \vert; a, b)\dd \bar{x}$ from Remark~\ref{rem:consequences_assumptions} and the bounds 
\begin{align*}
  \int_{[1,\infty)}
  K(1,b)^2 \pi(\dd b)
  \leq 
  \int_{[1,\infty)}
  \int_{[1,\infty)}
  K(a,b)^2\,\pi(\mathrm da)\pi(\dd b)
  =:B
  <
  \infty
\end{align*}
derived therein. Applying Cauchy-Schwarz repeatedly yields
\begin{align*}
\int_{[1,\infty)^{n}}\! \bigg( \prod_{j=1}^{n}K(a_{j-1}, a_{j})\bigg) \!\prod_{j=1}^{n} \pi(\dd a_j)\leq \sqrt{\int\! K(1,a)^2 \pi(\dd a) \Big(\iint K(a,b)^2 \pi^2(\dd (a,b))\Big)^{n-1}}\!\leq B^{\frac{n}{2}}.    
\end{align*}
Altogether, we have $\chi(\lambda)=\sum_n E[Z_n]\leq \sum_n (\lambda\sqrt B)^n$, which converges for $\lambda<\frac{1}{\sqrt B}$.
Hence, the expected cluster size is finite in this regime, and consequently
\[
\lambda_T \ge \Big(\iint K(a,b)^2 \pi^2(\dd (a,b))\Big)^{-\frac 12} > 0.
\]

Next, we show that $\theta(\lambda) > 0$ holds for all 
$\lambda$ sufficiently large, and thus $\lambda_c<\infty$: By condition~\eqref{ass:LP} and monotonicity, there exist
$\rho>0$ and $p_0>0$ such that
\[
  \varphi(r;a,b)\geq p_0,
  \quad 0\leq r\leq\rho,\quad a,b\geq1.
\]
By the projection property from Appendix~\ref{AI}, projecting $\eta$ onto the spatial coordinate yields a homogeneous PPP on $\mathbb R^d$ of intensity $\lambda$. On this process, with the origin inserted, consider the unweighted
RCM with adjacency function
\[
  g(r):=p_0\mathbf1_{\{r\leq\rho\}}.
\]
Since $g(r)\leq\varphi(r;a,b)$ for all $r\geq0$ and $a,b\geq1$, we may couple the models so that every edge of the unweighted model is also present in the marked model. This includes edges incident to the inserted root $o$.

For $d\geq2$, the unweighted model has a finite critical intensity $\lambda_c(g)$; see~\cite[Chapter~6]{MeeRoy96}.
Consequently, for $\lambda>\lambda_c(g)$, the root belongs to the infinite cluster with positive probability in the unweighted comparison model, and hence also in the original model. Thus $\lambda_c\leq\lambda_c(g)<\infty$.

\end{proof}

\section{Taking limits}\label{AIII}

In this appendix section we collect the limit transitions needed for passing 
from the discrete, finite-volume setting to the continuous, infinite-volume framework.  
The finite-volume discrete differential inequalities used earlier
 (see Corollaries~\ref{cor:3.1} and~\ref{cor:3.2}) are established on a bounded 
domain.  
To obtain their continuous and infinite-volume counterparts, we first take the 
continuum limit $n\to\infty$ at fixed volume $L$, and subsequently the 
infinite-volume limit $L\to\infty$.  
The results below formalize these two steps and justify the transition to the 
infinite-volume continuous differential inequalities in Section~\ref{sect:Inf_vol_lim}.

For the classical non-weighted lattice approximation,
see~\cite{Mee95}. We focus below on the weighted case;
the same argument applies to the non-weighted model by
taking $M_n=\{1\}$ and omitting the mark arguments.

\subsection{Continuum limit}
In this subsection we take the continuum limit $n\to\infty$ while keeping the volume $L$ fixed.  
This provides the passage from the discrete approximation of our model to its finite-volume continuum version.

The next result confirms that the contribution of edge influences vanishes as $n\to\infty$ (by design of our exploration scheme), a 
fact that is crucial for reducing the discrete differential inequality to its continuum form and for proceeding in the spirit of~\cite{Hut20}.

\begin{lemma}\label{lem:app1}
Let $k\in\N$, let $I\subset(0,\infty)$ be a compact interval, and suppose
that the magnetization comparison in
Assumption~\ref{mainass}\textup{(a)} holds. Assume that
\begin{align}
B_I:=
\sup_{\lambda\in I}\sup_{n\in\N}
\sum_{m\in M_n}\mathfrak R_\lambda(m)\Pi(m,n)<\infty.
\end{align}
Then, for every $L,n\in\N$ sufficiently large, in the range specified
by Assumption~\ref{mainass},
\[
I^{n,L,\tilde\gamma}_\lambda
:=
\frac{1}{M_{\tilde\gamma}^{2L,n}(o)}
\sum_{e\in E_n^L}\delta_{\tilde\gamma}(e)
\textsc{Inf}^P_e
\left(\mathds{1}_{\{\vert\mathcal C(o)\vert\geq k\}}\right)
\]
is bounded from above uniformly in $\lambda\in I$ and
$\tilde\gamma>0$. Moreover, for every fixed $L$ sufficiently large,
we have
\[
\lim_{n\to\infty}
\sup_{\lambda\in I}\sup_{\tilde\gamma>0}
I^{n,L,\tilde\gamma}_\lambda=0.
\]
\end{lemma}

\begin{proof}
Fix $\lambda\in I$.
For $L,n\in\N$, let $\hat E_n^L$ denote the set of edges in $E_n^L$
incident to $o$ and
$\tilde E_n^L:=E_n^L\setminus\hat E_n^L$.
The event that a vertex $v\neq o$ is revealed is independent of its
openness.
For every $e=\{v,o\}\in\hat E_n^L$ such that
$v=(\Bar{u},m)$ we have
\begin{align}
    \delta_{\tilde\gamma}(e)
    \leq \delta_{\tilde\gamma}(v)
    P^{L,n}_\lambda(v \text{ is open}),
\end{align}
since the edge $e$ is only revealed if $v$ is both revealed and open,
as the (tree) explorations halt upon reaching $o$.
Similarly, for every $e=\{v_1,v_2\}\in\tilde E_n^L$ such that
$v_1=(\Bar{u}_1,m_1)$ and $v_2=(\Bar{u}_2,m_2)$ we have
\begin{align}
    \delta_{\tilde\gamma}(e)
    \leq \delta_{\tilde\gamma}(v_1)
    P^{L,n}_\lambda(v_1 \text{ is open})
    +\delta_{\tilde\gamma}(v_2)
    P^{L,n}_\lambda(v_2 \text{ is open}),
\end{align}
since $e=\{v_1,v_2\}\in\tilde E_n^L$ is only revealed if at least
one of the vertices $v_1,v_2$ is both revealed and open.
Using the revealment bound
$\delta_{\tilde\gamma}(\Bar{u},m)
\leq M_{\tilde\gamma}^{L,n}(\Bar{u},m)$
and Assumption~\ref{mainass}\textup{(a)}, we get:
\begin{itemize}
    \item For $e=\{v,o\}\in\hat E_n^L$ such that
    $v=(\Bar{u},m)$
\begin{align*}
\frac{\delta_{\tilde\gamma}(e)}
{M_{\tilde\gamma}^{2L,n}(o)}
\leq
P^{L,n}_\lambda(v \text{ is open})\mathfrak R_\lambda(m)
\leq
\Big(1-\e^{-\lambda 2^{-nd}}\Big)\mathfrak R_\lambda(m).
\end{align*}
\item For $e=\{v_1,v_2\}\in\tilde E_n^L$ such that
$v_1=(\Bar{u}_1,m_1)$ and $v_2=(\Bar{u}_2,m_2)$
\begin{align*}
\frac{\delta_{\tilde\gamma}(e)}
{M_{\tilde\gamma}^{2L,n}(o)}
\leq&
P^{L,n}_\lambda(v_1 \text{ is open})\mathfrak R_\lambda(m_1)
+P^{L,n}_\lambda(v_2 \text{ is open})\mathfrak R_\lambda(m_2)
\\\leq&
\Big(1-\e^{-\lambda 2^{-nd}}\Big)
(\mathfrak R_\lambda(m_1)+\mathfrak R_\lambda(m_2)).
\end{align*}
\end{itemize}
Therefore, we obtain the estimate
\begin{align}\label{est:A_est1}
\nonumber&
\frac{1}{M_{\tilde\gamma}^{2L,n}(o)}
\sum_{e\in E_n^L}\delta_{\tilde\gamma}(e)
\textsc{Inf}^P_e
\left(\mathds 1_{\{\vert\mathcal C(o)\vert\geq k\}}\right)
\\\leq&
\sum_{e=\{(\Bar{u},m),o\}\in\hat E_n^L}
\mathfrak R_\lambda(m)\Big(1-\e^{-\lambda 2^{-nd}}\Big)
\textsc{Inf}^P_e
\left(\mathds 1_{\{\vert\mathcal C(o)\vert\geq k\}}\right)
\nonumber
\\&+
\sum_{e=\{(\Bar{u}_1,m_1),(\Bar{u}_2,m_2)\}\in\tilde E_n^L}
\left(\mathfrak R_\lambda(m_1)+\mathfrak R_\lambda(m_2)\right)
\Big(1-\e^{-\lambda 2^{-nd}}\Big)
\textsc{Inf}^P_e
\left(\mathds 1_{\{\vert\mathcal C(o)\vert\geq k\}}\right).
\end{align}

Next, we use the following estimates for the influences of the edges:
For $e=\{v,o\}\in\hat E_n^L$ such that
$v=(\Bar{u},m)$, we have
\begin{align*}
\textsc{Inf}^P_e(f)
\leq P^{L,n}_\lambda(v \text{ is open})
\leq \lambda 2^{-nd}\Pi(m,n)
\end{align*}
and for $e=\{v_1,v_2\}\in\tilde E_n^L$ such that
$v_1=(\Bar{u}_1,m_1)$ and $v_2=(\Bar{u}_2,m_2)$, we have
\begin{align*}
\textsc{Inf}^P_e(f)
\leq
P^{L,n}_\lambda(v_1 \text{ is open})
P^{L,n}_\lambda(v_2 \text{ is open})
\leq
\lambda^2 2^{-2nd}\Pi(m_1,n)\Pi(m_2,n).
\end{align*}

Plugging those estimates into~\eqref{est:A_est1} and using the notation
$\vert\hat E\vert$ (resp. $\vert\tilde E\vert$) for the number of
unweighted edges in $\hat E_n^L$ (resp. $\tilde E_n^L$), we get
\begin{align*}
\nonumber&
\frac{1}{M_{\tilde\gamma}^{2L,n}(o)}
\sum_{e\in E_n^L}\delta_{\tilde\gamma}(e)
\textsc{Inf}^P_e
\left(\mathds 1_{\{\vert\mathcal C(o)\vert\geq k\}}\right)
\\\leq&
\bigg(\sum_{m\in M_n}\mathfrak R_\lambda(m)\Pi(m,n)\bigg)
\lambda 2^{-nd}\Big(1-\e^{-\lambda 2^{-nd}}\Big)
\vert\hat E\vert
\nonumber
\\&+
\bigg(
\sum_{m_1\in M_n}\sum_{m_2\in M_n}
\Big(\mathfrak R_\lambda(m_1)+\mathfrak R_\lambda(m_2)\Big)
\Pi(m_1,n)\Pi(m_2,n)
\bigg)
\Big(1-\e^{-\lambda 2^{-nd}}\Big)
\lambda^2 2^{-2nd}\vert\tilde E\vert.
\end{align*}
Moreover, we have
\begin{align*}
\vert\hat E\vert\leq 2^{nd}C(L,d),
\qquad
\vert\tilde E\vert\leq 2^{2nd}C(L,d)^2
\end{align*}
for a constant $C(L,d)>0$ which only depends on $L$ and $d$.
Altogether, we obtain
\begin{align*}
&
\frac{1}{M_{\tilde\gamma}^{2L,n}(o)}
\sum_{e\in E_n^L}\delta_{\tilde\gamma}(e)
\textsc{Inf}^P_e
\left(\mathds 1_{\{\vert\mathcal C(o)\vert\geq k\}}\right)
\\\leq&
\big[\lambda C(L,d)+2\lambda^2C(L,d)^2\big]
\bigg(\sum_{m\in M_n}\mathfrak R_\lambda(m)\Pi(m,n)\bigg)
\Big(1-\e^{-\lambda 2^{-nd}}\Big).
\end{align*}
Since the sum is bounded by $B_I$ and
$1-\e^{-\lambda 2^{-nd}}\leq\lambda 2^{-nd}$,
the right-hand side is bounded by
\[
C(I,L,d)B_I\,2^{-nd}:=\big[\lambda^2_+ C(L,d)+2\lambda^3_+ C(L,d)^2\big]B_I\,2^{-nd},
\]
with $\lambda_+ := \max I$, uniformly in $\lambda\in I$ and $\tilde\gamma>0$.
This proves the claim.
\end{proof}

We next show that the discrete quantities appearing in the finite-volume
differential inequalities converge, as the mesh size $2^{-n}$ tends to zero, to their
continuous counterparts on $\Lambda^L = [-L,L]^d \times [1,\infty)$.  
To do so, it is convenient to couple our discrete Bernoulli models to the finite-volume
continuum random graph using a single Poisson point configuration
and shared independent edge variables (in a fixed cylinder containing all the approximating windows).

\begin{coupling*}
Fix $L\in\N$ and consider the Poisson point process $\eta$ of intensity $\lambda$ restricted to the cylinder 
$$\Lambda_0^L = [-L- 1/2,L+ 1/2]^d \times [1,\infty).$$
Naturally, $\nu(\Lambda_0^L)<\infty$ and $\eta(\Lambda_0^L)$ is a Poisson random variable with mean $\lambda\nu(\Lambda_0^L)$. Assign  uniform random variables $U_e\sim\mathcal U([0,1])$ to all potential edges $e$ between $(\eta \cap \Lambda^L_0)\cup \{o\}$ independently.

For each $n$, we consider the partition (up to $\nu$-null sets) of $\Lambda_n^L$ into space-weight
cells $\{R_{u,n}\colon u=(\bar u,m)\in V_n^L\}$. The cells are given, as already hinted at in the definition of the finite lattice approximation in Section~\ref{sect:2}, by 
$$R_{u,n}=\bigotimes_{i=1}^d[\Bar{u}_i-2^{-n-1},\Bar{u}_i+2^{-n-1}]\times [m,m+2^{-n}),\quad u=(\Bar{u},m)\in V^L_n,$$
if $\pi$ is absolutely continuous, and by 
$$R_{u,n}=\bigotimes_{i=1}^d[\Bar{u}_i-2^{-n-1},\Bar{u}_i+2^{-n-1}]\times \{m\},\quad u=(\Bar{u},m)\in V^L_n,$$
if $\pi$ is purely atomic. 

Consider a realization of the continuum marked Poisson graph $\xi^{o}$
in $\Lambda_0^L$ constructed using the uniforms $U_e$:  an edge $e$ between two Poisson points $(\bar x,a), (\bar y,b)$ is drawn if and only if $U_e<\varphi(\vert \bar x - \bar y\vert; a, b )$. We associate to it a configuration (of open vertices and edges) in $\Omega_n^L$ as follows: A vertex $v\in V_n^L$ is open if the cell
$R_{v,n}$ contains at least one point of $\eta^{o}$, and
closed otherwise. Moreover, choose in every occupied cell one Poisson point as a representative by some deterministic rule, e.g., the one closest to the center of the cell, such that $o$ is always chosen in its cell. The edge between two open vertices $v_1=(\bar u_1,m_1)$ and $v_2=(\bar u_2,m_2)$ is itself open if the uniform assigned to the pair of their respective representatives falls below the connection probability $\varphi(\vert \bar u_1-\bar u_2 \vert; m_1, m_2)$. Edges with a closed endpoint we open separately, by supplying an additional set of
independent uniform variables (that are not shared with $\xi^{o}$).

Naturally, the resulting configuration has the desired law $P^{L,n}_\lambda$ (essentially, by the restriction property from Appendix~\ref{AI}).

\end{coupling*}

\begin{lemma}\label{lem:fla} Suppose that Assumption~\ref{ass:general} holds as well as that $\pi$ is either absolutely continuous or purely atomic. Let $\lambda > 0$ and $L,k\in \N$. As $n$ goes to infinity, the following limits hold: 
\begin{enumerate}
\item[(1)]\label{it1:lem:fla} $\theta^{L,n}_\lambda(k) \xrightarrow[n\to\infty]{} \theta^L_\lambda(k)$,
~\\\item[(2)]\label{it3:lem:fla} $E_\lambda^{L,n}[\vert \mathcal{C}(o)\vert]\xrightarrow[n\to\infty]{}\E_{\lambda}[|\C_L(o)|]$,
~\\\item[(3)]\label{it2:lem:fla} $\frac{d}{d\lambda} \theta^{L,n}_\lambda(k) \xrightarrow[n\to\infty]{}\frac{d}{d\lambda} \theta^L_\lambda(k)$. 

\end{enumerate}
\end{lemma}

\begin{proof}

   Almost surely, for $n$ sufficiently large, every cell $R_{u,n}$ contains at most one Poisson point from $\eta$, while $R_{o,n}$ and the boundary region $\Lambda_n^L\setminus \Lambda^L$ contain none. For such $n$, this gives a one-to-one correspondence between open vertices in $V^L_n$ and $\eta^{o}\cap \Lambda^L$.

   For every Poisson point $x=(\bar x,a)\in \eta^{o}\cap \Lambda^L$, consider the sequence of the unique cells $R_{x_n,n}$ that $x$ occupies eventually, with $x_n=(\bar x_n,a_n)\in V^L_n$ for all $n$ sufficiently large. The sequence of spatial centers $\bar x_n$ converges to $\bar x$ for $n\to\infty$. The sequence $a_n$ converges to $a$, for purely atomic weight distributions $\pi$ we even have $a_n=a$ for all $n\in\N$ almost surely. 
   
    In the absolutely continuous case, coordinatewise monotonicity gives almost-everywhere continuity of $(r,a,b) \mapsto \varphi(r;a,b)$. The same observation, with one or both marks fixed, covers edges incident to $o$ and the purely atomic case. This gives the convergence of the connection probabilities: For $x$, $y$ distinct Poisson points, and the corresponding sequences $(x_n)_n$, $(y_n)_n$ of cell centers, almost surely
   $$
   \varphi(x_n,y_n)\xrightarrow[n\to\infty]{}\varphi(x,y).
   $$
    Conditional on the Poisson points $\eta^{o}\cap \Lambda^L$, the uniforms $U_{\{x,y\}}$ are almost surely not equal to $\varphi(x,y)$, so $U_{\{x,y\}}< \varphi(x,y)$ is equivalent to $U_{\{x,y\}}< \varphi(x_n,y_n)$ for sufficiently large $n\in\N$: the discrete and continuum edge states agree by the construction of the coupling. Since we almost surely have only finitely many vertices in $\Lambda^L$, and therefore finitely many edges, we get
    $$
    |\C_{L}(o)|=\vert \mathcal{C}(o)\vert,\quad \text{ for } n \text{ sufficiently large}.
    $$
    Moreover, the bounds 
    $$
    1\leq  \vert \mathcal{C}(o)\vert, |\C_{L}(o)| \leq 1+\eta(\Lambda^L_0)
    $$
    hold for all $n\in \N$. Dominated convergence and $\E_\lambda[\eta(\Lambda^L_0)]=\lambda\nu(\Lambda^L_0)<\infty$ then yield
    $$
    E^{L,n}_\lambda[\vert \mathcal{C}(o) \vert]
    \xrightarrow{n\to\infty} 
    \E_\lambda [|\C_L(o)|].
    $$
    Dominated convergence also yields 
    $$
    \theta^{L,n}_\lambda(k)=E^{L,n}_\lambda[\mathds 1_{\{\vert\mathcal C(o)\vert\geq k\}}] \xrightarrow[n\to\infty]{} 
    \E_\lambda[\mathds 1_{\{\vert\mathscr C_L(o)\vert\geq k\}}]
    =
    \theta^L_\lambda(k).
    $$
    Left to show is the convergence of the derivatives. Denote by $N_L:=\vert \eta\cap \Lambda^L_0\vert$ the number of Poisson points in $\Lambda^L_0$ (excluding $o$). Using our coupling, we rewrite
    \begin{align}
        \theta^{L,n}_\lambda(k)=\sum_{j=0}^\infty  \mathbb P_\lambda(\vert\mathcal C(o)\vert\geq k\mid N_L =j ) \mathbb P_\lambda (N_L=j ).
    \end{align}
    We see that the conditional probability  $\mathbb P_\lambda(\vert\mathcal C(o)\vert\geq k\mid N_L =j)$  does not depend on $\lambda$ and a simple Poisson calculation shows:
    $$
    \frac{\dd}{\dd \lambda} \mathbb P_\lambda(N_L=j) = \frac{\dd}{\dd \lambda} \Big( \e^{-\lambda\nu(\Lambda^L_0)} \frac{\lambda^j\nu(\Lambda^L_0)^j}{j!}\Big) = \Big(\frac{j}{\lambda}-\nu(\Lambda^L_0)\Big) \mathbb P_\lambda(N_L=j).
    $$
    Since the series $\sum_{j} \frac{\dd}{\dd \lambda} \mathbb P_\lambda(N_L=j)$ converges absolutely and uniformly in $\lambda$, we are allowed to take the derivative term by term, obtaining
    \begin{align*}
    \frac{\dd}{\dd \lambda} \theta^{L,n}_\lambda(k)
    =&
    \sum_{j=0}^\infty \E_\lambda\bigg[\Big(\frac{j}{\lambda}-\nu(\Lambda^L_0)\Big)\mathds{1}_{\{\vert\mathcal C(o)\vert\geq k\}}\mid N_L =j \bigg] \mathbb P_\lambda (N_L=j )
    \\=&
    \E_\lambda\bigg[\Big(\frac{N_L}{\lambda}-\nu(\Lambda^L_0)\Big)\mathds{1}_{\{\vert\mathcal C(o)\vert\geq k\}}\bigg]
    \end{align*}
    and dominated convergence yields
    $$ \frac{\dd}{\dd \lambda} \theta^{L,n}_\lambda(k)\xrightarrow{n\to\infty} \E_\lambda \bigg[\Big(\frac{N_L}{\lambda}-\nu(\Lambda^L_0)\Big)\mathds{1}_{\{\vert\mathscr C_L(o)\vert\geq k\}}\bigg],
    $$
    which, of course, can be rewritten as $\frac{\dd}{\dd \lambda} \theta^L_\lambda(k)$ by exactly the same argument.

\end{proof}

\begin{remark}
None of the conditions~\eqref{ass:LP}, \eqref{ass:BR}, \eqref{ass:RG}, and~\eqref{ass:MR} are needed for Lemma~\ref{lem:fla}. The lemma holds for every fixed $L\in\N$; in particular, we also apply it with $2L$ in place of $L$ for the reference quantities in the differential inequalities in Section~\ref{sect:Inf_vol_lim}.
\end{remark}

\subsection{Infinite-volume limit}

To finish, we show that the finite-volume continuum quantities converge to their infinite-volume counterparts.

\begin{lemma}\label{lem:converges_theta_inf_vol}
Let $\lambda>0$ and $k\in\N$. As $L\to\infty$, we have
\begin{align}
\theta_\lambda^L(k)\uparrow\theta_\lambda(k).
\label{convergence_cluster_probability_inf_vol}
\end{align}
\end{lemma}

\begin{proof}
We construct all finite-volume graphs by restricting the same
infinite-volume graph to $\Lambda^L$.
Every vertex of $\C(o)$ is connected to $o$ by a finite open
path, which is contained in $\Lambda^L$ for $L$ sufficiently
large. Consequently,
\[
\C_L(o)\uparrow\C(o),
\qquad
\{|\C_L(o)|\geq k\}\uparrow\{|\C(o)|\geq k\},
\]
and we conclude by taking the probability with respect to $\mathbb{P}_\lambda$.
\end{proof}

\begin{lemma}\label{lem:converges_expected_cluster_inf_vol}
Let $0<\lambda<\lambda_T$. As $L$ goes to infinity, we have
\begin{align}
\E_\lambda[|\C_L(o)|]
\xlongrightarrow[]{L\to\infty}
\E_\lambda[|\C(o)|].
\label{convergence_cluster_expectation_inf_vol}
\end{align}
\end{lemma}

\begin{proof}
Since $\lambda<\lambda_T$, we have
\[
\sum_{k=1}^\infty\theta_\lambda(k)
=\E_\lambda[|\C(o)|]<\infty.
\]
Moreover, $\theta_\lambda^L(k)\leq\theta_\lambda(k)$ for every $L$ and $k$. Thus, Lemma~\ref{lem:converges_theta_inf_vol} and dominated convergence give
\[
\E_\lambda[|\C_L(o)|]
=\sum_{k=1}^\infty\theta_\lambda^L(k)
\xlongrightarrow[]{L\to\infty}
\sum_{k=1}^\infty\theta_\lambda(k)
=\E_\lambda[|\C(o)|],
\]
which proves the claim.
\end{proof}
\begin{remark}
As we see, it suffices to observe that the finite-volume graphs are restrictions of the same infinite-volume graph to increasing windows exhausting the whole space, and their root clusters increase to the infinite-volume root cluster. In particular, neither Assumption~\ref{ass:general} nor conditions~\eqref{ass:LP}, \eqref{ass:BR}, \eqref{ass:RG}, and~\eqref{ass:MR} are needed for these limit transitions.
The restriction $\lambda<\lambda_T$ ensures that the limiting expected cluster size is finite; without it, convergence of the expectations still holds by monotone convergence, with $+\infty$ allowed as the limit.
\end{remark}

\section*{Declaration of AI use}

The authors made limited use of ChatGPT as an auxiliary tool during
the preparation and review of this manuscript. The authors take full
responsibility for the final version of the manuscript in its entirety.

\section*{Acknowledgements}
This work was supported by the Deutsche Forschungsgemeinschaft (DFG, project number 443880457) through the Priority Programme “Random Geometric Systems” (SPP 2265). The authors would like to thank Mikhail Chebunin and Benedikt Jahnel for interesting discussions. We are especially grateful to Markus Heydenreich for his constant support and insightful feedback, without which this work would not have been possible.


\begin{thebibliography}{99}
\bibliographystyle{unsrt}


\bibitem
{Aiz82}
M.~Aizenman.
\newblock Geometric analysis of $\phi^4$ fields and Ising models.
\newblock {\em Commun.\ Math.\ Phys.}, 86(1):1--48, 1982.

\bibitem
{AB87}
M.~Aizenman and D.~J.~Barsky.
\newblock Sharpness of the phase transition in percolation models.
\newblock {\em Commun. Math. Phys.}, 108(3):489--526, 1987.

\bibitem
{AB91}
D.~J.~Barsky and M.~Aizenman.
\newblock Percolation critical exponents under the triangle condition.
\newblock {\em Ann. Probab.}, 19(4):1520--1536, 1991.

\bibitem
{BGK93}
C.~E.~Bezuidenhout, G.~R.~Grimmett, and H.~Kesten.
\newblock Strict inequality for critical values of Potts models and random-cluster processes.
\newblock {\em Commun. Math. Phys.}, 158(1):1--16, 1993.

\bibitem
{CD24}
A.~Caicedo and M.~Dickson.
\newblock Critical exponents for marked random connection models.
\newblock {\em Electron. J. Probab.}, 29:1--57, 2024.

\bibitem
{CL24}
M.~Chebunin and G.~Last.
\newblock On the uniqueness of the infinite cluster and the cluster
density in the Poisson driven random connection model.
\newblock {\em Electron. J. Probab.}, 30, 1--42, 2025.

\bibitem
{CL25}
M.~Chebunin and G.~Last.
\newblock On strong sharp phase transition in the random connection model.
\newblock {\em Preprint arXiv:2512.00213 [math.PR]}, 2025.

\bibitem
{DH22}
M.~Dickson and M.~Heydenreich.
\newblock The triangle condition for the marked random connection model.
\newblock {\em Preprint arXiv:2210.07727 [math.PR]}, 2022.

\bibitem
{DRT19a}
H.~Duminil-Copin, A.~Raoufi, and V.~Tassion.
\newblock Exponential decay of connection probabilities for subcritical Voronoi percolation in $\mathbb{R}^d$.
\newblock {\em Probab. Theory Relat. Fields}, 173(1–2):479--490, 2019.

\bibitem
{DRT19b}
H.~Duminil-Copin, A.~Raoufi, and V.~Tassion.
\newblock Sharp phase transition for the random-cluster and Potts models via decision trees.
\newblock {\em Ann. Math.}, 189(1):75--99, 2019.

\bibitem
{DRT20}
H.~Duminil-Copin, A.~Raoufi, and V.~Tassion.
\newblock Subcritical phase of $d$-dimensional Poisson–Boolean percolation and its vacant set.
\newblock {\em Ann. H. Lebesgue}, 3:677--700, 2020.


\bibitem
{FM19}
A.~Faggionato and H.~A.~Mimun.
\newblock Connection probabilities in Poisson random graphs
with uniformly bounded edges.
\newblock {\em ALEA Lat. Am. J. Probab. Math. Stat.},
16(1):463--486, 2019.

\bibitem
{Gil61}
E.~N.~Gilbert.
\newblock Random plane networks.
\newblock {\em J. Soc. Indust. Appl. Math.}, 9(4):533--543, 1961.

\bibitem
{Gri99}
G.~Grimmett.
\newblock {\em Percolation}, volume 321 of {\em Grundlehren der Mathematischen Wissenschaften}.
\newblock Springer-Verlag, Berlin, 2nd edition, 1999.

\bibitem
{HHLM22}
M.~Heydenreich, R.~van~der~Hofstad, G.~Last, and K.~Matzke.
\newblock Lace expansion and mean-field behavior for the
random connection model.
\newblock {\em Ann. Inst. H. Poincar\'e Probab. Statist.}, 62(1):1--67, 2026.


\bibitem
{Hig25}
F.~Higgs.
\newblock Exponential decay for the random connection model
using asymptotic transitivity.
\newblock Preprint, arXiv:2509.02310 [math.PR], 2025.

\bibitem
{HJM22}
C.~Hirsch, B.~Jahnel, and S.~Muirhead.
\newblock Sharp phase transition for Cox percolation.
\newblock {\em Electron. Commun. Probab.}, 27:1–13, 2022.

\bibitem
{Hut20}
T.~Hutchcroft.
\newblock New critical exponent inequalities for percolation and the random cluster model.
\newblock {\em Probab. Math. Phys.}, 1(1):147--165, 2020.

\bibitem
{JKM23}
S.~Jansen, L.~Kolesnikov, K.~Matzke.
\newblock The direct-connectedness function in the random connection model. 
\newblock {\em Adv. Appl. Probab.}, 55(1):179--222, 2023.

\bibitem
{LPY21}
G.~Last, G.~Peccati, and D.~Yogeshwaran.
\newblock Phase transitions and noise sensitivity on the Poisson
space via stopping sets and decision trees.
\newblock {\em Random Structures Algorithms},
63(2):457--511, 2023.

\bibitem
{LasPen17}
G.~Last and M.~Penrose.
\newblock {\em Lectures on the Poisson Process}, volume~7 of {\em Institute of Mathematical Statistics Textbooks}.
\newblock Cambridge University Press, Cambridge, 2018.

\bibitem
{LasZie17}
G.~Last and S.~Ziesche.
\newblock On the Ornstein--Zernike equation for stationary cluster processes and the random connection model.
\newblock {\em Adv. Appl. Probab.}, 49(4):1260--1287, 2017.

\bibitem
{Mee95}
R.~Meester.
\newblock Equality of critical densities in continuum percolation.
\newblock {\em J. Appl. Probab.}, 32(1):90--104, 1995.

\bibitem
{MeeRoy96}
R.~Meester and R.~Roy.
\newblock {\em Continuum Percolation}, volume 119 of {\em Cambridge Tracts in Mathematics}.
\newblock Cambridge University Press, Cambridge, 1996.

\bibitem
{OSSS05}
R.~O'Donnell, M.~E.~Saks, O.~Schramm, and R.~A.~Servedio.
\newblock Every decision tree has an influential variable.
\newblock {\em Proc. 46th Annual IEEE Symposium on Foundations of Computer Science (FOCS'05)}, 31--39, 2005.

\bibitem
{Pab25}
D.~Pabst.
\newblock Percolation in the marked stationary random connection model for higher-dimensional simplicial complexes.
\newblock {\em Preprint arXiv:2506.15038 [math.PR]}, 2025.

\bibitem
{Pen91}
M.~Penrose.
\newblock On a continuum percolation model.
\newblock {\em Adv. Appl. Probab.}, 23(3):536--556, 1991.

\bibitem
{Pen03}
M.~Penrose.
\newblock {\em Random Geometric Graphs}, volume~5 of {\em Oxford Studies in Probability}.
\newblock Oxford University Press, Oxford, 2003.

\bibitem
{Van24}
H.~Vanneuville.
\newblock Exponential decay of the volume for Bernoulli percolation: a proof via stochastic comparison.
\newblock {\em Ann. H. Lebesgue}, 8:101--112, 2025.


\end{thebibliography}
\end{document}